\documentclass[a4paper, twoside, 11pt]{amsart}

\usepackage[utf8]{inputenc}
\usepackage[T1]{fontenc}
\usepackage[english]{babel}
\usepackage{lmodern}
\usepackage{amsmath,amssymb,amsthm}
\usepackage{geometry}
\usepackage{hyperref}
\usepackage{comment}

\usepackage{float}
\usepackage{tikz}
\usetikzlibrary{calc}

\hypersetup{
    colorlinks=true,
    linkcolor=blue,
    citecolor=blue,
    urlcolor=blue,
   pdftitle={CR Orbits for the Split Real Form of G2},
pdfauthor={Donato Mancini, Stefano Marini, Costantino Medori}
}

\newtheorem{theorem}{Theorem}[section]
\newtheorem{proposition}[theorem]{Proposition}
\newtheorem{lemma}[theorem]{Lemma}

\theoremstyle{definition}
\newtheorem{definition}[theorem]{Definition}

\theoremstyle{remark}
\newtheorem{remark}[theorem]{Remark}

\usepackage{tikz}
\usetikzlibrary{calc}

\providecommand{\GTwoCrossedDiagram}[4]{%
\begin{tikzpicture}[scale=0.9,baseline={(0,-0.05)}]
    \coordinate (A) at (0,0);
    \coordinate (B) at (2.1,0);

    \draw (A) -- (B);
    \draw ($(A)+(0,0.10)$) -- ($(B)+(0,0.10)$);
    \draw ($(A)+(0,-0.10)$) -- ($(B)+(0,-0.10)$);

    \node at (1.05,0) {$\langle$};

    \ifnum#1=1
        \filldraw[black] (A) circle (2.2pt);
    \else
        \draw[fill=white] (A) circle (2.2pt);
    \fi

    \ifnum#2=1
        \filldraw[black] (B) circle (2.2pt);
    \else
        \draw[fill=white] (B) circle (2.2pt);
    \fi

    \node[above=5pt] at (A) {$\alpha_1$};
    \node[above=5pt] at (B) {$\alpha_2$};

    \ifnum#3=1
        \node[below=6pt] at (A) {$\times$};
    \fi

    \ifnum#4=1
        \node[below=6pt] at (B) {$\times$};
    \fi
\end{tikzpicture}%
}
\title{CR Orbits on the complex flag manifolds of $G_2$}

\author{Donato Mancini}
\address{Donato Mancini, Department of Mathematical, Physical, and Computer Sciences, University of Parma, Parco Area delle Scienze 53/A (Campus), 43124 Parma, Italy }
\email{donato.mancini@unipr.it}

\author{Stefano Marini}
\address{Stefano Marini, Department of Mathematical, Physical, and Computer Sciences, University of Parma, Parco Area delle Scienze 53/A (Campus), 43124 Parma, Italy }
\email{stefano.marini@unipr.it}

\author{Costantino Medori}
\address{Costantino Medori, Department of Mathematical, Physical, and Computer Sciences, University of Parma, Parco Area delle Scienze 53/A (Campus), 43124 Parma, Italy }
\email{costantino.medori@unipr.it}

\makeatletter
\@namedef{subjclassname@2020}{\textup{2020} Mathematics Subject Classification}

\makeatother

\date{\today}
\begin{document}
\maketitle
\thispagestyle{plain}
\begin{abstract}
We give an explicit description of the orbits of the split real form
$G_{2(2)}$ on the complex flag manifolds of $G_2$ and study their
induced invariant CR structures. In particular we study their Levi forms and
show that all possible cases in the simple homogeneous setting
occur, namely  Levi nondegeneracy, $k$-nondegeneracy for $2\leq k\leq 3$,
and holomorphic degeneracy. Moreover, following the classical approach to orbits
on flag manifolds, we give equations for these homogeneous CR
manifolds.
\end{abstract}

\tableofcontents
\section*{Introduction}
The exceptional Lie group $G_2$, the smallest of the exceptional
simple Lie groups, occupies a distinguished place in mathematics
and mathematical physics. Its connections with octonion algebras,
exceptional holonomy, and special geometric structures make it
a natural meeting point of algebra and differential geometry, see \cite{Agricola2008} for an overview. The complex flag manifolds of $G_2$ provide a natural geometric
setting for studying its representations and the actions of its
real forms, see for instance \cite[Introduction, item~(4)]{HillMerkerNieNurowski}. In particular, the two maximal flag manifolds admit
concrete projective realizations as a quadric of isotropic lines
and the adjoint variety parametrizing special isotropic
two-dimensional subspaces, see
\cite[Section~6.1, Propositions~6.2--6.3 and Remark~6.4]{LandsbergManivel2003}.
Here we study the orbits of the split real form $G_{2(2)}$
on these spaces and on the complete flag manifold, together
with the Cauchy–Riemann structures induced by their
generic embeddings into the respective flag manifolds (see \cite{MariniMedoriNacinovich2020,MariniMedoriNacinovich2023Closed}
for an account of the CR structures induced on real orbits by their
embeddings into complex flag manifolds). 
\par
The Levi form is a fundamental invariant of a CR structure,
and its degeneracy plays a central role in the study of
local CR automorphisms and their determination by finite jets
\cite{KolarMeylan2019}.
When the Levi form is
degenerate, determining whether the structure satisfies a weaker notion of nondegeneracy, namely $k$-nondegenerate (see 
\cite[Definition~11.1.1]{BaouendiEbenfeltRothschild1999}),
and computing its precise order provides a natural
higher-order refinement of this invariant.
These questions have acquired particular importance in the contemporary
study of Levi-degenerate CR manifolds, especially in connection with
the classification of homogeneous models and the local equivalence
problem
\cite{FelsKaup2008,Gregorovic2021,IsaevZaitsev2013,KaupZaitsev2006,KruglikovSanti2025,MedoriSpiro2014,SykesZelenko2023}, e.g. the effectiveness of these
higher-order invariants is provided by their use in constructing
normal forms and explicit defining equations for $2$-nondegenerate
model CR hypersurfaces
\cite{GregorovicKolarSykes2025}. For all these reasons, we think that   the real orbits for the split real form
$G_{2(2)}$ are an interesting class of examples. In fact, it was proved in
\cite[Theorem~2.11]{MariniMedoriNacinovich2020}
that a $k$-nondegenerate real orbit in a complex flag manifold
$G/Q$ with $G$ simple satisfies $k\leq 3$.
In this paper, we show that all possible cases occur among
the $G_{2(2)}$-orbits, namely  Levi
nondegeneracy, $k$-nondegeneracy for $2\leq k\leq 3$,
and holomorphic degeneracy.

\par
The paper is structured as follows. In Section~1,  we introduce the necessary
preliminaries by  recalling the structure and root system of the complex
Lie algebra $\mathfrak g_{2}$ and its split real form
$\mathfrak g_{2(2)}$, describing the $G_2$  flag manifolds and
their projective realizations, and reviewing the CR-algebras tools used
to study the induced $CR$ structures on the split real-form orbits.
In Section~2, we determine the
$G_{2(2)}$-orbits on the complete flag manifold $G_{2}/B$. After
recalling the general orbit-counting procedure in \cite{Aomoto1966} ( see also \cite{RichardsonSpringer1990}), we carry out the
explicit calculations for $G_{2}$ (Lemma~\ref{thm:orbit-count-complete-flag}), derive defining equations   for the resulting real orbits (see Table~\ref{tab:example}), and study their
projections on the two maximal flag manifolds (Lemma~\ref{thm:orbit-count-maximal-flags}). We also derive their dimensions using SageMath.
In Section~3 we give intrinsic equations for the five $G_{2(2)}$-orbits on each maximal flag manifold $F_{\alpha_i}=G_2/Q_{\alpha_i}$, $i=1,2$, see
Tables~\ref{tab:Xalpha1-orbit-summary} and~\ref{tab:Xalpha2-orbit-summary}. Local equations for the non-open, non-totally-real orbits are given in Theorems~\ref{thm:G2-alpha1-explicit-local-models} and~\ref{thm:local-models-X-alpha2}. Finally, in Section~4, we determine all possible CR types and orders of $k$-nondegeneracy that can be associated with all the orbits, see
Tables~\ref{tab:complete-flag-coarse-types-count}, \ref{tab:Xalpha-CR-summary}, and~\ref{tab:Xbeta-CR-summary}. These tables record the CR type, the order of nondegeneracy and fundamentality of each orbit; the CR structure of an orbit is the one induced by its CR algebra $(\mathfrak g_0,\mathfrak q_x)$, described in the computations of Section~4. 
\section{Preliminaries}
\subsection{The Lie algebra $\mathfrak{g}_{2}$}\label{sec: all on g2}
In this section, we will recall the concrete matrix realization of the Lie algebra of type $\mathfrak g_2$ and of its split real form as shown in \cite{DRAPERFONTANALS201823}. Furthermore, we will  compute the generators of the compact real form of $\mathfrak{g}_{2}$ and derive from it an expression for the Cartan involution on the split real form. 
Let $\mathbb F\in\{\mathbb R,\mathbb C\}$.
For \(x=(x_1,x_2,x_3)^t\in \mathbb F^3\), let $l_x$
denote the matrix
\[
l_x=
\begin{pmatrix}
0 & -x_3 & x_2\\
x_3 & 0 & -x_1\\
-x_2 & x_1 & 0
\end{pmatrix}.
\]
We then consider the vector space
\[
\hat{\mathfrak g}_{\mathbb F}=
\left\{
M(a,x,y):=
\begin{pmatrix}
0 & -2y^t & -2x^t\\
x & a & l_y\\
y & l_x & -a^t
\end{pmatrix}
\;:\;
a\in \mathfrak{sl}_3(\mathbb F),\ x,y\in \mathbb F^3
\right\}
\subset \mathfrak{gl}_7(\mathbb F).
\]
A direct verification shows that \(\hat{\mathfrak g}_{\mathbb F}\) is a $14$-dimensional simple Lie subalgebra
of \(\mathfrak{gl}_7(\mathbb F)\); hence, for \(\mathbb F=\mathbb C\), it is  a realization of the complex simple Lie algebra \(\mathfrak g_2\), while $\hat{\mathfrak{g}}_{\mathbb R}$ is its split real form. In the sequel, we will denote $\hat{\mathfrak{g}}_{\mathbb C}$ simply by $\mathfrak{g}_{2}$ and the symbol $G_{2}$ will indicate the corresponding simply connected complex Lie group, while its real form will be denoted by $\mathfrak{g}_{2(2)}$ with corresponding connected Lie group $G_{2(2)}\subset G_{2}.$ 
A standard Cartan subalgebra of $\hat{\mathfrak{g}}_{\mathbb{F}}$ is given by 
\begin{equation}\label{eq: Cartan subalgebra}
   \mathfrak{h}= \left\{
H(s_1,s_2,s_3)=[0,s_1,s_2,s_3,-s_1,-s_2,-s_3]
\mid s_1,s_2,s_3\in\mathbb C,\; s_1+s_2+s_3=0
\right\},
\end{equation}
where square brackets denote the diagonal matrix with
the indicated entries.
Define two linear forms \(\alpha_1,\alpha_2\in \mathfrak h^*\) by $\alpha_1\bigl(H(s_1,s_2,s_3)\bigr)=s_2,
$ and $
\alpha_2\bigl(H(s_1,s_2,s_3)\bigr)=s_1-s_2.$

These are the short and long simple roots, respectively. The corresponding set of
positive roots is $\Delta_+
=
\{\alpha_1,\alpha_2,\alpha_1+\alpha_2,2\alpha_1+\alpha_2,3\alpha_1+\alpha_2,3\alpha_1+2\alpha_2\},$
and the full root system is $\Delta
=
\pm\{\alpha_1,\alpha_2,\alpha_1+\alpha_2,2\alpha_1+\alpha_2,3\alpha_1+\alpha_2,3\alpha_1+2\alpha_2\}.$ The corresponding generalized eigenvectors are $X_{\epsilon_i-\epsilon_j}=M(E_{ij},0,0)$, $X_{\epsilon_i}=M(0,e_i,0)$, and $X_{-\epsilon_i}=M(0,0,-e_i)$.
Following the approach of \cite{KnappLieGroups}, let us now find the generators of the compact real form of $\mathfrak{g}_{2}$.
Consider the following Lie subalgebra of $\mathfrak{g}_{2}$
\[\mathfrak{g}_{u}=\sum_{\alpha\in\Delta}\mathbb{R}(iH_{\alpha})+\sum_{\alpha\in\Delta}\mathbb{R}(U_{\alpha}-U_{-\alpha})+i\sum_{\alpha\in\Delta}\mathbb{R}(U_{\alpha}+U_{-\alpha}),\]
where for each $\alpha\in\Delta$, $U_{\alpha}\in\mathfrak{g}_{\alpha}$, $H_{\alpha}\in\mathfrak{h}$ such that the following properties hold:
\begin{itemize}
    \item [$(1)$]$[U_{\alpha},U_{-\alpha}]=H_{\alpha};$
    \item[$(2)$] $[U_{\alpha},U_{\beta}]=N_{\alpha,\beta}U_{\alpha+\beta}$\, \text{if $\alpha+\beta\in\Delta$};
    \item[$(3)$] $[U_{\alpha},U_{\beta}]=0\, \text{if $\alpha+\beta\notin\Delta$ and $\beta\neq-\alpha$}$ 
    \item[$(4)$] $N_{-\alpha,-\beta}=-N_{\alpha,\beta};$ 
\end{itemize}
Let $k$ denote the Killing form of the Lie algebra $\mathfrak{g}_{2(2)}$. It can be  proven that $k(X,Y)=4\operatorname{Tr}(XY).$
Let us now compute the vectors $X_{\alpha}, H_{\alpha}$.
We will examine the following two cases:
\begin{itemize}
    \item [$\boxed{\alpha=\epsilon_{i}}$]
    Consider $U_{\alpha}=\frac{1}{\sqrt{24}}M(0,e_{i},0)$ and $
        U_{-\alpha}=\frac{1}{\sqrt{24}}M(0,0,-e_{i}).$ Then, it is easy to check that $k(U_{\alpha},U_{-\alpha})=1$. As a consequence, we obtain that $[U_{\alpha},U_{-\alpha}]=H_{\alpha}$,
    where 
    \begin{alignat*}{3}
    H_{\epsilon_{1}}&=\frac{1}{24}[0,2,-1,-1,-2,1,1],\qquad&&
    H_{\epsilon_{2}}=\frac{1}{24}[0,-1,2,-1,1,-2,-1],\\
    H_{\epsilon_{3}}&=\frac{1}{24}[0,-1,-1,2,1,1,-2].
    \end{alignat*}
    \item[$\boxed{\alpha=\epsilon_{i}-\epsilon_{j}}$]
    Consider $U_{\alpha}=\frac{1}{\sqrt{8}}M(E_{ij},0,0)$ and $
    U_{-\alpha}=\frac{1}{\sqrt{8}}M(E_{ji},0,0).$
 Just as in the previous case, $[U_{\alpha},U_{-\alpha}]=H_{\alpha}$,
    where
    \begin{align*}
    &H_{\epsilon_{1}-\epsilon_{2}}=\scriptsize{\frac{1}{{8}}\begin{pmatrix}
        0&0&0\\
        0&\begin{pmatrix}
            1&0&0\\
            0&-1&0\\
            0&0&0
        \end{pmatrix}&0\\
        0&0&\begin{pmatrix}
            -1&0&0\\
            0&1&0\\
            0&0&0
        \end{pmatrix}
    \end{pmatrix};}\\ 
    &H_{\epsilon_{2}-\epsilon_{3}}=\scriptsize{\frac{1}{{8}}\begin{pmatrix}
        0&0&0\\
        0&\begin{pmatrix}
            0&0&0\\
            0&1&0\\
            0&0&-1
        \end{pmatrix}&0\\
        0&0&\begin{pmatrix}
            0&0&0\\
            0&-1&0\\
            0&0&1
        \end{pmatrix}
    \end{pmatrix};}\\
    &H_{\epsilon_{1}-\epsilon_{3}}=\scriptsize{\frac{1}{{8}}\begin{pmatrix}
        0&0&0\\
        0&\begin{pmatrix}
            1&0&0\\
            0&0&0\\
            0&0&-1
        \end{pmatrix}&0\\
        0&0&\begin{pmatrix}
            -1&0&0\\
            0&0&0\\
            0&0&1
        \end{pmatrix}
    \end{pmatrix}.}
    \end{align*}
    
\end{itemize}
It is an easy but lengthy computation to verify that the chosen vectors $U_{\alpha},U_{-\alpha},H_{\alpha}$ satisfy all the properties $(1)-(4)$.
Let $\mathfrak g_u$ be a compact real form of $\mathfrak g_2$
compatible with the above matrix realization. The corresponding
conjugation will be denoted by $\tilde\theta$.
Using the concrete expression of $\mathfrak{g}_{u}$, one may easily show that
\[\tilde{\tau}\begin{pmatrix}
    0&-2z_{2}^{t}&-2z_{1}^{t}\\
    z_{1}&Z&l_{z_{2}}\\
    z_{2}&l_{z_{1}}&-Z^{t}
\end{pmatrix}=\begin{pmatrix}
    0&-2\bar{z}_{1}^{t}&-2\bar{z}_{2}^{t}\\
    \bar{z}_{2}&-\bar{Z}^{t}&l_{\bar{z}_{1}}\\
    \bar{z}_{1}&l_{\bar{z}_{2}}&\bar{Z}
\end{pmatrix}.\]
Besides, we remark that the split real form $\mathfrak{g}_{2(2)}$ is $\tilde{\theta}-$invariant. We thus obtain a Cartan involution   $\theta:\ \mathfrak{g}_{2(2)}\ \rightarrow\ \mathfrak{g}_{2(2)}$ acting as follows
\[\theta\begin{pmatrix}
    0&-2z_{2}^{t}&-2z_{1}^{t}\\
    z_{1}&Z&l_{z_{2}}\\
    z_{2}&l_{z_{1}}&-Z^{t}
\end{pmatrix}=\begin{pmatrix}
    0&-2z_{1}^{t}&-2z_{2}^{t}\\
    z_{2}&-Z^{t}&l_{z_{1}}\\
    z_{1}&l_{z_{2}}&Z
\end{pmatrix}.\]
Therefore, we may consider the following Cartan decomposition $\mathfrak{g}_{2(2)}=\mathfrak{k}_{0}\oplus\mathfrak{p}_{0},$
where 
\begin{align*}
    &\mathfrak{k}_{0}=\left\{\begin{pmatrix}
    0&-2x^{t}&-2x^{t}\\
    x&a&l_{x}\\
    x&l_{x}&a
\end{pmatrix}\ |\ x\in\mathbb{R}^{3},\ a\in\mathfrak{so}(3,\mathbb{R})\right\};\\
    &\mathfrak{p}_{0}=\left\{\begin{pmatrix}
    0&2x^{t}&-2x^{t}\\
    x&a&-l_{x}\\
    -x&l_{x}&-a
\end{pmatrix}\ |\ x\in\mathbb{R}^{3},\ a\in\operatorname{sym}_{3}(\mathbb{R})\cap\mathfrak{sl}_3(\mathbb{R})\right\}.\\
\end{align*}
\subsection{Flag manifolds and $G_{0}$-orbits}\label{sec:Flag manifolds and $G_{0}$-orbits }
Fix a Cartan subalgebra $\mathfrak h$ of the complex semisimple Lie
algebra $\mathfrak g$, and let $\Delta=\Delta(\mathfrak g,\mathfrak h)$ be
the corresponding root system. A Borel subalgebra $\mathfrak{b}$ of $\mathfrak{g}$ is a maximal solvable subalgebra of $\mathfrak{g}.$ It is well known that all Borel subalgebras are conjugate to $\mathfrak{b}(\Pi)=\mathfrak{h}\oplus\bigoplus_{\alpha\in\Delta_{+}}\mathfrak{g}_{\alpha}$ for a suitable basis $\Pi\subset \Delta$.
A parabolic subalgebra $\mathfrak{q}$ is a subalgebra of $\mathfrak{g}$ containing a Borel subalgebra. It is a classic result that the conjugacy classes of  parabolic subalgebras containing a fixed Borel subalgebra $\mathfrak{b}(\Pi)$ are in one-to-one correspondence with subsets
$\Pi_0\subseteq\Pi$. For each $\Pi_0\subset \Pi$, we  will denote by $\mathfrak q_{\Pi_0}$ the
corresponding standard parabolic subalgebra. We use the
standard root-theoretic conventions for these objects, see 
\cite[Ch.~VIII, \S~3.4]{Bourbaki7-9}.\\
Given a complex semisimple Lie group $G$ with Lie algebra $\mathfrak{g}$, we say that a Lie subgroup $B\subset G$ is a Borel subgroup if its Lie algebra is a Borel subalgebra. Furthermore, a Lie subgroup $P\subset G$ is a parabolic subgroup if its Lie algebra is a parabolic subalgebra. Given $\Pi_0\subset\Pi$ with standard parabolic subalgebra $\mathfrak{q}_{\Pi_0}$, we denote by $Q_{\Pi_0}$ the corresponding connected parabolic subgroup of $G$. \par

 Our objective in this paper will be to study the CR structures on the orbits of the flag manifolds of $G_{2}$ under the action of its split real form $G_{2(2)}$. In order to do this, we will need to recall a few basic facts on real orbits on flag manifolds.
Our reference throughout will be \cite{Zierau2000}.
Let $G$ be a complex connected semisimple Lie group with Lie algebra $\mathfrak{g}$ and let $G_{0}$ be a real form of $G$ with Lie algebra $\mathfrak{g}_{0}$. Let $\sigma$ denote the conjugation of $G$ with respect to $G_0$.\par 
Let $Q$ be a parabolic subgroup of $G$ with Lie algebra $\mathfrak{q}$. It is well-known that the flag manifold $F=G/Q$ can be identified with the $G$-conjugates of $\mathfrak{q}$ via the bijection sending the coset space $gQ$ to the parabolic subalgebra $\operatorname{Ad}(g)(\mathfrak{q}).$
Furthermore, $G$ acts transitively on $X$ via left multiplication. We are concerned with the number and structure of the orbits of the real form $G_{0}$ on the space $X$. For this purpose we  recall  {\cite[Theorems~2.6 and~3.3]{Wolf1969}}:

\begin{theorem}
Let $G$ be a connected complex semisimple Lie group, let $G_0$
be a connected real form of $G$, and let $Q\subset G$ be a
parabolic subgroup. Set $F=G/Q$. Then
\begin{enumerate}
\item $F$ is the disjoint union of finitely many $G_0$-orbits.
\item There is a unique closed $G_0$-orbit in $F$, and its real
dimension is strictly smaller than that of every other $G_0$-orbit.
\end{enumerate}
\end{theorem}
We now specialize this general framework to the complex flag
manifolds of $G_2$. Recall that $G_2$ has only two real forms, namely its compact real form
and the split real form $G_{2(2)}$,  since the compact real form acts
transitively on every complex flag manifold of $G_2$, we henceforth
will specialize to the of $G_{2(2)}$ real orbits.
Consider the Borel subalgebra $\mathfrak{b}$ associated with the base $\Pi$ and let $B$ indicate the connected Lie subgroup of $G_{2}$ with Lie algebra $\mathfrak{b}$. Let $\mathfrak{q}_{\alpha_{1}}$ and $\mathfrak{q}_{\alpha_{2}}$ be the parabolic subalgebras of $\mathfrak{g}_{2}$ associated with the subsets of $\Pi$ given by $\{\alpha_{1}\}$ and $\{\alpha_{2}\}$, respectively. More explicitly, we have that
\begin{align} 
\label{eq: parabolics} 
    &\mathfrak{q}_{\alpha_{1}}=\left\{\begin{pmatrix}
    0&0&0&-2y&-2x_{1}&-2x_{2}&0\\
    x_{1}&a_{11}&a_{12}&a_{13}&0&-y&0\\
    x_{2}&a_{21}&a_{22}&a_{23}&y&0&0\\
    0&0&0&a_{33}&0&0&0\\
    0&0&0&x_{2}&-a_{11}&-a_{21}&0\\
    0&0&0&-x_{1}&-a_{12}&-a_{22}&0\\
    y&-x_{2}&x_{1}&0&-a_{13}&-a_{23}&-a_{33}
\end{pmatrix}\right\};\\
    &\mathfrak{q}_{\alpha_{2}}=\left\{\begin{pmatrix}
    0&0&-2z&-2y&-2x_{1}&-2x_{2}&0\\
    x_{1}&a_{11}&a_{12}&a_{13}&0&-y&z\\
    x_{2}&0&a_{22}&a_{23}&y&0&0\\
    0&0&0&-a_{11}-a_{22}&-z&0&0\\
    0&0&0&x_{2}&-a_{11}&0&0\\
    z&0&0&-x_{1}&-a_{12}&-a_{22}&0\\
    y&-x_{2}&x_{1}&0&-a_{13}&-a_{23}&a_{11}+a_{22}
\end{pmatrix}\right\}.
\end{align}
Let $Q_{\alpha_{1}}$ and $Q_{\alpha_{2}}$ be the corresponding connected parabolic Lie subgroups of $G_{2}.$ In what follows, we will need to view the flag spaces of $G_{2}$ in a more geometric way. More precisely, we will realize $G_{2}/Q_{\alpha_{i}}$ as the unique closed orbit of $G_{2}$ on $\mathbb{P}(V_{\omega_{i}})$, where $V_{\omega_{i}}$ denotes the irreducible representation of $\mathfrak{g}_{2}$ with highest weight $\omega_{i}$. This will also allow us to embed $G_{2}/B$ into the product of projective spaces. As shown in \cite[Section~23.3]{Fulton Harris}, this can be done  for any quotient $G/Q$. Let us briefly recall the construction of this embedding in general terms and then apply it to the flags of $G_{2}$.  
Let $G$ be a connected complex simple Lie group with Lie algebra $\mathfrak{g}$ and let $V_{\lambda}$ be an irreducible representation of $\mathfrak{g}$ with highest weight $\lambda$. Let $v$ be a highest weight vector. Consider $P_{\lambda}$ to be the parabolic subgroup of $G$ corresponding to the set of all simple roots which are perpendicular to $\lambda$ and suppose the action of $\mathfrak{g}$ is integrable to a Lie group action of $G$ on $V_{\lambda}$. Then, $V_{\lambda}$ coincides with the irreducible representation of $G$ of weight $\lambda$ and the map
\[i:\ G/P_{\lambda}\ \rightarrow\ \mathbb{P}(V_{\lambda})\]
\[gP\ \rightarrow\ [g\cdot v]\]
embeds $G/P_{\lambda}$ as the unique closed $G$-orbit of the action of $G$ on $\mathbb{P}(V_{\lambda}).$ 
 In the case at hand, the fundamental weights with respect to the chosen base are $\omega_{1}=-\epsilon_3$ with highest weight vector $v=e_{7}$ and $\omega_{2}=\epsilon_{1}-\epsilon_{3}$ with highest  weight vector $M=M(E_{13},0,0).$ 
The fundamental representation $V(\omega_{1})$ is the standard action of $\mathfrak{g}_{2}$ on $\mathbb{C}^{7}$ by matrix multiplication, while $V(\omega_{2})$ is the adjoint action of $\mathfrak{g}_{2}$ on itself. According to the construction described above, we obtain the following embeddings of the flag manifolds $G_{2}/Q_{\alpha_{i}}$ and $G_{2}/B$:
\begin{alignat*}{2}
    i_{1}:\ &G_{2}/Q_{\alpha_{1}}\ \hookrightarrow\ \mathbb{P}(V_{\omega_{1}}),\qquad i_{2}:\ &&G_{2}/Q_{\alpha_{2}}\ \hookrightarrow\ \mathbb{P}(V_{\omega_{2}})\\
    &g Q_{\alpha_{1}}\ \mapsto\ [g\cdot v]\qquad &&g Q_{\alpha_{2}}\ \mapsto\ [g\cdot M].\\
    \end{alignat*}
    These embeddings also yield
\begin{align*}
i_{3}:&\ G_{2}/B\ \hookrightarrow \mathbb{P}(V_{\omega_1})\times \mathbb{P}(V_{\omega_2})\\
    &gB\ \hookrightarrow ([g\cdot v],[g\cdot M]).    
\end{align*}
We remark that the map sending $gB$ to the pair $([g\cdot v],[g\cdot M])$ is injective. Indeed, suppose $([g\cdot v],[g\cdot M])=([g'\cdot v],[g'\cdot M]).$ This means that $gQ_{\alpha_{1}}=g'Q_{\alpha_{1}}$ and $gQ_{\alpha_{2}}=g'Q_{\alpha_{2}}$. Therefore, we have that $g^{-1}g'\in Q_{\alpha_1}\cap Q_{\alpha_2}=B$, so $gB=g'B.$

\subsection{Cauchy-Riemann real orbits}\label{CR_orbits section}
We briefly recall some notions of Cauchy–Riemann geometry that will be used in the study of real orbits in complex flag manifolds. A CR manifold is a smooth real manifold $M$ endowed with a complex subbundle
$$H^{0,1}M\subset T^{\mathbb C}M:=\mathbb C \otimes_{\mathbb R} TM$$ such that $H^{0,1}M\cap \overline{H^{0,1}M}=\{0\},$ and $H^{0,1}M$ is formally integrable, i.e.
$$[\Gamma(H^{0,1}M),\Gamma(H^{0,1}M)]\subset \Gamma(H^{0,1}M).$$
Its conjugate bundle is defined as $H^{1,0}M:=\overline{H^{0,1}M},$ and the underlying real CR distribution is
$$HM:=\Re\bigl(H^{1,0}M\oplus H^{0,1}M\bigr)\subset TM.$$ The \emph{CR dimension} and \emph{CR codimension} at $p\in M$ are defined by $$\dim_{\mathrm{CR}}M:=\dim_{\mathbb C}H_{p}^{0,1}M,
\qquad
\operatorname{codim}_{\mathrm{CR}}M:=\dim_{\mathbb R}M-2\,\dim_{\mathrm{CR}}M.$$ 
The pair $(n,k)$, where $n$ is the CR dimension and $k$ is the CR codimension, is called the \emph{CR type} of $M$. Thus, a CR manifold of type $(n,k)$ has real dimension $2n+k$, and when $k=1$ the CR manifold $M$ is of hypersurface type. A $CR$ manifold is said to be \emph{totally real} if its $CR$ dimension is zero, and \emph{totally complex} if its $CR$ codimension is zero. 
\par
The \emph{Levi form}  is an Hermitian form defined on $H^{1,0}M$ that measures the  obstruction to the involutivity of the real distribution
$HM$. If $p\in M$, $Z_p,W_p\in H_{p}^{1,0}M$, and $ Z, W$ are
local sections of $H^{1,0}M$ extending them, then the \emph{Levi form} is defined by
\begin{equation*}
    \mathcal L_{p}(Z_p,\overline W_p)
=
\frac{1}{2i}\,[ Z,\overline{ W}]_p
\mod \bigl(H_{p}^{1,0}M\oplus H_{p}^{0,1}M\bigr),
\end{equation*}
with values in $(T_{p}M/H_{p}M)^{\mathbb C}
\simeq
T_{p}^{\mathbb C}M/\bigl(H_{p}^{1,0}M\oplus H_{p}^{0,1}M\bigr).$ 
The Levi kernel at $p$ is defined by
$$
\mathcal K_p^{(1)}
:=
\bigl\{
Z_p\in H_p^{1,0}M:
\mathcal L_p(Z_p,\overline W_p)=0
\text{ for all }\overline W_p\in H_p^{0,1}M
\bigr\}.
$$

A notion of higher order Levi kernels, defined inductively from
$\mathcal K_p^{(1)}$, was introduced in \cite{Freeman1974}.
At each step of the construction, we assume that the preceding
kernels have locally constant rank, so that they form smooth
complex subbundles of $H^{1,0}M$. This condition is automatically
satisfied for homogeneous CR manifolds, since the kernels are
invariant under the transitive action by CR automorphisms.

Set $\mathcal L_p^{(1)}:=\mathcal L_p$.
For $r\geq 2$, suppose that $\mathcal K^{(r-1)}$ has been
defined as a smooth complex subbundle near $p$.
The $r$-th order Levi form at $p$ is the map
$$
\mathcal L_p^{(r)}:
\mathcal K_p^{(r-1)}\times H_p^{0,1}M
\longrightarrow
T_p^{\mathbb C}M/
\bigl(\mathcal K_p^{(r-1)}\oplus H_p^{0,1}M\bigr)
$$
defined by
$$
\mathcal L_p^{(r)}(Z_p,\overline W_p)
:=
[Z,\overline W]_p
\mod\bigl(\mathcal K_p^{(r-1)}\oplus H_p^{0,1}M\bigr),
$$
where $Z$ and $\overline W$ are local sections of
$\mathcal K^{(r-1)}$ and $H^{0,1}M$, respectively, extending
$Z_p$ and $\overline W_p$. The class on the right-hand side
is independent of the chosen extensions.
The $r$-th order Levi kernel at $p$ is then defined by
$$
\mathcal K_p^{(r)}
:=
\bigl\{
Z_p\in\mathcal K_p^{(r-1)}:
\mathcal L_p^{(r)}(Z_p,\overline W_p)=0
\text{ for all }\overline W_p\in H_p^{0,1}M
\bigr\}.
$$
Thus one obtains a decreasing sequence of complex subspaces
\begin{equation}
\label{Freeman_seq}
H_p^{1,0}M\supseteq\mathcal K_p^{(1)}
\supseteq\mathcal K_p^{(2)}
\supseteq\mathcal K_p^{(3)}
\supseteq\cdots .
\end{equation}

\begin{definition}\label{dfn_k_nndg}
A CR manifold $M$ is said to be $k$-nondegenerate at $p$
if $k$ is the smallest positive integer such that
$\mathcal K_p^{(k)}=\{0\}$.
In particular, $1$-nondegeneracy is the usual Levi
nondegeneracy, namely $\mathcal K_p^{(1)}=\{0\}$.
\end{definition}

For the real-analytic homogeneous CR manifolds considered here,
failure of finite nondegeneracy is equivalent to holomorphic
degeneracy. More precisely, if no positive integer $k$ satisfies
$\mathcal K_p^{(k)}=\{0\}$, then $M$ is locally CR equivalent
near $p$ to a product $M'\times\mathbb C^r$, where $M'$ is a
real-analytic CR manifold and $r>0$. Conversely, such a local
product decomposition implies that $M$ is not finitely
nondegenerate at $p$
(see \cite[Chapter~11]{BaouendiEbenfeltRothschild1999}).
\par
We now move in the context of homogeneous $CR$ manifolds, i.e. $G_0$-homogeneous manifolds that admit a $G_0$-invariant $CR$ structure.  By transitivity and $G_0$-invariance of the CR structure,
the pointwise properties introduced above hold at every
point whenever they hold at one point. For a homogeneous CR manifolds, it is  convenient then to use the 
formalism of CR algebras introduced in \cite{MedoriNacinovich2005}.

\begin{definition}
\label{dfn_CR_alg}
  A CR algebra is a pair $(\mathfrak g_0,\mathfrak q),$ where $\mathfrak g_0$ is a real Lie algebra  and $\mathfrak q$ a complex subalgebra of the complexification  $\mathfrak g{:=}\mathbb C\otimes_{\mathbb R}\mathfrak g_0$,  such that $\dim_{\mathbb R}\bigl(\mathfrak g_0/(\mathfrak g_0\cap\mathfrak q)\bigr){<}\infty.$
 
\end{definition}
Let $G_0$ be a real Lie group with Lie algebra $\mathfrak g_0$,
and let $I_0\subset G_0$ be the connected analytic subgroup
with Lie algebra $\mathfrak i_0:=\mathfrak g_0\cap\mathfrak q$.
Assume that $I_0$ is closed in $G_0$. Then the quotient
$M=G_0/I_0$ is a smooth manifold, and
$(\mathfrak g_0,\mathfrak q)$ determines a
$G_0$-invariant CR structure on $M$. If $\pi:G_{0}\to M$ denotes  the standard quotient projection, and $o=eI_{0}\in M$ the base point then we have
 $\mathfrak q=(d\pi_{e}^{\mathbb C})^{-1}(H_{o}^{0,1}M).$
At the base point one has the canonical identifications $$H_{o}^{0,1}M\simeq \mathfrak q/(\mathfrak q\cap\overline{\mathfrak q}), \qquad
H_{o}^{1,0}M\simeq \overline{\mathfrak q}/(\mathfrak q\cap\overline{\mathfrak q}),$$
and $$T_{o}^{\mathbb C}M/
\bigl(H_{o}^{1,0}M\oplus H_{o}^{0,1}M\bigr)
\simeq
\mathfrak g/(\mathfrak q+\overline{\mathfrak q}),$$ where $\mathfrak q\cap\overline{\mathfrak q}=\mathbb C\otimes\mathfrak i_0$. Accordingly the $CR$ dimension and codimension can be computed as 
\begin{equation}\label{eq:CRdim-CRalg}
\dim_{\mathrm{CR}}M
=
\dim_{\mathbb C}\mathfrak q
-
\dim_{\mathbb C}(\mathfrak q\cap\overline{\mathfrak q}),
\qquad
\operatorname{codim}_{\mathrm{CR}}M
=
\dim_{\mathbb C}\mathfrak g
-
\dim_{\mathbb C}(\mathfrak q+\overline{\mathfrak q}).
\end{equation}
A CR algebra $(\mathfrak g_0,\mathfrak q)$ is called \emph{fundamental} if the complex Lie subalgebra generated by $\mathfrak q+\overline{\mathfrak q}$ is the whole of $\mathfrak g$, that is, $\langle\mathfrak q+\overline{\mathfrak q}\rangle_{\mathrm{Lie}}=\mathfrak g$,  equivalently, on the associated homogeneous CR manifold, the pullback of the CR distribution is bracket-generating, namely its iterated Lie brackets generate the whole tangent bundle.\par
The Levi form at the base point is given purely algebraically, if
\(Z,W\in\mathfrak q\), then
\begin{equation}\label{eq:Levi-CRalg}
\mathcal L_{o}(Z,\overline W)
=
[Z,\overline W]
\mod (\mathfrak q+\overline{\mathfrak q}).
\end{equation}
Thus the first Levi kernel is represented by
$\mathfrak q^{(1)}
:=
\{Z\in\mathfrak q:[Z,\overline{\mathfrak q}]
\subset \mathfrak q+\overline{\mathfrak q}\}.$ More generally, one defines recursively the Freeman sequence \eqref{Freeman_seq}
\begin{equation}\label{eq:Freeman-filtration-CRalg}
\mathfrak q=\mathfrak q^{(0)}\supset \mathfrak q^{(1)}\supset \mathfrak q^{(2)}\supset\cdots
\end{equation}
by $\mathfrak q^{(r+1)}
:=
\{Z\in \mathfrak q^{(r)}:[Z,\overline{\mathfrak q}]
\subset \mathfrak q^{(r)}+\overline{\mathfrak q}\},
$ for $r\geq 0.$ Each \(\mathfrak q^{(r)}\) is a complex Lie subalgebra of \(\mathfrak q\), and one has $K^{(r)}_o\simeq\mathfrak q^{(r)}/(\mathfrak q\cap\overline{\mathfrak q})$ for $r\ge1$.
Thus Levi nondegeneracy means $1$-nondegeneracy, i.e.\ $\mathfrak q^{(1)}=\mathfrak q\cap\overline{\mathfrak q}$, i.e. the Levi form has trivial kernel.
\begin{definition}
A $G_0$-homogeneous CR manifold with associated CR algebra
$(\mathfrak g_0,\mathfrak q)$ at $p$ is said to be
$k$-nondegenerate at $p$ if $k$ is the smallest positive
integer such that $\mathfrak q^{(k)}=\mathfrak q\cap\overline{\mathfrak q}$. In particular, a totally real CR manifold is
$1$-nondegenerate according to this convention.
\end{definition}
This formalism is especially natural for a real orbit $M$ in complex flag
manifolds $F$, since its  embedding  yields natural $CR$-manifold structures. Let $F_{\Pi_0}=  G/Q_{\Pi_0}$ be as above, let
$G_{0}\subset   G$ be a real form with Lie algebra $\mathfrak g_0$, and let $M=G_{0}\cdot p\subset X_{\Pi_0}$
be a real orbit, with $p\in F_{\Pi_0}$. Since $F_{\Pi_0}$ is a complex manifold, $M$ inherits the natural
CR structure $$H_{p}^{0,1}M=T_{p}^{\mathbb C}M\cap T_{p}^{0,1}F_{\Pi_0}.$$ If $p=gQ_{\Pi_0}\in F_{\Pi_0}$, then the isotropy subgroup of the orbit at $p$
is $I_{p}
=
G_{0}\cap gQ_{\Pi_0}g^{-1},$ with Lie algebra $\mathfrak i_{p}
=
\mathfrak g_0\cap \operatorname{Ad}(g)(\mathfrak q_{\Pi_0}).$
The corresponding CR algebra is given by $(\mathfrak g_0,\mathfrak q_{p}),
$ with $
\mathfrak q_{p}:=\operatorname{Ad}(g)(\mathfrak q_{\Pi_0}).$ 
Namely every $G_0$ real orbit in a  complex flag manifold is a homogeneous  $CR$ manifold with $G_0$-invariant CR structure, and the associated $CR$ algebra at $p$ is given by $(\mathfrak g_0,\mathfrak q_{p})$ where $\mathfrak q_{p}$ is parabolic.

\section{Orbits on the complete flag $G_{2}/B$}
\subsection{Counting orbits on $G/B$}

In this section, we will count the number of orbits of the action of the split real form $G_{2(2)}$ on the flag manifold $G_{2}/B$. Before we proceed with the calculations in this specific case, we will describe the general algorithm provided in \cite{Aomoto1966}. 
Let $\mathfrak{g}_{0}$ be a real semisimple Lie algebra with complexification $\mathfrak{g}$ and denote by $\Tilde{\sigma}$ the conjugation determined by $\mathfrak{g}_{0}$. Let $\mathfrak{g}_{u}$ be a compact real form of $\mathfrak{g}$ that is invariant under $\Tilde{\sigma}$ and let $\sigma$ denote the restriction of $\Tilde{\sigma}$ to $\mathfrak{g}_{u}$. Let $\tilde{\theta}$ be the conjugation of $\mathfrak{g}$ with respect to $\mathfrak{g}_{u}$. We know that $\tilde{\theta}(\mathfrak{g}_{0})=\mathfrak{g}_{0}$. Let $\theta=\tilde{\theta}_{|\mathfrak{g}_{0}}$. One then has that  $\theta$ is a Cartan involution of $\mathfrak{g}_{0}$. Let $\mathfrak{g}_{0}=\mathfrak{k}_{0}\oplus\mathfrak{p}_{0}$
be the Cartan decomposition of $\mathfrak{g}_{0}$ with respect to $\theta$. Consider $\mathfrak{h}_{0}$ to be a CSA of $\mathfrak{g}_{0}$ that is $\theta$-invariant and let $\mathfrak{h}_{0}=\mathfrak{h}_{\mathfrak{k}_{0}}\oplus\mathfrak{h}_{\mathfrak{p}_{0}}$ be its decomposition into its toral and vectorial parts. Then, $\mathfrak{h}_{u}=\mathfrak{h}_{\mathfrak{k}_{0}}\oplus i \mathfrak{h}_{\mathfrak{p}_{0}}$ is a CSA of $\mathfrak{g}_{u}$ and $\mathfrak{h}=(\mathfrak{h}_{u})^{\mathbb{C}}$ is a CSA of $\mathfrak{g}.$ Let $\Delta$ indicate the root system of $\mathfrak{g}$ with respect to $\mathfrak{h}$ and denote by $\Pi$ a basis. Let $\Delta_{+}$ denote the set of positive roots of $\Delta$ with respect to $\Pi$. 
Let $G$ be a connected complex semisimple Lie group with Lie algebra $\mathfrak{g}$ such that $\Tilde{\sigma}$ can be lifted from $\mathfrak{g}$ to $G$. Let $G_{0}$ denote the set of fixed points of $\Tilde{\sigma}$ and let $G_{u}$ be the connected Lie subgroup of $G$ with Lie algebra $\mathfrak{g}_{u}$. Consider $T$ to be the maximal torus of $G_{u}$ with Lie algebra $\mathfrak{h}_{u}$. Denote by $W=N(T)/T$ the Weyl group of $G_{u}$ with respect to $\mathfrak{h}_{u}$. For each $w\in W$, we choose one representative element $\omega(w)$ inside $N(T).$  More specifically, one has that $\operatorname{Ad}(\omega(w))_{|\mathfrak{h}_{u}}=w.$  
Let us now fix $w\in W$ with representative $\omega$ and suppose that $\Tilde{\sigma}\circ w$ is an involution on $\mathfrak{h}_{u}$. Let $W_{\operatorname{inv}}$ denote the set of all such elements. We can consider the eigenspace decomposition of $\mathfrak{h}_{u}$ given by the two subspaces
\[
    \mathfrak{h}^{+}_{w}=\{\ X\in\mathfrak{h}_{u}\ |\ \Tilde{\sigma}\circ w(X)=X\};\qquad
    \mathfrak{h}^{-}_{w}=\{X\in\mathfrak{h}_{u}\ |\ \Tilde{\sigma}\circ w(X)=-X\}.\]

Let $T^{+}_{w}$ and $T^{-}_{w}$ denote the connected Lie subgroups of $T$ with Lie algebras $\mathfrak{h}^{+}_{w}$ and $\mathfrak{h}^{-}_{w}$ respectively. 
Given $w\in \operatorname{W}_{\text{inv}}$ with $\omega(w)=\omega$, set $a_{0}^{-1}=\tilde{\sigma}(\omega(w))\omega(w)$. We may now define the following sets:
\begin{alignat*}{2}
    &\Phi=\{ g\in G\ |\ \Tilde{\sigma}(g)=g^{-1}\},\qquad& &\Phi_{0}=\{ g\in G\ |\ g=g_{1}\Tilde{\sigma}(g_{1})^{-1}\ \text{for some $g_{1}\in G$}\};\\
    &\hat{T}^{+}_{w}=\{a\in T^{+}_{w}\ |\ a^{2}=a_{0}\},\qquad&&
    \Psi=\{\omega(w) a\ |\  w\in W_{\operatorname{inv}},\ a\in \hat{T}^{+}_{w}\};\\
    &\Psi_{0}=\Psi\cap\Phi_{0};
\end{alignat*}
We now define an equivalence relation $\sim$ on $\Psi$. We say that $\omega a\sim \omega' a'$ if and only if $\omega=\omega'$ and $a\equiv a'\ \operatorname{mod}(T^{+}_{w}\cap T^{-}_{w})$. Let $[\Psi]$ indicate the set of equivalence classes of $\Psi$ with respect to the relation $\sim$. 
We can now consider the following subsets of $G_{u}:$
\[
    \Theta=\{g\in G_{u}\ |\ \Tilde{\sigma}(g)=g^{-1}\};\qquad
    \Theta_{0}=\{g\in G_{u}\ |\ g=h\Tilde{\sigma}(h)^{-1}\ \text{for some $h\in G_{u}$}\}.
\]
It can be proven that $\Theta_{0}$ coincides with $\operatorname{exp(i\mathfrak{p}_{0}}).$ 
We have thus established all the necessary preliminaries to count the number of $G_{0}$-orbits on the full flag $G/B$. Indeed, one may show that $\Psi_{0}=\Theta_{0}\cap\Psi$. For every $u\in \Theta_{0}$, there exists $\sqrt{u}\in\Theta_{0}$ such that $(\sqrt{u})^{2}=u$. Let us choose a square root for each element of $\Psi$ and define
\[\sqrt{[\Psi_{0}]}=\{\sqrt{\omega a}\ |\ \omega a\in [\Psi_{0}]\}.\]
We can now state the following
\begin{theorem}[Theorem $2$ in \cite{Aomoto1966}]\label{teo: Aomoto, teorema finitezza orbite sul flag completo}
$G=B\cdot\sqrt{[\Psi_{0}]}\cdot G_{0}$ holds.
The double coset $B\backslash G/G_{0}$ is finite and can be identified with $\sqrt{[\Psi_{0}]}$ canonically.
\end{theorem}
We now turn to the problem of determining the dimension of each orbit. Fix $\gamma=\sqrt{\omega a}$ for some $\omega a\in [\Psi_{0}].$ Consider the $G_{0}-$orbit through $B\gamma$. Its isotropy group is
\[S=\{g\in G_{0}\ |\ \gamma g \gamma^{-1}\in B\}.\]
Define $\Tilde{\sigma_{1}}=\operatorname{Ad}(\omega a)\circ \Tilde{\sigma}$ and let $\mathfrak{h}_{1}$ be the set of $\Tilde{\sigma}_{1}-$fixed points. Then, $\mathfrak{h}_{1}$ is a real form of $\mathfrak{h}$. Define the following subsets of $\Delta_{+}:$
\[\Delta_{+}^{\Tilde{\sigma}_{1}}=\{\alpha\in \Delta_{+}\ |\ \Tilde{\sigma}_{1}(\alpha)>0\},\qquad \Delta_{-}^{\Tilde{\sigma}_{1}}=\{\alpha\in \Delta_{+}\ |\ \Tilde{\sigma}_{1}(\alpha)<0\}.\]
Let $\mathfrak{n}_{1}$ be the Lie algebra
\[\mathfrak{n}_{1}=\sum_{\alpha\in \Delta_{+}^{\sigma_{1}}}\mathbb{R}(X_{\alpha}+\sigma_{1}X_{\alpha})+\sum_{\alpha\in \Delta_{+}^{\sigma_{1}}}i\mathbb{R}(X_{\alpha}-\sigma_{1}X_{\alpha}).\]
One has that the Lie algebra $\mathfrak{i}_{\gamma}$ of the isotropy group $I_{\gamma}$ is $\mathfrak{i}_{\gamma}=\operatorname{Ad}_{\gamma^{-1}}(\mathfrak{h}_{1}+\mathfrak{n}_{1})$.
Let $m_{1}=\operatorname{dim}\mathfrak{n}_{1}$. Then, it can be proven that 
\[\operatorname{dim}\mathcal{O}(\gamma)=|\Delta|-m_{1}.\]
\begin{remark}\label{rem: from right to left}
The algorithm described in \cite{Aomoto1966} computes the number and the dimension of the 
$G_{0}$-orbits on $B\backslash G$, while we will study the left $G_{0}$-orbits on $G/B$. 
Assume $\tilde\sigma(B)=B$, which holds for the split real form $G_{2(2)}$: all roots are 
real, since $\tilde\sigma=-1$ on $\mathfrak h_{u}$ and $\alpha(\mathfrak h_{u})\subset 
i\mathbb R$, so $\tilde\sigma(\mathfrak g_{\alpha})=\mathfrak g_{\alpha}$ for every 
$\alpha\in\Delta$. Then the map 
 \begin{align*}
        \phi:&\ B\backslash G\ \rightarrow\ G/B\\
        &Bg\ \mapsto\ \tilde{\sigma}(g)^{-1}B.
    \end{align*}
is a well defined bijection, because $\tilde\sigma(bg)^{-1}\in\tilde\sigma(g)^{-1}B$ for 
$b\in B$. It is anti-$G_{0}$-equivariant, in fact  for $h\in G_{0}$, $\phi(Bgh)=h^{-1}\cdot\phi(Bg).$ Hence $\phi$ induces a bijection $B\backslash G/G_{0}\to G_{0}\backslash G/B$, and 
$\mathrm{Stab}_{G_{0}}(\phi(x))=\mathrm{Stab}_{G_{0}}(x)$ for every $x\in B\backslash G$. In 
particular, corresponding orbits have the same real dimension. Let $\gamma_{1},\dots,\gamma_{m}\in\sqrt{[\Psi_{0}]}\subset\Theta_{0}$ be representatives of $B\backslash G/G_{0}$, see Theorem~\ref{teo: Aomoto, teorema finitezza orbite sul flag completo}. By anti-equivariance, $\tilde\sigma(\gamma_{1})^{-1}B,\dots,\tilde\sigma(\gamma_{m})^{-1}B$ are representatives of the left $G_{0}$-orbits on $G/B$. Since $\tilde\sigma(\gamma)=\gamma^{-1}$ for $\gamma\in\Theta_{0}$, we have $\tilde\sigma(\gamma_{j})^{-1}B=\gamma_{j}B$, so the same elements $\gamma_{1},\dots,\gamma_{m}$ can be used, and $\dim G_{0}\gamma_{j}B=\dim B\gamma_{j}G_{0},$ since $G_{0}\cap\gamma_{j}B\gamma_{j}^{-1}=G_{0}\cap\gamma_{j}^{-1}B\gamma_{j}$.
\end{remark}

\subsection{Explicit calculations for $G_{2}/B$}\label{sec:explicit-G2-B}
Let us apply the algorithm described in the previous section to compute the number of $G_{2(2)}-$orbits on the  complete flag manifold $G_{2}/B$ of real dimension $12$. As explained in Remark \ref{rem: from right to left}, it will suffice to compute the number and the dimension of the $G_{0}-$orbits on the coset space $B\backslash G$.\\
First of all, let us remark that the involution $\tilde{\sigma}$ of $\mathfrak{g}_{2}$ defining the split real form $\mathfrak{g}_{2(2)}$ simply consists of complex conjugation. The real form $\mathfrak{g}_{u}$ we constructed in section \ref{sec: all on g2} is clearly $\tilde{\sigma}-$invariant and determines the Cartan involution $\theta$ we described above. We further remark that the Cartan subalgebra $\mathfrak{h}_{0}$ of $\mathfrak{g}_{2(2)}$  illustrated in \eqref{eq: Cartan subalgebra} is $\theta-$invariant. This allows us to apply the machinery developed in the previous sections.
Because $\mathfrak{h}_{\mathfrak{k}_{0}}=0$, we have that $\mathfrak{h}_{u}=i\mathfrak{h}_{0}.$ The conjugation $\sigma=\tilde{\sigma}_{|\mathfrak{g}_{u}}$ sends an element $H$ of $\mathfrak{h}_{u}$ to $-H$. Let $W$ be the Weyl group of $\mathfrak{g}_{2}$. Because of the way $\sigma$ acts on $\mathfrak{h}_{u}$, we know that every element $w\in W$ commutes with $\sigma.$  As a result, $\sigma\circ w$ is an involution if and only if $w$ is an involution. Therefore, to determine the various orbits, we will have to look at the conjugacy classes of elements of $W$ that act as involutions. The conjugacy classes in this case are exactly $4$ and can be represented as follows:
\begin{itemize}
    \item [$(1)$] $\boxed{w_1}$ $\operatorname{I}_7=\operatorname{exp}(i0)$;
    \item[$(2)$] \begin{itemize}
        \item [$\boxed{w_{2}}$] $\begin{pmatrix}
        1&0&0\\
        0&\begin{pmatrix}
            0&i&0\\
            i&0&0\\
            0&0&1\end{pmatrix}&0\\
            0&0&\begin{pmatrix}
            0&-i&0\\
            -i&0&0\\
            0&0&1
        \end{pmatrix}\end{pmatrix}=\operatorname{exp}\left(i\frac{\pi}{2}(X_{\epsilon_1-\epsilon_2}+X_{\epsilon_2-\epsilon_1})\right);$
        \item [$\boxed{w_{3}}$] $\begin{pmatrix}
        1&0&0\\
        0&\begin{pmatrix}
            1&0&0\\
            0&0&i\\
            0&i&0\end{pmatrix}&0\\
            0&0&\begin{pmatrix}
            1&0&0\\
            0&0&-i\\
            0&-i&0
        \end{pmatrix}\end{pmatrix}=\operatorname{exp}(i\frac{\pi}{2}(X_{\epsilon_2-\epsilon_3}+X_{\epsilon_3-\epsilon_2}));$
              \item [$\boxed{w_{4}}$] $\begin{pmatrix}
        1&0&0\\
        0&\begin{pmatrix}
            0&0&i\\
            0&1&0\\
            i&0&0\end{pmatrix}&0\\
            0&0&\begin{pmatrix}
            0&0&-i\\
            0&1&0\\
            -i&0&0
        \end{pmatrix}\end{pmatrix}=\operatorname{exp}(i\frac{\pi}{2}(X_{\epsilon_1-\epsilon_3}+X_{\epsilon_3-\epsilon_1}));$
        \end{itemize}
     \item[$(3)$] 
     \begin{itemize}
            \item [$\boxed{w_{5}}$]$\begin{pmatrix}
            -1&0&0\\
        0&0&\begin{pmatrix}
                0&i&0\\
                -i&0&0\\
                0&0&1
            \end{pmatrix}\\
            0&\begin{pmatrix}
                0&-i&0\\
                i&0&0\\
                0&0&1
            \end{pmatrix}&0\end{pmatrix}=\operatorname{exp}(i\frac{\pi}{2}(X_{\epsilon_3}+X_{-\epsilon_3}));$
                            \item [$\boxed{w_{6}}$]
        $\begin{pmatrix}
            -1&0&0\\
        0&0&\begin{pmatrix}
                1&0&0\\
                0&0&i\\
                0&-i&0
            \end{pmatrix}\\
            0&\begin{pmatrix}
                1&0&0\\
                0&0&-i\\
                0&i&0
            \end{pmatrix}&0\end{pmatrix}=\operatorname{exp}(i\frac{\pi}{2}(X_{\epsilon_1}+X_{\epsilon_{-1}}));$
                            \item [$\boxed{w_{7}}$]
        $\begin{pmatrix}
            -1&0&0\\
        0&0&\begin{pmatrix}
                0&0&-i\\
                0&1&0\\
                i&0&0
            \end{pmatrix}\\
            0&\begin{pmatrix}
                0&0&i\\
                0&1&0\\
                -i&0&0
            \end{pmatrix}&0\end{pmatrix}=\operatorname{exp}(i\frac{\pi}{2}(X_{\epsilon_2}+X_{-\epsilon_2}));$
            \end{itemize}
            \item[$(4)$] $\boxed{w_{8}}\begin{pmatrix}
                -1&0&0\\
                0&0&-I_{3}\\
                0&-I_{3}&0
            \end{pmatrix};$
    \end{itemize}
We shall compute the number of orbits and the dimension of the orbits in each of these cases. We note that in each case $\sigma(\omega)\omega=1.$
\begin{itemize}
    \item [$(1)$]
     $\sigma\circ w([0,ih_{1},ih_{2},ih_{3},-ih_{1},-ih_{2},-ih_{3}])=-[0,ih_{1},ih_{2},ih_{3},-ih_{1},-ih_{2},-ih_{3}].$\\
    Therefore, we have that:
    {\small
    \begin{alignat*}{2}
    &\mathfrak{h}^{+}_{w}=\{0\},\qquad&& \mathfrak{h}^{-}_{w}=\mathfrak{h};\\  
    &T^{+}_{w}=\{I_{7}\},\qquad &&T^{-}_{w}=\{[1,z_{1},z_{2},(z_{1}z_{2})^{-1},z_{1}^{-1},z_{2}^{-1},z_{1}z_{2}]\};\\
    &\hat{T}^{+}_{w}=\{I_{7}\},\qquad&& T^{+}_{w}\cap T^{-}_{w}=\{I_{7}\}\\
    &[\Psi]=[\Psi_{0}]=\{[I_{7}]\},\qquad&& 
    \sqrt{[\Psi_{0}]}=\{[I_{7}]\}.
    \end{alignat*}
    }

As for the dimensions, we remark that $\tilde{\sigma}_{1}=\operatorname{Ad}(I_{7})=\operatorname{Id}$ acts trivially on the roots, so that $\Delta^{+}_{\tilde{\sigma}_{1}}=\Delta^{+}.$
Let us compute the action of $\tilde{\sigma}_{1}$ on the set of positive roots of $\Delta$ with respect to the basis $\Pi=\{\alpha_{1},\alpha_{2}\}.$ In this case of $\tilde{\sigma}_{1}=\operatorname{Id}$, so we  have that $\Delta^{+}_{\tilde{\sigma_{1}}}=\Delta^{+}.$
We infer that that
\[\mathfrak{n}_{1}=\mathbb{R}X_{\alpha_{1}}+\mathbb{R}X_{\alpha_{2}}+\mathbb{R}X_{\alpha_{1}+\alpha_{2}}+\mathbb{R}X_{3\alpha_{1}+\alpha_{2}}+\mathbb{R}X_{2\alpha_{1}+\alpha_{2}}+\mathbb{R}X_{3\alpha_{1}+2\alpha_{2}}\]
and so $m_{1}=6$.
Thus, $\operatorname{dim}_{\mathbb{R}}(\mathcal{O}(I_{7}))=6.$ This shows that $\mathcal{O}(I_{7})$ is the closed orbit.
    \item[$(2)$] Now, we will work with $w_{2},w_{3}$ and $w_{4}.$
    \begin{itemize}
    \item[$\boxed{w=w_{2}}$]
    $\sigma\circ w([0,ih_{1},ih_{2},ih_{3},-ih_{1},-ih_{2},-ih_{3}])=-[0,ih_{2},ih_{1},ih_{3},-ih_{2},-ih_{1},-ih_{3}].$
    {\small
    \begin{alignat*}{2}
    &\mathfrak{h}^{+}_{w}=\{[0,ih_{1},-ih_{1},0,-ih_{1},ih_{1},0]\},\qquad && \mathfrak{h}^{-}_{w}=\{[0,ih_{1},ih_{1},-2ih_{1},-ih_{1},-ih_{1},2ih_{1}]\};\\
    &T^{+}_{w}=\{[1,z_{1},z_{1}^{-1},1,z_{1}^{-1},z_{1},1]\},\qquad && T^{-}_{w}=\{[1,z_{1},z_{1},z_{1}^{-2},z_{1}^{-1},z_{1}^{-1},z_{1}^{2}\};\\
    &\hat{T}^{+}_{w}=\{I_{7},[1,-1,-1,1,-1,-1,1]\},\qquad&& T^{+}_{w}\cap T^{-}_{w}=\{I_{7},[1,-1,-1,1,-1,-1,1]\};\\
    &[\Psi]=[\Psi_{0}]=\{[\omega(w_{2})]\},\qquad&& \sqrt{[\Psi_{0}]}=\{\gamma_{2}=\operatorname{exp}(i\frac{\pi}{4}(X_{\epsilon_1-\epsilon_2}+X_{\epsilon_2-\epsilon_1}))\};
    \end{alignat*}
    }\[\tilde{\sigma}_{1}(\alpha)([0,z_{1},z_{2},z_{3},-z_{1},-z_{2},-z_{3}])=\alpha([0,z_{2},z_{1},z_{3},-z_{2},-z_{1},-z_{3}]).\]
As a result, we have that:
{\small
\begin{alignat*}{3}
&\tilde{\sigma}_{1}(\alpha_{1})=\alpha_{1}+\alpha_{2},\quad  &&\tilde{\sigma}_{1}(\alpha_{2})=-\alpha_{2},\quad &&\tilde{\sigma}_{1}(\alpha_{1}+\alpha_{2})=\alpha_{1};\\
&\tilde{\sigma}_{1}(3\alpha_{1}+\alpha_{2})=3\alpha_{1}+2\alpha_{2},\quad
&&{\sigma}_{1}(2\alpha_{1}+\alpha_{2})=2\alpha_{1}+\alpha_{2},\quad 
&&\tilde{\sigma}_{1}(3\alpha_{1}+2\alpha_{2})=3\alpha_{1}+\alpha_{2}.
\end{alignat*}
}
So $\Delta^{+}_{\tilde{\sigma_{1}}}=\{\alpha_{1},\alpha_{1}+\alpha_{2}, 3\alpha_{1}+\alpha_{2},2\alpha_{1}+\alpha_{2},3\alpha_{1}+2\alpha_{2}\}.$\\
It thus follows that
\begin{alignat*}{3}
    \mathfrak{n}_{1}=&\mathbb{R}(X_{\alpha_{1}}+X_{\alpha_{1}+\alpha_{2}})&+&i\mathbb{R}(X_{\alpha_{1}}-X_{\alpha_{1}+\alpha_{2}})\\
            +&\mathbb{R}(X_{3\alpha_{1}+\alpha_{2}}+X_{3\alpha_{1}+2\alpha_{2}})&+&i\mathbb{R}(X_{3\alpha_{1}+\alpha_{2}}-X_{3\alpha_{1}+2\alpha_{2}})\\
            +&\mathbb{R}(X_{2\alpha_{1}+\alpha_{2}})
        \end{alignat*}
and so $m_{1}=5.$
As a consequence, $\operatorname{dim}_{\mathbb{R}}(\mathcal{O}(\gamma_{2}))=7$.
\item[$\boxed{w=w_{3}}$] $\sigma\circ w([0,ih_{1},ih_{2},ih_{3},-ih_{1},-ih_{2},-ih_{3}])=-[0,ih_{1},ih_{3},ih_{2},-ih_{1},-ih_{3},-ih_{2}].$
Therefore, we have that:
{\small
\begin{align*}
&\mathfrak{h}^{+}_{w}=\{[0,0,ih_{3},-ih_{3},0,-ih_{3},ih_{3}]\},\: &&\mathfrak{h}^{-}_{w}=\{[0,-2ih_{2},ih_{2},ih_{2},2ih_{2},-ih_{2},-ih_{2}]\};\\
&T^{+}_{w}=\{[1,1,z_{3},z_{3}^{-1},1,z_{3}^{-1},z_{3}]\},\:&& T^{-}_{w}=\{[1,z_{2}^{-2},z_{2},z_{2},z_{2}^{2},z_{2}^{-1},z_{2}^{-1}\};\\
&\hat{T}^{+}_{w}=\{I_{7},[1,1,-1,-1,1,-1,-1]\},\: &&T^{+}_{w}\cap T^{-}_{w}=\{I_{7},[1,1,-1,-1,1,-1,-1]\};\\
&[\Psi]=[\Psi_{0}]=\{[\omega(w_{3})]\},\:&& \sqrt{[\Psi_{0}]}=\{\gamma_{3}=\operatorname{exp}(i\frac{\pi}{4}(X_{\epsilon_2-\epsilon_3}+X_{\epsilon_3-\epsilon_2}))\};
\end{align*}
}\[\tilde{\sigma}_{1}(\alpha)([0,z_{1},z_{2},z_{3},-z_{1},-z_{2},-z_{3}])=\alpha([0,z_{1},z_{3},z_{2},-z_{1},-z_{3},-z_{2}]).\]
As a result, we have that:
{\small
\begin{alignat*}{3}
&\tilde{\sigma}_{1}(\alpha_{1})=-(2\alpha_{1}+\alpha_{2}),\qquad&&\tilde{\sigma}_{1}(\alpha_{2})=3\alpha_{1}+2\alpha_{2},\qquad&&\tilde{\sigma}_{1}(\alpha_{1}+\alpha_{2})=\alpha_{1}+\alpha_{2},\\
&\tilde{\sigma}_{1}(3\alpha_{1}+2\alpha_{2})=\alpha_{2},\qquad&&\tilde{\sigma}_{1}(2\alpha_{1}+\alpha_{2})=-\alpha_{1},\qquad&&\tilde{\sigma}_{1}(3\alpha_{1}+\alpha_{2})=-(3\alpha_{1}+\alpha_{2}).
\end{alignat*}
} So $\Delta^{+}_{\tilde{\sigma_{1}}}=\{\alpha_{2},\alpha_{1}+\alpha_{2},3\alpha_{1}+2\alpha_{2}\}.$\begin{align*}
     \mathfrak{n}_{1}=&\mathbb{R}(X_{\alpha_{2}}+X_{3\alpha_{1}+2\alpha_{2}})+i\mathbb{R}(X_{\alpha_{2}}-X_{3\alpha_{1}+2\alpha_{2}})\\
            +&\mathbb{R}X_{\alpha_{1}+\alpha_{2}}   
    \end{align*} and so $m_{1}=3.$
Thus, $\operatorname{dim}_{\mathbb{R}}(\mathcal{O}(\gamma_{3}))=9$.
\item[$\boxed{w=w_{4}}$]$\sigma\circ w([0,ih_{1},ih_{2},ih_{3},-ih_{1},-ih_{2},-ih_{3}])=-[0,ih_{3},ih_{2},ih_{1},-ih_{3},-ih_{2},-ih_{1}].$
{\small
    \begin{alignat*}{2}
    &\mathfrak{h}^{+}_{w}=\{[0,ih_{1},0,-ih_{1},-ih_{1},0,ih_{1}]\},\qquad &&\mathfrak{h}^{-}_{w}=\{[0,ih_{1},-2ih_{1},ih_{1},-ih_{1},2ih_{1},-ih_{1}]\};\\
    &T^{+}_{w}=\{[1,z_{1},1,z_{1}^{-1},z_{1}^{-1},1,z_{1}]\},\qquad &&T^{-}_{w}=\{[1,z_{1},z_{1}^{-2},z_{1},z_{1}^{-1},z_{1}^{2},z_{1}^{-1}\};\\
    &\hat{T}^{+}_{w}=\{I_{7},[1,-1,1,-1,-1,1,-1]\},\qquad &&
    T^{+}_{w}\cap T^{-}_{w}=\{I_{7},[1,-1,1,-1,-1,1,-1]\};\\
    &[\Psi]=[\Psi_{0}]=\{[\omega(w_{4})]\},\qquad &&\sqrt{[\Psi_{0}]}=\{\gamma_{4}=\operatorname{exp}(i\frac{\pi}{4}(X_{\epsilon_1-\epsilon_3}+X_{\epsilon_3-\epsilon_1})\}.
    \end{alignat*}
    }\[\tilde{\sigma}_{1}(\alpha)([0,z_{1},z_{2},z_{3},-z_{1},-z_{2},-z_{3}])=\alpha([0,z_{3},z_{2},z_{1},-z_{3},-z_{2},-z_{1}]).\]
As a result, we have that:
{\small
\begin{alignat*}{3}
&\tilde{\sigma}_{1}(\alpha_{1})=\alpha_{1},\quad &&\tilde{\sigma}_{1}(\alpha_{2})=-(3\alpha_{1}+\alpha_{2}),\quad&&\tilde{\sigma}_{1}(\alpha_{1}+\alpha_{2})=-(2\alpha_{1}+\alpha_{2});\\
&\tilde{\sigma}_{1}(3\alpha_{1}+\alpha_{2})=-\alpha_{2},\quad&& \tilde{\sigma}_{1}(2\alpha_{1}+\alpha_{2})=-(\alpha_{1}+\alpha_{2}),\quad &&\tilde{\sigma}_{1}(3\alpha_{1}+2\alpha_{2})=-(3\alpha_{1}+2\alpha_{2}).
\end{alignat*}
}So $\Delta^{+}_{\tilde{\sigma_{1}}}=\{\alpha_{1}\}.$
It thus follows that $\mathfrak{n}_{1}=\mathbb{R}X_{\alpha_{1}}$ and so $m_{1}=1.$
As a result, $\operatorname{dim}_{\mathbb{R}}(\mathcal{O}(\gamma_{4}))=11.$
\end{itemize}
\medskip
\item[$(3)$] Let us work in the conjugacy class of $w_{5},w_{6}$ and $w_{7}$. The calculations are similar to the previous case and will therefore be omitted.
\begin{itemize}
\item [$\boxed{w=w_{5}}$]
\[
    [\Psi]=[\Psi_{0}]=\{[\omega(w_{5})]\},\quad\sqrt{[\Psi_{0}]}=\{\gamma_{5}=\operatorname{exp}(i\frac{\pi}{4}(X_{\epsilon_3}+X_{-\epsilon_3}))\}.
    \]  
\[\Delta^{+}_{\tilde{\sigma_{1}}}=\{\alpha_{2}\}.\]
It thus follows that $\mathfrak{n}_{1}=\mathbb{R}X_{\alpha_{2}}$
and so $m_{1}=1.$
We conclude that $\operatorname{dim}_{\mathbb{R}}(\mathcal{O}(\gamma_{5}))=11.$
\item[$\boxed{w=w_{6}}$]
\[[\Psi]=[\Psi_{0}]=\{[\omega(w_{6})]\},\qquad\sqrt{[\Psi_{0}]}=\{\gamma_{6}=\operatorname{exp}(i\frac{\pi}{4}(X_{\epsilon_1}+X_{-\epsilon_1}))\}.\]
\begin{align*}
&\Delta^{+}_{\tilde{\sigma_{1}}}=\{\alpha_{1},2\alpha_{1}+\alpha_{2},3\alpha_{1}+\alpha_{2}\}.\\
    &\mathfrak{n}_{1}=\mathbb{R}(X_{\alpha_{1}}+X_{2\alpha_{1}+\alpha_{2}})+i\mathbb{R}(X_{\alpha_{1}}-X_{2\alpha_{1}+\alpha_{2}})
            +\mathbb{R}X_{3\alpha_{1}+\alpha_{2}}    
\end{align*}
and so $m_{1}=3.$ Thus, $\operatorname{dim}_{\mathbb{R}}(\mathcal{O}(\gamma_{6}))=9.$
\item [$\boxed{w=w_{7}}$]
         \[[\Psi]=[\Psi_{0}]=\{[\omega(w_{7})]\},\qquad\sqrt{[\Psi_{0}]}=\{\gamma_{7}=\operatorname{exp}(i\frac{\pi}{4}(X_{\epsilon_2}+M_{-\epsilon_2}))\}.\]
\[\Delta^{+}_{\tilde{\sigma_{1}}}=\{\alpha_{2},3\alpha_{1}+\alpha_{2},\alpha_{1}+\alpha_{2},2\alpha_{1}+\alpha_{2},3\alpha_{1}+2\alpha_{2}\}.\]\begin{alignat*}{2}
    \mathfrak{n}_{1}=&\mathbb{R}(X_{\alpha_{2}}+X_{3\alpha_{1}+\alpha_{2}})+&&i\mathbb{R}(X_{\alpha_{2}}-X_{3\alpha_{1}+\alpha_{2}})\\
            +&\mathbb{R}(X_{\alpha_{1}+\alpha_{2}}+X_{2\alpha_{1}+\alpha_{2}})+&&i\mathbb{R}(X_{\alpha_{1}+\alpha_{2}}-X_{2\alpha_{1}+\alpha_{2}})\\
            +&\mathbb{R}(X_{3\alpha_{1}+2\alpha_{2}})
        \end{alignat*}
and so $m_{1}=5.$ This means that
$\operatorname{dim}_{\mathbb{R}}(\mathcal{O}(\gamma_{7}))=7.$
\end{itemize}
\item[$(4)$] Finally, we examine the case $w=w_{8}:$
$$\sigma\circ w([0,ih_{1},ih_{2},ih_{3},-ih_{1},-ih_{2},-ih_{3}])=[0,ih_{1},ih_{2},ih_{3},-ih_{1},-ih_{2},-ih_{3}].$$
    Therefore, we have that:
    {\small
    \begin{alignat*}{2}
    &\mathfrak{h}^{+}_{w}=\mathfrak{h},\quad&&  \mathfrak{h}^{-}_{w}=\{0 \};\\
    &T^{+}_{w}=\{[1,z_{1},z_{2},(z_{1}z_{2})^{-1},z_{1}^{-1},z_{2}^{-1},z_{1}z_{2}]\},\quad&&T^{-}_{w}=\{I_{7}\}.
    \end{alignat*}}
    As a consequence, we obtain that:
    {\small
    \begin{align*}
    & \hat{T}^{+}_{w}=\{I_{7},[1,1,-1,-1,1,-1,-1],    [1,-1,1,-1,-1,1,-1],[1,-1,-1,1,-1,-1,1]\}.\\
    &T^{+}_{w}\cap T^{-}_{w}=\{I_{7}\}.\\
    &[\Psi]=\{[\omega(w_{8})],[M_{1}=\omega(w_{8})[1,1,-1,-1,1,-1,-1]],
    [M_{2}=\omega(w_{8})[1,-1,1,-1,-1,1,-1]],\\&[M_{3}=\omega(w_{8})[1,-1,-1,1,-1,-1,1]]\}. 
    \end{align*}}Unlike the previous cases, we have that $[\Psi_{0}]$ is a proper subset of $[\Psi].$ Indeed,
   $$\begin{aligned}
M_1
&=\exp\left(Y_1=\frac{i\pi}{2}M(E_{12}+E_{21},-e_3,e_3)\right)
=\begin{pmatrix}
-1&0&0\\
0&0&\begin{pmatrix}
-1&0&0\\
0&1&0\\
0&0&1
\end{pmatrix}\\
0&\begin{pmatrix}
-1&0&0\\
0&1&0\\
0&0&1
\end{pmatrix}&0
\end{pmatrix},\\
M_2
&=\exp\left(Y_2=\frac{i\pi}{2}M(E_{23}+E_{32},-e_1,e_1)\right)
=\begin{pmatrix}
-1&0&0\\
0&0&\begin{pmatrix}
1&0&0\\
0&-1&0\\
0&0&1
\end{pmatrix}\\
0&\begin{pmatrix}
1&0&0\\
0&-1&0\\
0&0&1
\end{pmatrix}&0
\end{pmatrix},\\
M_3
&=\exp\left(Y_3=\frac{i\pi}{2}M(E_{13}+E_{31},-e_2,e_2)\right)
=\begin{pmatrix}
-1&0&0\\
0&0&\begin{pmatrix}
1&0&0\\
0&1&0\\
0&0&-1
\end{pmatrix}\\
0&\begin{pmatrix}
1&0&0\\
0&1&0\\
0&0&-1
\end{pmatrix}&0
\end{pmatrix}.
\end{aligned}
$$
            However, the matrix $A=\omega(w_{8})$ doesn't belong to 
            $\Theta_{0}$ and thus $A\notin\Psi_{0}.$ Indeed, suppose by 
            contradiction that there exists a matrix $Y\in \mathfrak{p}_{0}$ 
            such that $\operatorname{exp}(iY)=A.$ Clearly $iY$ commutes 
            with its exponential and so we have that $AY=YA.$
This yields $AYA^{-1}=Y.$
However, it is easily computed that for any matrix $Y\in\mathfrak{p}_{0}$, $AYA^{-1}=-Y.$
As a result, we would get that $Y=-Y$,
which would mean $Y=0$, thus leading to a contradiction. In conclusion, we have that:
\[\sqrt{[\Psi_{0}]}=\left\{\left[\gamma_{8}=\operatorname{exp}\left(\frac{Y_{1}}{2}\right)\right],\left[\gamma_{9}=\operatorname{exp}\left(\frac{Y_{2}}{2}\right)\right],\left[\gamma_{10}=\operatorname{exp}\left(\frac{Y_{3}}{2}\right)\right]\right\}.\]
We, thus, have a total of $3$ orbits associated with this element of the Weyl group.\\
As for the dimensions, we remark that for every $i=1,2,3$, 
$\tilde{\sigma_{1}}=\operatorname{Ad}(M_{i})$ acts on roots as $\tilde{\sigma}_{1}(\alpha)([0,z_{1},z_{2},z_{3},-z_{1},-z_{2},-z_{3}])=-\alpha([0,z_{1},z_{2},z_{3},-z_{1},-z_{2},-z_{3}]).$\\
So, we get that $\Delta^{+}_{\tilde{\sigma_{1}}}=\emptyset,$ whence
$\mathfrak{n}_{1}=0$ and $m_{1}=0.$ This proves that, for $i=8,9,10$, $\operatorname{dim}_{\mathbb{R}}(\mathcal{O}(\gamma_{i}))=12$, revealing that these orbits are open.
\end{itemize}
From the above discussion and Remark \ref{rem: from right to left}, we obtain the following
\begin{lemma}\label{thm:orbit-count-complete-flag}
The action of $G_{2(2)}$ on $G_2/B$ has exactly ten orbits,
consisting of one closed orbit of real dimension $6$,
two orbits of each real dimension $7$, $9$, and $11$,
and three open orbits of real dimension $12$.
\end{lemma}
\begin{remark}\label{rem:G22-Aomoto-count}
The explicit list of representatives can be checked intrinsically using
Aomoto's orbit-counting formula \cite[Theorem~3]{Aomoto1966}. If
$\mathfrak h^1,\ldots,\mathfrak h^4$ represent the
$G_{2(2)}$-conjugacy classes of Cartan subalgebras of
$\mathfrak g_{2(2)}$, then
$\#(B\backslash G_2/G_{2(2)})=\sum_{j=1}^{4}|W/W^j|$, where
$W^j$ is the subgroup associated with
$\mathfrak h^j$. In the present case there are four Cartan classes. Indeed, by the
classification via Cayley transforms along mutually strongly
orthogonal roots \cite[Theorem~7]{Sugiura1959} (see also
\cite{Sugiura1971}), they are represented by the split Cartan
subalgebra, the two intermediate Cartan subalgebras obtained by
Cayley transforms along $\alpha_1$ and $\alpha_2$, and a fundamental
Cartan subalgebra, which may be obtained from the strongly orthogonal
pair $\{\alpha_1+\alpha_2,3\alpha_1+\alpha_2\}$. The corresponding subgroups are
$W^1=W(G_2)$,
$W^2=\langle s_{\alpha_1},s_{3\alpha_1+2\alpha_2}\rangle$,
$W^3=\langle s_{\alpha_2},s_{2\alpha_1+\alpha_2}\rangle$, and
$W^4=\langle s_{\alpha_1+\alpha_2},s_{3\alpha_1+\alpha_2}\rangle$.
Thus $|W^1|=12$ and $|W^2|=|W^3|=|W^4|=4$. Since
$|W(G_2)|=12$, Aomoto's formula gives
$\#(B\backslash G_2/G_{2(2)})=10$. This agrees with the ten
representatives $\gamma_1,\ldots,\gamma_{10}$ obtained above and
shows that the list is complete.
\end{remark}

\subsection{Equations for the $G_{2(2)}$-orbits on $G_{2}/B$}
\label{sec: equation and closure relations for G2/B}
Let us characterize each orbit of the $G_{2(2)}$-action on $G_{2}/B.$ First of all, we will use the embedding described in \ref{sec:Flag manifolds and $G_{0}$-orbits } to obtain a more geometric description of the space $G_{2}/B$.  
Recall that each element $gB$ can be identified with the pair $([g\cdot v],[g\cdot M]).$ We remark that $\operatorname{rk}(M)=2$. As a result, we may interpret $[gM]=[gMg^{-1}]$ as the plane $\pi$ spanned by the columns of the matrix $gMg^{-1}$.  Remark that up to scalar multiples, the columns of $M$ are $e_{2}$ and $e_{7}$. Let $\pi_0$ be the plane spanned by these vectors. The plane associated with $gMg^{-1}$ is spanned by $g e_{2}$ and $g e_{7}$ and can be expressed as $g\pi_0$. Note that the map sending $[g\cdot M]$ to its corresponding plane $g\pi_0$ is injective. Indeed, suppose that $g\cdot \pi_0=g'\cdot \pi_0$. This means that $h=g^{-1}g'\in\operatorname{Stab}(\pi_0)$. Clearly, $Q_{\alpha_2}\subset \operatorname{Stab}(\pi_0)$ and so $\operatorname{Stab}(\pi_0)$ is a parabolic subgroup of $G_{2}$. By the maximality of $Q_{\alpha_{2}}$ and the fact that $\operatorname{Stab}(\pi_0)$ is strictly contained in $G_2$, we get that  $Q_{\alpha_2}=\operatorname{Stab}(\pi_0)$. This means that $g Q_{\alpha_2}=g'Q_{\alpha_2}$ and so, by the injectivity of the embedding $i_2$ in \ref{sec:Flag manifolds and $G_{0}$-orbits }, we get that $[g\cdot M]=[g'\cdot M]$. We further remark that the vector $g e_{7}$, which is the first component of the pair $([g\cdot v],[g\cdot M])$ lies in the plane $\pi$ determined by the second component. This allows us to embed $G_{2}/B$ in the space of flags $F_{(1,2)}(\mathbb{C}^{7}){=}\{(V_{1},V_{2})|\ V_{1}\subset V_2\subset \mathbb{C}^{7},\ \operatorname{dim}_{\mathbb{C}}(V_{j})=j,\ j=1,2\}$. We can actually obtain a more accurate description of $G_{2}/B$ by considering some $G_{2}$-specific objects. Let $n:\ \mathbb{C}^{7}\ \rightarrow\ \mathbb{C}$ be the quadratic form on $\mathbb{C}^{7}$
whose associated matrix with respect to the standard basis is
\begin{equation}\label{N_matrix}
    N=\begin{pmatrix}
    -1&0&0\\
    0&0&-2I_{3}\\
    0&-2I_{3}&0
\end{pmatrix}.
\end{equation}
In \cite[Section~3.1]{DRAPERFONTANALS201823} the author proves that $\mathfrak{g}_{2}$ is a skew-symmetric set of endomorphisms with respect to this quadratic form via its natural action on $\mathbb{C}^{7}.$ As a consequence, $G_{2}$ acts by isometries of $n$.
A quick computation shows that $n(v)=0$. As a result, for every $g\in G_{2}$ we have that $n(g\cdot v)=0$. This proves that for every element $([v],\pi)\in G_{2}/B$, $[v]$ lies in the complex projective quadric $\mathcal{Q}=\{[w]\in\mathbb{P}(\mathbb{C}^{7})\ |\ n(w)=0\ \}.$\\
We can also give a better characterization of the planes $\pi$ via the the natural cross product $\times$, defined as follows:
\begin{equation}\label{eq:G2-cross-product}
\times:\mathbb C^7\times\mathbb C^7\longrightarrow\mathbb C^7,
\qquad
\begin{pmatrix}
s\\ u\\ v
\end{pmatrix}
\times
\begin{pmatrix}
t\\ x\\ y
\end{pmatrix}
=
\begin{pmatrix}
2u^ty-2v^tx\\
sx-tu-2v\times y\\
-sy+tv+2u\times x
\end{pmatrix}.
\end{equation}
One can show that $\mathfrak{g}_{2}$ acts as derivations of $(\mathbb{C}^{7},\times)$ and so $G_{2}$ consists of automorphisms of $\mathbb{C}^{7}$ preserving $\times$. It is easy to check that the the plane $\pi_{0}=<e_{2},e_{7}>$ corresponding to $M$ is $\times-$isotropic, i.e. for every $u,v\in \pi_{0}$, $u\times v=0$. This continues to hold for every plane $g\pi_{0}$ with $g\in G_{2}$, proving that the planes $\pi$ in the pair $([v],\pi)\in G_{2}/B$ lie in the space of all $\times-$isotropic complex planes in $\mathbb{C}^{7}.$ As a consequence, $G_{2}/B$ can be interpreted as the flags $([v],\pi)$, where $[v]$ is an $n$-isotropic line contained in the $\times$-isotropic plane $\pi$. \par
In the spirit of the classical approach to real-form orbits on complex flag manifolds, see J.~A.~Wolf \cite[Introduction]{Wolf1969}, one may distinguish the orbits by 
means of algebraic invariants associated with the real structure. In the classical setting, 
these invariants are often given by the signatures and radicals of invariant Hermitian forms 
restricted to the subspaces defining the flag (a closely related construction appears in 
\cite[Section~6]{Fels2007}, where a first example of a $3$-nondegenerate $\mathrm{SO}(3,4)^\circ$-
orbit in the Grassmannian of isotropic $2$-planes is described in terms of a symmetric 
bilinear form and its associated Hermitian form). In the present exceptional setting, the 
invariant Hermitian form does not by itself distinguish all the $G_{2(2)}$-orbits and must 
be supplemented by invariants arising from the $G_2$-invariant cross product.
That is why we consider the  sesquilinear form $q$ such that for every 
$u,w\in\mathbb{C}^{7}$, $q(u,w)=\bar{u}^{t}Nw$. Since $G_{2(2)}\subset G_{2}$ 
consists of real matrices, $q$ is  clearly $G_{2(2)}-$invariant. Another 
essential real invariant can be built from the cross product $\times.$ Set 
$\omega(v)=\bar{v}\times v$. The vanishing or not of $\omega(v)$ is a 
$G_{2(2)}-$projective invariant for $[v]$. Indeed, suppose $\omega(v)\neq 0$. 
Then, for every $\lambda\in\mathbb{C}-\{0\}$, we have that 
$\overline{\lambda\ v}\times \lambda v=|\lambda|^{2}\overline{v}\times v\neq 
0.$  Besides, if $g\in G_{2(2)}$, then $\overline{gv}\times gv=g\overline{v}\times g v=g(\overline{v}\times v)\neq 0$, because $g$ is invertible. An analogous argument shows that $\omega(v)=0$ is a $G_{2(2)}-$invariant for $[v].$ One final invariant we will use to distinguish between $G_{2(2)}$-orbits is the reality of both the isotropic lines $[v]$ and the planes $\pi$. More formally, a line $[v]$ is real if $\bar{v}\wedge v=0$, i.e. if $\bar{v}$ and $v$ are linearly dependent over $\mathbb{C}$. Likewise, a plane $\pi$ is real if $\bar{\pi}=\pi$. Because $G_{2(2)}$ commutes with conjugation, it is clear that these conditions are $G_{2(2)}$-invariants.\par
\par The main idea will be to write $G_{2}/B$ as the disjoint union of $10$ $G_{2(2)}-$invariant subsets, each containing exactly one of the orbits $\mathcal O(\gamma_j)$. The ten subsets defined below are pairwise disjoint, as one checks from their defining conditions. Those with $\operatorname{sign}(q_{|\pi})=(1,1)$, and those with $\operatorname{sign}(q_{|\pi})=(0,1)$, differ in the sign of $q([v])$. The subsets with $q_{|\pi}=0$ are separated by $\omega(v)\neq0$ and, for $\omega(v)=0$, by the reality of $\pi$ and of $[v]$. Since the ten orbits cover $G_2/B$ (Lemma~\ref{thm:orbit-count-complete-flag}), each of these subsets, being disjoint from the orbits contained in the other nine, coincides with the orbit it contains.
\begin{itemize}
    \item [$(\gamma_{1})$] 
    \[\gamma_{1}B=\left(\left[\begin{pmatrix}
        0\\
        0\\
        0\\
        0\\
        0\\
        0\\
        1\\
    \end{pmatrix}\right],\left[\begin{pmatrix}
         0 & 0 & 0 & 0 & 0 & 0 & 0 \\
0 & 0 & 0 & 1 & 0 & 0 & 0 \\
0 & 0 & 0 & 0 & 0 & 0 & 0 \\
0 & 0 & 0 & 0 & 0 & 0 & 0 \\
0 & 0 & 0 & 0 & 0 & 0 & 0 \\
0 & 0 & 0 & 0 & 0 & 0 & 0 \\
0 & 0 & 0 & 0 & -1 & 0 & 0
    \end{pmatrix}\right]\right).\]
    We remark that the plane $\pi$ associated with the matrix is real and $q$-isotropic. Besides, $\bar{v}\wedge v=0$. Therefore, $\mathcal{O}(\gamma_{1})$ lies entirely in the $G_{2(2)}-$invariant subset $\mathcal{O}_{6}$, where
    \[\mathcal{O}_{6}=\{([v],\pi)|\ [v]\subset \pi,\ q_{|\pi}=0,\ \pi=\bar{\pi},\ \bar{v}\wedge v=0\}.\]
    \item[$(\gamma_{2})$]  
    \[\gamma_{2}B=\left(\left[\begin{pmatrix}
        0\\
        0\\
        0\\
        0\\
        0\\
        0\\
        1\\
    \end{pmatrix}\right],\left[\begin{pmatrix}
          0 & 0 & 0 & 0 & 0 & 0 & 0 \\
0 & 0 & 0 & \frac{1}{2}\sqrt{2} & 0 & 0 & 0 \\
0 & 0 & 0 & \frac{1}{2}i\sqrt{2} & 0 & 0 & 0 \\
0 & 0 & 0 & 0 & 0 & 0 & 0 \\
0 & 0 & 0 & 0 & 0 & 0 & 0 \\
0 & 0 & 0 & 0 & 0 & 0 & 0 \\
0 & 0 & 0 & 0 & -\frac{1}{2}\sqrt{2} & -\frac{1}{2}i\sqrt{2} & 0
    \end{pmatrix}\right]\right).\]
    We remark that the plane $\pi$ associated with the matrix is $q$-isotropic, but is not real. Besides, $\bar{v}\wedge v=0$. Therefore, $\mathcal{O}(\gamma_{2})$ lies entirely in the $G_{2(2)}-$invariant subset $\mathcal{O}_{7}$, where
    \[\mathcal{O}_{7}=\{([v],\pi)|\ [v]\subset \pi,\ q_{|\pi}=0,\ \pi\neq \bar{\pi},\ \bar{v}\wedge v=0\}.\]
        \item[$(\gamma_{3})$] 
    \[\gamma_{3}B=\left(\left[\begin{pmatrix}
        0\\
        0\\
        0\\
        0\\
        0\\
        -i\\
        1\\
    \end{pmatrix}\right],\left[\begin{pmatrix}
          0 & 0 & 0 & 0 & 0 & 0 & 0 \\
0 & 0 & -\frac{1}{2} i \sqrt{2} & \frac{1}{2} \sqrt{2} & 0 & 0 & 0 \\
0 & 0 & 0 & 0 & 0 & 0 & 0 \\
0 & 0 & 0 & 0 & 0 & 0 & 0 \\
0 & 0 & 0 & 0 & 0 & 0 & 0 \\
0 & 0 & 0 & 0 & \frac{1}{2} i \sqrt{2} & 0 & 0 \\
0 & 0 & 0 & 0 & -\frac{1}{2} \sqrt{2} & 0 & 0
    \end{pmatrix}\right]\right).\]
     We remark that the plane $\pi$ associated with the matrix is $q$-isotropic.
    Furthermore, $\bar{v}\times v\neq 0$. Hence $\pi\neq\bar\pi$: if $\pi=\bar\pi$, then
    $\bar v\in\pi$, and $\bar v\times v=0$ because $\pi$ is $\times$-isotropic. Therefore,
    $\mathcal{O}(\gamma_{3})$ lies entirely in the $G_{2(2)}-$invariant subset
    $\mathcal{O}_{9}^{(0,0)}$, where
    \[\mathcal{O}^{(0,0)}_{9}=\{([v],\pi)|\ [v]\subset \pi,\ q_{|\pi}=0,\   \omega(v)\neq 0\}.\]
    \item[$(\gamma_{4})$] 
        \[\gamma_{4}B=\left(\left[\begin{pmatrix}
        0\\
        0\\
        0\\
        0\\
        -i\\
        0\\
        1\\
    \end{pmatrix}\right],\left[\begin{pmatrix}
         0 & 0 & 0 & 0 & 0 & 0 & 0 \\
0 & -\frac{1}{2}i & 0 & \frac{1}{2} & 0 & 0 & 0 \\
0 & 0 & 0 & 0 & 0 & 0 & 0 \\
0 & \frac{1}{2} & 0 & \frac{1}{2}i & 0 & 0 & 0 \\
0 & 0 & 0 & 0 & \frac{1}{2}i & 0 & -\frac{1}{2} \\
0 & 0 & 0 & 0 & 0 & 0 & 0 \\
0 & 0 & 0 & 0 & -\frac{1}{2} & 0 & -\frac{1}{2}i
    \end{pmatrix}\right]\right).\]
    We remark that the hermitian form $q_{|\pi}$ has signature $(1,1)$ and that $q([v])=0$. Moreover, $\bar{v}\times v\neq 0$. Therefore, $\mathcal{O}(\gamma_{4})$ lies entirely in the $G_{2(2)}-$invariant subset $\mathcal{O}^{(+,-)}_{11}$, where
    \[\mathcal{O}^{(+,-)}_{11}=\{([v],\pi)|\ [v]\subset \pi,\operatorname{sign}(q_{|\pi})=(1,1),\ q([v])=0,\omega(v)\neq 0\}.\]
        \item[$(\gamma_{5})$] 
        \[\gamma_{5}B=\left(\left[\begin{pmatrix}
        -i\\
        0\\
        0\\
        \frac{1}{2}\\
        0\\
        0\\
        \frac{1}{2}\\
    \end{pmatrix}\right],\left[\begin{pmatrix}
        0 & 0 & -\frac{1}{2}\sqrt{2} & 0 & \frac{1}{2}i\sqrt{2} & 0 & 0 \\
-\frac{1}{4}i\sqrt{2} & 0 & 0 & \frac{1}{4}\sqrt{2} & 0 & 0 & \frac{1}{4}\sqrt{2} \\
0 & 0 & 0 & 0 & 0 & 0 & 0 \\
0 & 0 & -\frac{1}{4}i\sqrt{2} & 0 & -\frac{1}{4}\sqrt{2} & 0 & 0 \\
0 & 0 & 0 & 0 & 0 & 0 & 0 \\
\frac{1}{4}\sqrt{2} & 0 & 0 & \frac{1}{4}i\sqrt{2} & 0 & 0 & \frac{1}{4}i\sqrt{2} \\
0 & 0 & -\frac{1}{4}i\sqrt{2} & 0 & -\frac{1}{4}\sqrt{2} & 0 & 0
    \end{pmatrix}\right]\right).\]
    We remark that the hermitian form $q_{|\pi}$ has signature $(0,1)$ and that $q([v])<0$.  Therefore, $\mathcal{O}(\gamma_{5})$ lies entirely in the $G_{2(2)}-$invariant subset $\mathcal{O}^{(0,-)}_{11}$, where
       \[\mathcal{O}^{(0,-)}_{11}=\{([v],\pi)|\ [v]\subset \pi,\ \operatorname{sign}(q_{|\pi})=(0,1),\ q([v])<0\}.\]
            \item[$(\gamma_{6})$] 
        \[\gamma_{6}B=\left(\left[\begin{pmatrix}
        0\\
        0\\
        i\\
        0\\
        0\\
        0\\
        1\\
    \end{pmatrix}\right],\left[\begin{pmatrix}
        0 & 0 & 0 & \frac{1}{2}i\sqrt{2} & 0 & -\frac{1}{2}\sqrt{2} & 0 \\
0 & 0 & 0 & \frac{1}{4}\sqrt{2} & 0 & \frac{1}{4}i\sqrt{2} & 0 \\
\frac{1}{4}\sqrt{2} & -\frac{1}{4}i\sqrt{2} & 0 & 0 & -\frac{1}{4}i\sqrt{2} & 0 & 0 \\
0 & 0 & 0 & 0 & 0 & 0 & 0 \\
0 & 0 & 0 & \frac{1}{4}\sqrt{2} & 0 & \frac{1}{4}i\sqrt{2} & 0 \\
0 & 0 & 0 & 0 & 0 & 0 & 0 \\
-\frac{1}{4}i\sqrt{2} & -\frac{1}{4}\sqrt{2} & 0 & 0 & -\frac{1}{4}\sqrt{2} & 0 & 0
    \end{pmatrix}\right]\right).\]
    We remark that the hermitian form $q_{|\pi}$ has signature $(0,1)$, $q([v])=0$ and $\bar{v}\times v=0$. Therefore, $\mathcal{O}(\gamma_{6})$ lies entirely in the $G_{2(2)}-$invariant subset $\mathcal{O}^{(0,-)}_{9}$, where
        \[\mathcal{O}^{(0,-)}_{9}=\{([v],\pi)|\ [v]\subset \pi,\ \operatorname{sign}(q_{|\pi})=(0,1),\ q([v])=0\}.\]
    \item[$(\gamma_{7})$] 
            \[\gamma_{7}B=\left(\left[\begin{pmatrix}
        0\\
        -i\\
        0\\
        0\\
        0\\
        0\\
        1\\
    \end{pmatrix}\right],\left[\begin{pmatrix}
        0 & 0 & 0 & 0 & 0 & 0 & 0 \\
0 & 0 & 0 & 1 & 0 & 0 & 0 \\
0 & 0 & 0 & 0 & 0 & 0 & 0 \\
0 & 0 & 0 & 0 & 0 & 0 & 0 \\
0 & 0 & 0 & 0 & 0 & 0 & 0 \\
0 & 0 & 0 & 0 & 0 & 0 & 0 \\
0 & 0 & 0 & 0 & -1 & 0 & 0
    \end{pmatrix}\right]\right).\]
    We remark that $q_{|\pi}=0$ and $\pi=\bar{\pi}$, but $\bar{v}\wedge v\neq 0$.  Therefore, $\mathcal{O}(\gamma_{7})$ lies entirely in the $G_{2(2)}-$invariant subset $\mathcal{O}^{'}_{7}$, where
    \[\mathcal{O}^{'}_{7}=\{([v],\pi)|\ [v]\subset \pi,\ q_{|\pi}= 0,\ \pi=\bar{\pi},\ \bar{v}\wedge v\neq 0\}.\]
        \item[$(\gamma_{8})$] 
            \[\gamma_{8}B=\left(\left[\begin{pmatrix}
        2i\\
        0\\
        0\\
        1\\
        0\\
        0\\
        1\\
    \end{pmatrix}\right],\left[\begin{pmatrix}
        0 & \frac{1}{2}i & -\frac{1}{2} & 0 & -\frac{1}{2}i & \frac{1}{2} & 0 \\
\frac{1}{4}i & 0 & 0 & \frac{1}{4} & 0 & 0 & \frac{1}{4} \\
-\frac{1}{4} & 0 & 0 & \frac{1}{4}i & 0 & 0 & \frac{1}{4}i \\
0 & \frac{1}{4} & \frac{1}{4}i & 0 & -\frac{1}{4} & -\frac{1}{4}i & 0 \\
-\frac{1}{4}i & 0 & 0 & -\frac{1}{4} & 0 & 0 & -\frac{1}{4} \\
\frac{1}{4} & 0 & 0 & -\frac{1}{4}i & 0 & 0 & -\frac{1}{4}i \\
0 & \frac{1}{4} & \frac{1}{4}i & 0 & -\frac{1}{4} & -\frac{1}{4}i & 0
    \end{pmatrix}\right]\right).\]
    We remark that $q_{|\pi}$ has signature $(1,1)$ and $q([v])<0$. Therefore, $\mathcal{O}(\gamma_{8})$ lies entirely in the $G_{2(2)}-$invariant subset $\mathcal{O}^{(+,-)}_{-}$, where
    \[\mathcal{O}^{(+,-)}_{-}=\{([v],\pi)|\ [v]\subset \pi,\ \operatorname{sign}(q)=(1,1),\ q([v])<0\}.\]
            \item[$(\gamma_{9})$] 
            \[\gamma_{9}B=\left(\left[\begin{pmatrix}
        0\\
        0\\
        -i\\
        1\\
        0\\
        -i\\
        1\\
    \end{pmatrix}\right],\left[\begin{pmatrix}
        0 & 0 & -\frac{1}{2} & -\frac{1}{2}i & 0 & -\frac{1}{2} & -\frac{1}{2}i \\
0 & 0 & -\frac{1}{4}i & \frac{1}{4} & 0 & -\frac{1}{4}i & \frac{1}{4} \\
\frac{1}{4} & \frac{1}{4}i & 0 & 0 & \frac{1}{4}i & 0 & 0 \\
\frac{1}{4}i & -\frac{1}{4} & 0 & 0 & -\frac{1}{4} & 0 & 0 \\
0 & 0 & -\frac{1}{4}i & \frac{1}{4} & 0 & -\frac{1}{4}i & \frac{1}{4} \\
\frac{1}{4} & \frac{1}{4}i & 0 & 0 & \frac{1}{4}i & 0 & 0 \\
\frac{1}{4}i & -\frac{1}{4} & 0 & 0 & -\frac{1}{4} & 0 & 0
    \end{pmatrix}\right]\right).\]
            We remark that $q_{|\pi}$ has signature $(0,2)$. Therefore, $\mathcal{O}(\gamma_{9})$ lies entirely in the $G_{2(2)}-$invariant subset $\mathcal{O}^{(-,-)}$, where
       \[\mathcal{O}^{(-,-)}=\{([v],\pi)|\ [v]\subset \pi,\ \operatorname{sign}(q_{|\pi})=(0,2)\}.\]
                \item[$(\gamma_{10})$] 
            \[\gamma_{10}B=\left(\left[\begin{pmatrix}
        0\\
        i\\
        0\\
        -1\\
        -i\\
        0\\
        1\\
    \end{pmatrix}\right],\left[\begin{pmatrix}
        0 & 0 & 0 & 0 & 0 & 0 & 0 \\
0 & -\frac{1}{2}i & 0 & \frac{1}{2} & 0 & 0 & 0 \\
0 & 0 & 0 & 0 & 0 & 0 & 0 \\
0 & \frac{1}{2} & 0 & \frac{1}{2}i & 0 & 0 & 0 \\
0 & 0 & 0 & 0 & \frac{1}{2}i & 0 & -\frac{1}{2} \\
0 & 0 & 0 & 0 & 0 & 0 & 0 \\
0 & 0 & 0 & 0 & -\frac{1}{2} & 0 & -\frac{1}{2}i
    \end{pmatrix}\right]\right).\]
    We remark that $q_{|\pi}$ has signature $(1,1)$ and that $q([v])>0$. Therefore, $\mathcal{O}(\gamma_{10})$ lies entirely in the $G_{2(2)}-$invariant subset $\mathcal{O}^{(+,-)}_{+}$, where
    \[\mathcal{O}^{(+,-)}_{+}=\{([v],\pi)|\ [v]\subset \pi,\ \operatorname{sign}(q_{|\pi})=(1,1),\ q([v])>0\}.\]
\end{itemize}
Let us summarize our findings in the following:
\begin{theorem}
    $G_{2(2)}$ acts on $G_{2}/B$ with $10$ orbits. Their dimensions, defining equations and topological features are illustrated in the following table:
    \end{theorem}
\begin{table}[h]
\centering
\small
\renewcommand{\arraystretch}{1.15}
\begin{tabular}{c|c|l|c}
$G_{2(2)}$-orbit & $\dim_{\mathbb R}$ & defining conditions & topology\\
\hline
$\mathcal{O}^{(+,-)}_{+}$ & $12$ & $\operatorname{sign}(q_{|\pi})=(1,1)$, $q([v])>0$ & open\\
$\mathcal{O}^{(+,-)}_{-}$ & $12$ & $\operatorname{sign}(q_{|\pi})=(1,1)$, $q([v])<0$ & open\\
$\mathcal{O}^{(-,-)}$ & $12$ & $\operatorname{sign}(q_{|\pi})=(0,2)$ & open\\
$\mathcal{O}^{(+,-)}_{11}$ & $11$ & $\operatorname{sign}(q_{|\pi})=(1,1)$, $q([v])=0$, $\omega(v)\neq0$ & intermediate\\
$\mathcal{O}^{(0,-)}_{11}$ & $11$ & $\operatorname{sign}(q_{|\pi})=(0,1)$, $q([v])<0$ & intermediate\\
$\mathcal{O}^{(0,-)}_{9}$ & $9$ & $\operatorname{sign}(q_{|\pi})=(0,1)$, $q([v])=0$ & intermediate\\
$\mathcal{O}^{(0,0)}_{9}$ & $9$ & $q_{|\pi}=0$, $\omega(v)\neq0$ & intermediate\\
$\mathcal{O}_{7}$ & $7$ & $q_{|\pi}=0$, $\pi\neq\bar\pi$, $\bar v\wedge v=0$ & intermediate\\
$\mathcal{O}^{'}_{7}$ & $7$ & $q_{|\pi}=0$, $\pi=\bar\pi$, $\bar v\wedge v\neq0$ & intermediate\\
$\mathcal{O}_{6}$ & $6$ & $q_{|\pi}=0$, $\pi=\bar\pi$, $\bar v\wedge v=0$ & closed
\end{tabular}
\bigskip
\caption{$G_{2(2)}$-orbits on $G_2/B$.}
\label{tab:example}
\end{table}
\subsection{Projecting the orbits of $G_{2}/B$ to the generalized flag $G_{2}/Q_{\Pi_0}$ }\label{subsec:projecting-orbits}
For $i=1,2$, let us consider the natural $G_{2(2)}-$equivariant projection map with fiber $\mathbb{CP}^1$
\begin{align*}
\pi_i:\ G_2/B &\longrightarrow G_2/Q_{\alpha_i}\\
gB &\longmapsto gQ_{\alpha_i}.
\end{align*}
Because this map is equivariant, it maps $G_{2(2)}-$orbits of $G_{2}/B$ to $G_{2(2)}-$orbits of $G_{2}/Q_{\alpha_{i}}.$ Knowing how the orbits on the complete flag map down to the partial flags, of real dimension $10$, will allow us to completely reconstruct the orbit structure on the partial flags. Indeed, suppose we have an orbit $G_{2(2)}\cdot gQ_{\alpha_{i}}\subset G_{2}/Q_{\alpha_{i}}$. Clearly, $gQ_{\alpha_{i}}=\pi_{i}(gB)$ and, by the $G_{2(2)}$-equivariance of $\pi_{i}$, we have that $\pi_{i}(G_{2(2)}\cdot gB)=G_{2(2)}\cdot gQ_{\alpha_{i}}$. This shows that every $G_{2(2)}-$orbit on the generalized flag $G_{2}/Q_{\alpha_{i}}$ must be a projection of a $G_{2(2)}-$orbit on the complete flag. Therefore, to compute the number of $G_{2(2)}$-orbits on the partial flag, it is enough to find which orbits $\mathcal{O}(\gamma_{i})$ project to distinct orbits on $G_{2}/Q_{\alpha_{i}}$.\par
On the complete flag manifold, we have identified $10$-distinct orbits $\mathcal{O}(\gamma_{i})=G_{2(2)}\cdot\gamma_{i}B$ for $i=1,\dots,10,$ where the explicit form of $\gamma_{i}$ is indicated in section \ref{sec:explicit-G2-B}. 
\begin{remark}\label{rem:computing-dimension-of-projections}
In addition to projecting each orbit and identifying which ones collapse into the same orbit on the partial flags, we compute the dimension of each projected orbit. This can be done by the following algorithmic procedure, which can be implemented in a symbolic computation software such as SageMath.
\begin{itemize}
    \item [$(1)$] Let $M$ be the $G_{2(2)}$-orbit through $x=\gamma Q_{\alpha_{j}}$. Its isotropy subgroup is $I_{x}=G_{2(2)}\cap \gamma Q_{\alpha_j}\gamma^{-1}$, so that $M\cong G_{2(2)}/I_{x}$ and $\dim_{\mathbb{R}}M=14-\dim_{\mathbb{R}}(\mathfrak{g}_{2(2)}\cap \gamma \mathfrak{q}_{\alpha_{j}}\gamma^{-1}).$
    \item[$(2)$] Since the explicit expressions of $\gamma$ and of $\mathfrak{q}_{\alpha_j}$ are known, see \eqref{eq: parabolics}, we compute the form of a generic matrix of $\gamma \mathfrak{q}_{\alpha_{j}}\gamma^{-1}$, whose entries are functions of the complex parameters $x_k$, $y$, $a_{kj}$ of \eqref{eq: parabolics}, and of $z$ in the case of $\mathfrak q_{\alpha_2}$.
    \item[$(3)$] Since $\mathfrak g_{2(2)}$ consists of the real matrices of $\mathfrak g_2$, the intersection $\mathfrak{g}_{2(2)}\cap \gamma \mathfrak{q}_{\alpha_{j}}\gamma^{-1}$ is obtained by imposing that all entries of the matrices of Step $(2)$ are real. This is a linear system in the real and imaginary parts of the complex parameters, which can be solved with SageMath. If $f$ is the number of free parameters of its solution, then $\dim_{\mathbb R}M=14-f$, and writing the complex parameters in terms of the free parameters gives the form of a generic matrix of the isotropy algebra.
\end{itemize}
\end{remark}
We start with $G_{2}/Q_{\alpha_{1}}$.
\begin{itemize}
    \item[$\boxed{\mathcal{O}(\gamma_{1})}$] 
    We know that $\gamma_{1}=I_{7}.$ As a consequence, $\pi_1(\gamma_{1}B)=\gamma_{1}Q_{\alpha_{1}}=\left[\begin{pmatrix}
        0\\
        0\\
        e_{3}\\
    \end{pmatrix}\right].$
    Because $\pi_{1}$ is continuous, $\pi_1(\mathcal{O}(\gamma_{1}))$ is the only compact $G_{2(2)}-$orbit of $G_{2}/Q_{\alpha_{1}}$ with real dimension $5$. 
    \item[$\boxed{\mathcal{O}(\gamma_{2})}$]
    \[\pi_{1}(\gamma_{2}B)=\pi_{1}\left(
\begin{pmatrix}
1 & 0 & 0 & 0 & 0 & 0 & 0 \\
0 & \frac{1}{2}\sqrt{2} & \frac{1}{2}i\sqrt{2} & 0 & 0 & 0 & 0 \\
0 & \frac{1}{2}i\sqrt{2} & \frac{1}{2}\sqrt{2} & 0 & 0 & 0 & 0 \\
0 & 0 & 0 & 1 & 0 & 0 & 0 \\
0 & 0 & 0 & 0 & \frac{1}{2}\sqrt{2} & -\frac{1}{2}i\sqrt{2} & 0 \\
0 & 0 & 0 & 0 & -\frac{1}{2}i\sqrt{2} & \frac{1}{2}\sqrt{2} & 0 \\
0 & 0 & 0 & 0 & 0 & 0 & 1
    \end{pmatrix}B\right)=\left[\begin{pmatrix}
        0\\
        0\\
        e_{3}
    \end{pmatrix}\right].\]
    This means that $G_{2(2)}\cdot\gamma_{2}Q_{\alpha_{1}}$ coincides with $\mathcal{O}(\gamma_{1}).$
    \item[$\boxed{\mathcal{O}(\gamma_{3})}$] 
    We have that \[\pi_{1}(\gamma_{3}B)=\pi_{1}\left(
\begin{pmatrix}
1 & 0 & 0 & 0 & 0 & 0 & 0 \\
0 & 1 & 0 & 0 & 0 & 0 & 0 \\
0 & 0 & \frac{1}{2} \sqrt{2} & \frac{1}{2} i \sqrt{2} & 0 & 0 & 0 \\
0 & 0 & \frac{1}{2} i \sqrt{2} & \frac{1}{2} \sqrt{2} & 0 & 0 & 0 \\
0 & 0 & 0 & 0 & 1 & 0 & 0 \\
0 & 0 & 0 & 0 & 0 & \frac{1}{2} \sqrt{2} & -\frac{1}{2} i \sqrt{2} \\
0 & 0 & 0 & 0 & 0 & -\frac{1}{2} i \sqrt{2} & \frac{1}{2} \sqrt{2}
    \end{pmatrix}B\right)=\left[\begin{pmatrix}
        0\\
        0\\
        0\\
        0\\
        0\\
        -i\\
        1\\
    \end{pmatrix}\right].\]
   Following Remark \ref{rem:computing-dimension-of-projections}, to find the dimension of this orbit, we compute that
   \[\mathfrak{g}_{2(2)}\cap\gamma_{3}\mathfrak{q}_{\alpha_1}\gamma_3^{-1}=\left\{\begin{pmatrix}
   0 & 0 & 0 & 0 & -2 \, r_{1} & 0 & 0 \\
r_{1} & r_{4} & 0 & \sqrt{2} r_{5} & 0 & 0 & 0 \\
0 & 0 & -\frac{1}{2} \, r_{4} & r_{2} & 0 & 0 & 0 \\
0 & 0 & -r_{2} & -\frac{1}{2} \, r_{4} & 0 & 0 & 0 \\
0 & 0 & 0 & 0 & -r_{4} & 0 & 0 \\
0 & 0 & 0 & -r_{1} & 0 & \frac{1}{2} \, r_{4} & r_{2} \\
0 & 0 & r_{1} & 0 & -\sqrt{2} r_{5} & -r_{2} & \frac{1}{2} \, r_{4}\end{pmatrix}|\ r_{j}\in\mathbb{R}\right\}.\]
This shows that $f=5$. As a result, this orbit has dimension $14-f=9$.
\item[$\boxed{\mathcal{O}(\gamma_{4})}$]
    \[\pi_{1}(\gamma_{4}B)=\pi_{1}\left(
\begin{pmatrix}
1 & 0 & 0 & 0 & 0 & 0 & 0 \\
0 & \frac{1}{2}\sqrt{2} & 0 & \frac{1}{2}i\sqrt{2} & 0 & 0 & 0 \\
0 & 0 & 1 & 0 & 0 & 0 & 0 \\
0 & \frac{1}{2}i\sqrt{2} & 0 & \frac{1}{2}\sqrt{2} & 0 & 0 & 0 \\
0 & 0 & 0 & 0 & \frac{1}{2}\sqrt{2} & 0 & -\frac{1}{2}i\sqrt{2} \\
0 & 0 & 0 & 0 & 0 & 1 & 0 \\
0 & 0 & 0 & 0 & -\frac{1}{2}i\sqrt{2} & 0 & \frac{1}{2}\sqrt{2}
    \end{pmatrix}B\right)=\left[\begin{pmatrix}
        0\\
        0\\
        0\\
        0\\
        -i\\
        0\\
        1\\
    \end{pmatrix}\right].\]
     Its dimension is $9$. We will now show that this orbit is equal to the orbit through $\gamma_{3}Q_{\alpha_{1}}.$ Consider the element $g\in G_{2(2)}$ given by
     \[g=\operatorname{exp}\left(M(\frac{\pi}{2}(E_{21}-E_{12}),0,0)\right)=\left(\begin{array}{rrrrrrr}
1 & 0 & 0 & 0 & 0 & 0 & 0 \\
0 & 0 & -1 & 0 & 0 & 0 & 0 \\
0 & 1 & 0 & 0 & 0 & 0 & 0 \\
0 & 0 & 0 & 1 & 0 & 0 & 0 \\
0 & 0 & 0 & 0 & 0 & -1 & 0 \\
0 & 0 & 0 & 0 & 1 & 0 & 0 \\
0 & 0 & 0 & 0 & 0 & 0 & 1
\end{array}\right).\]
     Applying $g$ to $[\left(0,\ 0,\ 0,\ 0,\ -i,\ 0,\ 1\right)^{t}]$ yields $[\left(0,\ 0,\ 0,\ 0,\ 0,\ -i,\ 1\right)^{t}]$,
     which is the representative of the orbit through $\gamma_{3}Q_{\alpha_{1}}.$ We thus conclude that the orbits are equal.
     \item[$\boxed{\mathcal{O}(\gamma_{5})}$]
\[\pi_{1}(\gamma_{5}B)=\pi_{1}\left(
\begin{pmatrix}
0 & 0 & 0 & i & 0 & 0 & -i \\
0 & \frac{1}{2}\sqrt{2} & 0 & 0 & 0 & \frac{1}{2}i\sqrt{2} & 0 \\
0 & 0 & \frac{1}{2}\sqrt{2} & 0 & -\frac{1}{2}i\sqrt{2} & 0 & 0 \\
\frac{1}{2}i & 0 & 0 & \frac{1}{2} & 0 & 0 & \frac{1}{2} \\
0 & 0 & -\frac{1}{2}i\sqrt{2} & 0 & \frac{1}{2}\sqrt{2} & 0 & 0 \\
0 & \frac{1}{2}i\sqrt{2} & 0 & 0 & 0 & \frac{1}{2}\sqrt{2} & 0 \\
-\frac{1}{2}i & 0 & 0 & \frac{1}{2} & 0 & 0 & \frac{1}{2}
    \end{pmatrix}B\right)=\left[\begin{pmatrix}
        -i\\
        0\\
        0\\
        \frac{1}{2}\\
        0\\
        0\\
        \frac{1}{2}\\
    \end{pmatrix}\right].\]
    Its dimension is $10$. This means that $\mathcal{O}(\gamma_{5})$ is an open orbit of $G_{2}/Q_{\alpha_{1}}.$
         \item[$\boxed{\mathcal{O}(\gamma_{6})}$]
\[\pi_{1}(\gamma_{6}B)=\pi_{1}\left(
\begin{pmatrix}
0 & i & 0 & 0 & -i & 0 & 0 \\
\frac{1}{2}i & \frac{1}{2} & 0 & 0 & \frac{1}{2} & 0 & 0 \\
0 & 0 & \frac{1}{2}\sqrt{2} & 0 & 0 & 0 & \frac{1}{2}i\sqrt{2} \\
0 & 0 & 0 & \frac{1}{2}\sqrt{2} & 0 & -\frac{1}{2}i\sqrt{2} & 0 \\
-\frac{1}{2}i & \frac{1}{2} & 0 & 0 & \frac{1}{2} & 0 & 0 \\
0 & 0 & 0 & -\frac{1}{2}i\sqrt{2} & 0 & \frac{1}{2}\sqrt{2} & 0 \\
0 & 0 & \frac{1}{2}i\sqrt{2} & 0 & 0 & 0 & \frac{1}{2}\sqrt{2}
    \end{pmatrix}B\right)=\left[\begin{pmatrix}
        0\\
        0\\
        i\\
        0\\
        0\\
        0\\
        1\\
    \end{pmatrix}\right].\]
    Its dimension is $7$. 
\item[$\boxed{\mathcal{O}(\gamma_{7})}$] 
\[\pi_{1}(\gamma_{7}B)=\pi_{1}\left(
\begin{pmatrix}
0 & 0 & i & 0 & 0 & -i & 0 \\
0 & \frac{1}{2}\sqrt{2} & 0 & 0 & 0 & 0 & -\frac{1}{2}i\sqrt{2} \\
\frac{1}{2}i & 0 & \frac{1}{2} & 0 & 0 & \frac{1}{2} & 0 \\
0 & 0 & 0 & \frac{1}{2}\sqrt{2} & \frac{1}{2}i\sqrt{2} & 0 & 0 \\
0 & 0 & 0 & \frac{1}{2}i\sqrt{2} & \frac{1}{2}\sqrt{2} & 0 & 0 \\
-\frac{1}{2}i & 0 & \frac{1}{2} & 0 & 0 & \frac{1}{2} & 0 \\
0 & -\frac{1}{2}i\sqrt{2} & 0 & 0 & 0 & 0 & \frac{1}{2}\sqrt{2}
    \end{pmatrix}B\right)=\left[\begin{pmatrix}
        0\\
        -i\\
        0\\
        0\\
        0\\
        0\\
        1\\
    \end{pmatrix}\right].\]
    Its dimension is $7$. We now prove that this orbit is the same as the orbit through $\gamma_{6}Q_{\alpha_{1}}.$ Indeed, consider the element $h$ of $G_{2(2)}$ given by
    
    \begin{equation}\label{eq:alpha1-conjugating-element}
h=g^{-1}=
\begin{pmatrix}
1&0&0&0&0&0&0\\
0&0&1&0&0&0&0\\
0&-1&0&0&0&0&0\\
0&0&0&1&0&0&0\\
0&0&0&0&0&1&0\\
0&0&0&0&-1&0&0\\
0&0&0&0&0&0&1
\end{pmatrix}.
\end{equation}
    
    Applying this, we obtain that $h\cdot [\left(0,\ -i,\ 0,\ 0,\ 0,\ 0,\ 1 \right)^{t}]=[\left(0,\ 0,\ i,\ 0,\ 0,\ 0,\ 1\right)^{t}]$,
    which is the same base point as the $G_{2(2)}-$orbit through $\gamma_{6}Q_{\alpha_{1}}.$ This shows that the two orbits are equal.
        \item[$\boxed{\mathcal{O}(\gamma_{8})}$]
\[\pi_{1}(\gamma_{8}B)=\pi_{1}\left(
\begin{pmatrix}
0 & 0 & 0 & -i & 0 & 0 & i \\
0 & \frac{1}{2} & \frac{1}{2}i & 0 & -\frac{1}{2} & -\frac{1}{2}i & 0 \\
0 & \frac{1}{2}i & \frac{1}{2} & 0 & \frac{1}{2}i & \frac{1}{2} & 0 \\
-\frac{1}{2}i & 0 & 0 & \frac{1}{2} & 0 & 0 & \frac{1}{2} \\
0 & -\frac{1}{2} & \frac{1}{2}i & 0 & \frac{1}{2} & -\frac{1}{2}i & 0 \\
0 & -\frac{1}{2}i & \frac{1}{2} & 0 & -\frac{1}{2}i & \frac{1}{2} & 0 \\
\frac{1}{2}i & 0 & 0 & \frac{1}{2} & 0 & 0 & \frac{1}{2}
    \end{pmatrix}B\right)=\left[\begin{pmatrix}
        2i\\
        0\\
        0\\
        1\\
        0\\
        0\\
        1\\
    \end{pmatrix}\right].\]
 It has dimension $10$ and it coincides with the orbit though the point $\gamma_{5}Q_{\alpha_{1}}.$ Indeed, consider the element 
 \[k=\operatorname{exp}\left(M(\frac{\pi}{2}(E_{31}-E_{13}),\frac{\pi}{2}e_2,\frac{\pi}{2}e_2)\right)=
     \begin{pmatrix}
 -1 & 0 &  0 & 0 & 0 &  0 & 0 \\
  0 & 0 &  0 & 0 & 1 &  0 & 0 \\
  0 & 0 &  0 & 0 & 0 & -1 & 0 \\
  0 & 0 &  0 & 0 & 0 &  0 & 1 \\
  0 & 1 &  0 & 0 & 0 &  0 & 0 \\
  0 & 0 & -1 & 0 & 0 &  0 & 0 \\
  0 & 0 &  0 & 1 & 0 &  0 & 0 
\end{pmatrix}.\]
It is easy to check that $k\cdot [(-2i,\ 0,\ 0,\ 1,\ 0,\ 0,\ 1)^{t}]=[(2i,\ 0,\ 0,\ 1,\ 0,\ 0,\ 1)^{t}].$ This means that $G_{2(2)}\cdot \gamma_{5}Q_{\alpha_{1}}=G_{2(2)}\cdot\gamma_{8}Q_{\alpha_{1}}.$
    \item[$\boxed{\mathcal{O}(\gamma_{9})}$] 
 \[\pi_{1}(\gamma_{9}B)=\pi_{1}\left(
\begin{pmatrix}
0 & -i & 0 & 0 & i & 0 & 0 \\
-\frac{1}{2}i & \frac{1}{2} & 0 & 0 & \frac{1}{2} & 0 & 0 \\
0 & 0 & \frac{1}{2} & \frac{1}{2}i & 0 & -\frac{1}{2} & -\frac{1}{2}i \\
0 & 0 & \frac{1}{2}i & \frac{1}{2} & 0 & \frac{1}{2}i & \frac{1}{2} \\
\frac{1}{2}i & \frac{1}{2} & 0 & 0 & \frac{1}{2} & 0 & 0 \\
0 & 0 & -\frac{1}{2} & \frac{1}{2}i & 0 & \frac{1}{2} & -\frac{1}{2}i \\
0 & 0 & -\frac{1}{2}i & \frac{1}{2} & 0 & -\frac{1}{2}i & \frac{1}{2}
    \end{pmatrix}B\right)=\left[\begin{pmatrix}
        0\\
        0\\
        -i\\
        1\\
        0\\
        -i\\
        1\\
    \end{pmatrix}\right].\]
    It has dimension $10$ and it coincides with the orbit though the point $\gamma_{5}Q_{\alpha_{1}}$. It is enough to consider the element $m$ of $G_{2(2)}$ given by
    \[m=\operatorname{exp}\left(M(\frac{\pi}{4}(E_{31}-E_{13}),-\frac{\pi}{4}e_2,-\frac{\pi}{4}e_2)\right)
=\left(\begin{array}{rrrrrrr}
0 & 0 & 1 & 0 & 0 & 1 & 0 \\
0 & \frac{1}{2} & 0 & \frac{1}{2} & \frac{1}{2} & 0 & -\frac{1}{2} \\
-\frac{1}{2} & 0 & \frac{1}{2} & 0 & 0 & -\frac{1}{2} & 0 \\
0 & -\frac{1}{2} & 0 & \frac{1}{2} & \frac{1}{2} & 0 & \frac{1}{2} \\
0 & \frac{1}{2} & 0 & -\frac{1}{2} & \frac{1}{2} & 0 & \frac{1}{2} \\
-\frac{1}{2} & 0 & -\frac{1}{2} & 0 & 0 & \frac{1}{2} & 0 \\
0 & \frac{1}{2} & 0 & \frac{1}{2} & -\frac{1}{2} & 0 & \frac{1}{2}
\end{array}\right).\]
We have that $m\cdot[(2i,\ 0,\ 0,\ 1,\ 0,\ 0,\ 1)^{t}]=[(0,\ 0,\ -i,\ 1,\ 0,\ -i,\ 1)^{t}]$. 
This shows that the two orbits are equal.
        \item[$\boxed{\mathcal{O}(\gamma_{10})}$]
\[\pi_{1}(\gamma_{10}B)=\pi_{1}\left(
\begin{pmatrix}
0 & 0 & -i & 0 & 0 & i & 0 \\
0 & \frac{1}{2} & 0 & \frac{1}{2}i & \frac{1}{2} & 0 & \frac{1}{2}i \\
-\frac{1}{2}i & 0 & \frac{1}{2} & 0 & 0 & \frac{1}{2} & 0 \\
0 & \frac{1}{2}i & 0 & \frac{1}{2} & -\frac{1}{2}i & 0 & -\frac{1}{2} \\
0 & \frac{1}{2} & 0 & -\frac{1}{2}i & \frac{1}{2} & 0 & -\frac{1}{2}i \\
\frac{1}{2}i & 0 & \frac{1}{2} & 0 & 0 & \frac{1}{2} & 0 \\
0 & \frac{1}{2}i & 0 & -\frac{1}{2} & -\frac{1}{2}i & 0 & \frac{1}{2}
    \end{pmatrix}B\right)=\left[\begin{pmatrix}
        0\\
        i\\
        0\\
        -1\\
        -i\\
        0\\
        1\\
    \end{pmatrix}\right].\]
    It has dimension $10$ and is open. Besides, it is clear from our discussion in Section~\ref{sec: equation and closure relations for G2/B} that this orbit cannot be equivalent to $G_{2(2)}\cdot \gamma_{5}Q_{\alpha_{1}}$ as $q(\gamma_{5}Q_{\alpha_{1}})<0$, while $q(\gamma_{10}Q_{\alpha_{1}})>0.$
\end{itemize}
\par
We now move to the case $G_{2}/Q_{\alpha_{2}}$.
\begin{itemize}
    \item[$\boxed{\mathcal{O}(\gamma_{1})}$]
     \[\pi(\gamma_{1}B)=\gamma_{1}Q_{\alpha_{2}}=\left[\begin{pmatrix}
       0 & 0 & 0 & 0 & 0 & 0 & 0 \\
0 & 0 & 0 & 1 & 0 & 0 & 0 \\
0 & 0 & 0 & 0 & 0 & 0 & 0 \\
0 & 0 & 0 & 0 & 0 & 0 & 0 \\
0 & 0 & 0 & 0 & 0 & 0 & 0 \\
0 & 0 & 0 & 0 & 0 & 0 & 0 \\
0 & 0 & 0 & 0 & -1 & 0 & 0
    \end{pmatrix}\right].\]
    Because $\pi_{1}$ is continuous, $\pi(\mathcal{O}(\gamma_{1}))$ is the only compact $G_{2(2)}-$orbit of $G_{2}/Q_{\alpha_{2}}$ with real dimension $5$. 
    \item[$\boxed{\mathcal{O}(\gamma_{2})}$]
\[\pi_{2}(\gamma_{2}B)=\left[
    \begin{pmatrix}
      0 & 0 & 0 & 0 & 0 & 0 & 0 \\
0 & 0 & 0 & \frac{1}{2}\sqrt{2} & 0 & 0 & 0 \\
0 & 0 & 0 & \frac{1}{2}i\sqrt{2} & 0 & 0 & 0 \\
0 & 0 & 0 & 0 & 0 & 0 & 0 \\
0 & 0 & 0 & 0 & 0 & 0 & 0 \\
0 & 0 & 0 & 0 & 0 & 0 & 0 \\
0 & 0 & 0 & 0 & -\frac{1}{2}\sqrt{2} & -\frac{1}{2}i\sqrt{2} & 0
    \end{pmatrix}\right].\]
    The dimension of the orbit $G_{2(2)}\cdot \gamma_{2}Q_{\alpha_{2}}$ is $7$.
    \item[$\boxed{\mathcal{O}(\gamma_{3})}$] 
   \[\pi_{2}(\gamma_{3}B)=\left[\begin{pmatrix}
       0 & 0 & 0 & 0 & 0 & 0 & 0 \\
0 & 0 & -\frac{1}{2} i \sqrt{2} & \frac{1}{2} \sqrt{2} & 0 & 0 & 0 \\
0 & 0 & 0 & 0 & 0 & 0 & 0 \\
0 & 0 & 0 & 0 & 0 & 0 & 0 \\
0 & 0 & 0 & 0 & 0 & 0 & 0 \\
0 & 0 & 0 & 0 & \frac{1}{2} i \sqrt{2} & 0 & 0 \\
0 & 0 & 0 & 0 & -\frac{1}{2} \sqrt{2} & 0 & 0
    \end{pmatrix}\right].\]
   The orbit $G_{2(2)}\cdot \gamma_{3}Q_{\alpha_{2}}$ is $7$-dimensional. We will now prove that this orbit is the same as the $G_{2(2)}$- orbit through the point $\gamma_{2}Q_{\alpha_{2}}.$ Indeed, consider the element $g$ of $G_{2(2)}$ given by
   \begin{equation}\label{eq:alpha2-conjugating-element}
       g=\operatorname{exp}\left(M(0,-\frac{\pi}{2}e_{2},-\frac{\pi}{2}e_{2})\right)=\begin{pmatrix}
       -1 & 0 & 0 & 0 & 0 & 0 & 0 \\
0 & 0 & 0 & 0 & 0 & 0 & -1 \\
0 & 0 & 0 & 0 & 0 & -1 & 0 \\
0 & 0 & 0 & 0 & 1 & 0 & 0 \\
0 & 0 & 0 & -1 & 0 & 0 & 0 \\
0 & 0 & -1 & 0 & 0 & 0 & 0 \\
0 & 1 & 0 & 0 & 0 & 0 & 0
   \end{pmatrix}.\end{equation}

The plane associated with $\gamma_2Q_{\alpha_2}$ is
\begin{equation}\label{eq:Xalpha2-gamma2-representative}
\pi_{\gamma_2}
=
\operatorname{span}_{\mathbb C}
\bigl\{
(0,1,i,0,0,0,0)^t,
(0,0,0,0,0,0,-1)^t
\bigr\}.
\end{equation}
    We can thus conclude that $g\cdot \gamma_{2}Q_{\alpha_{2}}=\gamma_{3}Q_{\alpha_{2}}$ and so the two orbits coincide.
       \item[$\boxed{\mathcal{O}(\gamma_{4})}$] 
    \[\pi_{2}(\gamma_{4}B)=\left[\begin{pmatrix}
0 & 0 & 0 & 0 & 0 & 0 & 0 \\
0 & -\frac{1}{2}i & 0 & \frac{1}{2} & 0 & 0 & 0 \\
0 & 0 & 0 & 0 & 0 & 0 & 0 \\
0 & \frac{1}{2} & 0 & \frac{1}{2}i & 0 & 0 & 0 \\
0 & 0 & 0 & 0 & \frac{1}{2}i & 0 & -\frac{1}{2} \\
0 & 0 & 0 & 0 & 0 & 0 & 0 \\
0 & 0 & 0 & 0 & -\frac{1}{2} & 0 & -\frac{1}{2}i
\end{pmatrix}\right].\]
  It is $10$-dimensional and thus open. 
   \item[$\boxed{\mathcal{O}(\gamma_{5})}$] 
    \[\pi_{2}(\gamma_{5}B)=\left[\begin{pmatrix}
0 & 0 & -\frac{1}{2}\sqrt{2} & 0 & \frac{1}{2}i\sqrt{2} & 0 & 0 \\
-\frac{1}{4}i\sqrt{2} & 0 & 0 & \frac{1}{4}\sqrt{2} & 0 & 0 & \frac{1}{4}\sqrt{2} \\
0 & 0 & 0 & 0 & 0 & 0 & 0 \\
0 & 0 & -\frac{1}{4}i\sqrt{2} & 0 & -\frac{1}{4}\sqrt{2} & 0 & 0 \\
0 & 0 & 0 & 0 & 0 & 0 & 0 \\
\frac{1}{4}\sqrt{2} & 0 & 0 & \frac{1}{4}i\sqrt{2} & 0 & 0 & \frac{1}{4}i\sqrt{2} \\
0 & 0 & -\frac{1}{4}i\sqrt{2} & 0 & -\frac{1}{4}\sqrt{2} & 0 & 0
\end{pmatrix}\right].\]
One may compute that the dimension of this orbit is given by $9$.
\item[$\boxed{\mathcal{O}(\gamma_{6})}$]    
\[\pi_{2}(\gamma_{6}B)=\left[\begin{pmatrix}
0 & 0 & 0 & \frac{1}{2}i\sqrt{2} & 0 & -\frac{1}{2}\sqrt{2} & 0 \\
0 & 0 & 0 & \frac{1}{4}\sqrt{2} & 0 & \frac{1}{4}i\sqrt{2} & 0 \\
\frac{1}{4}\sqrt{2} & -\frac{1}{4}i\sqrt{2} & 0 & 0 & -\frac{1}{4}i\sqrt{2} & 0 & 0 \\
0 & 0 & 0 & 0 & 0 & 0 & 0 \\
0 & 0 & 0 & \frac{1}{4}\sqrt{2} & 0 & \frac{1}{4}i\sqrt{2} & 0 \\
0 & 0 & 0 & 0 & 0 & 0 & 0 \\
-\frac{1}{4}i\sqrt{2} & -\frac{1}{4}\sqrt{2} & 0 & 0 & -\frac{1}{4}\sqrt{2} & 0 & 0
\end{pmatrix}\right].\]
One may compute that the dimension of this orbit is given by $9$.\\
We will now show that the two $9$-dimensional orbits are the same orbit. Indeed, consider the element $g$ of $G_{2,(2)}$ in \eqref{eq:alpha2-conjugating-element}.
The plane $\pi_{5}$ associated with the point
$\gamma_{5}Q_{\alpha_{2}}$ is spanned by the vectors
$v_{1}=(2,0,0,i,0,0,i)^{t}$ and
$v_{2}=(0,1,0,0,0,i,0)^{t}$.
The plane $g\cdot\pi_{5}$ is spanned by the vectors
$gv_{1}$ and $gv_{2}$. One easily computes that
$gv_{1}=(-2,i,0,0,i,0,0)^{t}$ and
$gv_{2}=(0,0,-i,0,0,0,-1)^{t}$,
which are exactly the generators of the plane $\pi_{6}$ associated with the point $\gamma_{6}Q_{\alpha_{2}}.$ This shows that $g\gamma_{5}Q_{\alpha_{2}}=\gamma_{6}Q_{\alpha_{2}}$ and thus the two orbits coincide.
\item[$\boxed{\mathcal{O}(\gamma_{7})}$]     
\[\pi_{2}(\gamma_{7}B)=\left[\begin{pmatrix}
0 & 0 & 0 & 0 & 0 & 0 & 0 \\
0 & 0 & 0 & 1 & 0 & 0 & 0 \\
0 & 0 & 0 & 0 & 0 & 0 & 0 \\
0 & 0 & 0 & 0 & 0 & 0 & 0 \\
0 & 0 & 0 & 0 & 0 & 0 & 0 \\
0 & 0 & 0 & 0 & 0 & 0 & 0 \\
0 & 0 & 0 & 0 & -1 & 0 & 0\end{pmatrix}\right].\]
We remark that $\pi_{2}(\gamma_{7}B)=\gamma_{1}Q_{\alpha_{2}}$ and, as a result, the two orbits $G_{2(2)}\cdot \gamma_{1}Q_{\alpha_{2}}$ and $G_{2(2)}\cdot \gamma_{7}Q_{\alpha_{2}}$ coincide. This shows that $\pi_{2}(\gamma_{7}B)$ is the unique $5$-dimensional closed orbit.
\item[$\boxed{\mathcal{O}(\gamma_{8})}$]     \[\pi_{2}(\gamma_{8}B)=\left[\begin{pmatrix}
0 & \frac{1}{2}i & -\frac{1}{2} & 0 & -\frac{1}{2}i & \frac{1}{2} & 0 \\
\frac{1}{4}i & 0 & 0 & \frac{1}{4} & 0 & 0 & \frac{1}{4} \\
-\frac{1}{4} & 0 & 0 & \frac{1}{4}i & 0 & 0 & \frac{1}{4}i \\
0 & \frac{1}{4} & \frac{1}{4}i & 0 & -\frac{1}{4} & -\frac{1}{4}i & 0 \\
-\frac{1}{4}i & 0 & 0 & -\frac{1}{4} & 0 & 0 & -\frac{1}{4} \\
\frac{1}{4} & 0 & 0 & -\frac{1}{4}i & 0 & 0 & -\frac{1}{4}i \\
0 & \frac{1}{4} & \frac{1}{4}i & 0 & -\frac{1}{4} & -\frac{1}{4}i & 0\end{pmatrix}\right].\]

 This orbit has dimension $10$ and it is open. It coincides with the orbit through $\gamma_{4}Q_{\alpha_{2}}$. 
 Indeed, consider the vectors 
\[
\begin{gathered}
v_{1}=(0,i,0,-1,i,0,-1)^{t},\quad v_{2}=(0,i,0,-1,-i,0,1)^{t},\\
v_{3}=(0,i,-1,0,-i,1,0)^{t},\quad v_{4}=(2i,0,0,1,0,0,1)^{t}.
\end{gathered}
\]
The vectors $v_{1}$ and $v_{2}$ span $\pi_4$, while
$v_{3}$ and $v_{4}$ span $\pi_8$.
Let
     \[ g'=\operatorname{exp}\left(M(\frac{\pi}{4}(E_{32}-E_{23}),-\frac{\pi}{4}e_1,-\frac{\pi}{4}e_{1})\right)=\begin{pmatrix}
         0 & 1 & 0 & 0 & 1 & 0 & 0 \\
-\frac{1}{2} & \frac{1}{2} & 0 & 0 & -\frac{1}{2} & 0 & 0 \\
0 & 0 & \frac{1}{2} & -\frac{1}{2} & 0 & \frac{1}{2} & \frac{1}{2} \\
0 & 0 & \frac{1}{2} & \frac{1}{2} & 0 & -\frac{1}{2} & \frac{1}{2} \\
-\frac{1}{2} & -\frac{1}{2} & 0 & 0 & \frac{1}{2} & 0 & 0 \\
0 & 0 & \frac{1}{2} & \frac{1}{2} & 0 & \frac{1}{2} & -\frac{1}{2} \\
0 & 0 & -\frac{1}{2} & \frac{1}{2} & 0 & \frac{1}{2} & \frac{1}{2}
     \end{pmatrix}.\]
     One easily shows that $g'v_{4}=-v_{1}$ and $g'v_{3}=v_{2}.$ 
 \item[$\boxed{\mathcal{O}(\gamma_{9})}$]  
    \[\pi_{2}(\gamma_{9}B)=\left[\begin{pmatrix}
0 & 0 & -\frac{1}{2} & -\frac{1}{2}i & 0 & -\frac{1}{2} & -\frac{1}{2}i \\
0 & 0 & -\frac{1}{4}i & \frac{1}{4} & 0 & -\frac{1}{4}i & \frac{1}{4} \\
\frac{1}{4} & \frac{1}{4}i & 0 & 0 & \frac{1}{4}i & 0 & 0 \\
\frac{1}{4}i & -\frac{1}{4} & 0 & 0 & -\frac{1}{4} & 0 & 0 \\
0 & 0 & -\frac{1}{4}i & \frac{1}{4} & 0 & -\frac{1}{4}i & \frac{1}{4} \\
\frac{1}{4} & \frac{1}{4}i & 0 & 0 & \frac{1}{4}i & 0 & 0 \\
\frac{1}{4}i & -\frac{1}{4} & 0 & 0 & -\frac{1}{4} & 0 & 0\end{pmatrix}\right].\]
This orbit is $10$-dimensional and distinct from the previous one. Indeed, as remarked in Section~\ref{sec: equation and closure relations for G2/B}, the signature of $q_{|\pi_9}$ is $(0,2)$, while the signature of $q_{|\pi_{4}}$ is $(1,1)$.
    \item[$\boxed{\mathcal{O}(\gamma_{10})}$] 
    \[\pi_{2}(\gamma_{10}B)=\left[\begin{pmatrix}
0 & 0 & 0 & 0 & 0 & 0 & 0 \\
0 & -\frac{1}{2}i & 0 & \frac{1}{2} & 0 & 0 & 0 \\
0 & 0 & 0 & 0 & 0 & 0 & 0 \\
0 & \frac{1}{2} & 0 & \frac{1}{2}i & 0 & 0 & 0 \\
0 & 0 & 0 & 0 & \frac{1}{2}i & 0 & -\frac{1}{2} \\
0 & 0 & 0 & 0 & 0 & 0 & 0 \\
0 & 0 & 0 & 0 & -\frac{1}{2} & 0 & -\frac{1}{2}i\end{pmatrix}\right].\]
This orbit is $10$-dimensional and an inspection of the columns spanning the plane $\pi_{10}$ immediately reveals that the orbit coincides with $G_{2(2)}\cdot \gamma_{4}.$
\end{itemize}
We have obtained the following 
\begin{lemma}\label{thm:orbit-count-maximal-flags}
For each $i=1,2$, the action of $G_{2(2)}$ on
$F_{\alpha_i}=G_2/Q_{\alpha_i}$ has exactly five orbits,
consisting of one closed orbit of real dimension $5$,
one orbit of each real dimension $7$ and $9$,
and two open orbits of real dimension $10$.
\end{lemma}
\section{Orbits on the generalized flags $G_{2}/Q_{\Pi_0}$}\label{counting_maximal flags}
We retain 
throughout the computation, realization and notation introduced in Section~\ref{sec: equation and closure relations for G2/B}.
In this realization, the first partial flag manifold is 
the five-dimensional projective quadric of $n$-isotropic lines, namely 
$F_{\alpha_1}=G_2/Q_{\alpha_1}\simeq\{[z]\in\mathbb P(\mathbb C^7)\mid n(z)=0\}$. 
 The second partial flag 
manifold, namely the adjoint variety, is realized as the variety of $\times$-isotropic 
complex $2$-planes, that is, 
$F_{\alpha_2}=G_2/Q_{\alpha_2}\simeq\{\pi\in\operatorname{Gr}_2(\mathbb C^7)\mid u\times 
v=0\text{ for all }u,v\in\pi\}$.

\subsection{Orbits on $G_{2}/Q_{\alpha_1}$}
\label{subsec:local-equations-X-alpha1}

We describe the $G_{2(2)}$-orbits on the partial flag manifold
$F_{\alpha_1}$ using its realization as a projective quadric. We retain
the matrix $N$ defined in \eqref{N_matrix} and the associated $G_{2(2)}$-invariant Hermitian form
$q(z,w)=\overline z^{,t}Nw$. In particular,
 $$q(z)
=
-|z_0|^2
-4\operatorname{Re}\left(
\overline z_1z_4+
\overline z_2z_5+
\overline z_3z_6
\right).$$
We shall also use the $G_2$-invariant cross product \eqref{eq:G2-cross-product}
and write its associated  vector valued form as $\omega(z)=(C_0(z),\ldots,C_6(z))$
where the components are given by

$$
\begin{aligned}
C_0={}&
2\left(
\overline z_1z_4+\overline z_2z_5+\overline z_3z_6
-\overline z_4z_1-\overline z_5z_2-\overline z_6z_3
\right),\\
C_1={}&
\overline z_0z_1-z_0\overline z_1
-2\left(\overline z_5z_6-\overline z_6z_5\right),\\
C_2={}&
\overline z_0z_2-z_0\overline z_2
-2\left(\overline z_6z_4-\overline z_4z_6\right),\\
C_3={}&
\overline z_0z_3-z_0\overline z_3
-2\left(\overline z_4z_5-\overline z_5z_4\right),\\
C_4={}&
-\overline z_0z_4+z_0\overline z_4
+2\left(\overline z_2z_3-\overline z_3z_2\right),\\
C_5={}&
-\overline z_0z_5+z_0\overline z_5
+2\left(\overline z_3z_1-\overline z_1z_3\right),\\
C_6={}&
-\overline z_0z_6+z_0\overline z_6
+2\left(\overline z_1z_2-\overline z_2z_1\right).
\end{aligned}
$$
Finally, set $M_{ij}(z)=\overline z_i z_j-\overline z_j z_i$ for $0\leq i<j\leq6$. The equations $M_{ij}(z)=0$ for all $i<j$ are equivalent to $\overline z\wedge z=0$.
We prove the following
\begin{theorem}
\label{prop:orbit-equations-X-alpha1}
The $G_{2(2)}$-orbits on $F_{\alpha_1}$ are described in the following table
\begin{table}[H]
\centering
\small
\renewcommand{\arraystretch}{1.15}
\begin{tabular}{c|c|l|c}
$G_{2(2)}$-orbit
&
$\dim_{\mathbb R}$
&
defining conditions
&
topology
\\
\hline
$\mathcal O_{10}^{+}$
&
$10$
&
$q(z)>0$
&
open
\\
$\mathcal O_{10}^{-}$
&
$10$
&
$q(z)<0$
&
open
\\
$\mathcal O_{9}$
&
$9$
&
$q(z)=0$, $\omega(z)\neq0$
&
intermediate
\\
$\mathcal O_{7}$
&
$7$
&
$q(z)=0$, $\omega(z)=0$, $z\wedge\overline z\neq0$
&
intermediate
\\
$\mathcal O_{5}$
&
$5$
&
$z\wedge\overline z=0$
&
closed
\end{tabular}
\caption{$G_{2(2)}$-orbits on $F_{\alpha_1}=G_2/Q_{\alpha_1}$.}
\label{tab:Xalpha1-orbit-summary}
\end{table}
\end{theorem}

\begin{proof}
All topological statements are understood with respect to the ordinary 
topology of the projective quadric $F_{\alpha_1}$.
\par
Since $q(\lambda z,\lambda 
z)=|\lambda|^2q(z)$ for every $\lambda\in\mathbb C^*$, the sign and the 
vanishing of $q(z)$ depend only on the projective point $[z]$. The function \([z] \mapsto {q(z)}\backslash {\sum_{j=0}^{6} |z_j|^2}\) is a well-defined continuous real-valued function on \(\mathbb{P}(\mathbb{C}^7)\) with the same sign as \(q(z)\).
Therefore, $\mathcal O_{10}^{+}:=\{[z]\in F_{\alpha_1}\mid q(z)>0\}$ and $\mathcal 
O_{10}^{-}:=\{[z]\in F_{\alpha_1}\mid q(z)<0\}$ are disjoint $G_{2(2)}$-invariant open 
subsets of $F_{\alpha_1}$, whereas $H:=\{[z]\in F_{\alpha_1}\mid q(z)=0\}$ is a closed $G_{2(2)}$-
invariant subset. Evaluation of $q(z)$ on the representatives obtained in 
the preceding orbit classification shows that the two open $G_{2(2)}$-orbits 
lie in $\mathcal O_{10}^{+}$ and $\mathcal O_{10}^{-}$, respectively. Since 
the action has exactly two open orbits, they coincide with these two subsets. 
The set $H$ is the common boundary of the two open orbits.\par
We next examine the vanishing locus of the cross-product invariant 
$\omega(z)=\overline z\times z$. 
Set $Z:=\{[z]\in F_{\alpha_1}\mid\omega(z)=0\}$, 
it is closed in $F_{\alpha_1}$, since it is the common zero locus of the continuous functions  obtained from $C_0,\ldots,C_6$ after normalizing by ${\sum_{j=0}^{6} |z_j|^2}$. 
Therefore, $H\setminus Z{=}\{[z]\in F_{\alpha_1}\mid q(z)=0,\ \omega(z)\neq0\}$ is 
relatively open in $H$. The representatives and dimensions computed in the 
Lemma~\ref{thm:orbit-count-maximal-flags} identify this stratum with the unique orbit of 
real dimension $9$, hence $\mathcal O_9{=}H\setminus Z$. It remains to decompose $H\cap Z$. \par
Since the vanishing of 
$\overline z\wedge z$ is projectively well defined and $G_{2(2)}$-invariant, 
define $R{:=}\{[z]\in F_{\alpha_1}\mid\overline z\wedge z=0\}{=}\{[z]\in F_{\alpha_1}\mid M_{ij}(z)=0\text{ for all }0\leq 
i<j\leq6\}$, hence after normalizing $R$ is closed in $F_{\alpha_1}$. It is also the fixed-point set of 
projective complex conjugation on $F_{\alpha_1}$, after 
multiplication by a suitable nonzero complex scalar, $z$ can therefore be 
chosen real. Consequently, $R=\mathcal{Q}\cap\mathbb{RP}^{6}=\{[x]\in\mathbb{RP}^{6}\mid 
x_0^2+4x_1x_4+4x_2x_5+4x_3x_6=0\}$. If $x\in\mathbb R^7$ represents a point of $R$, then $q(x)=n(x)=0$.  Moreover, the skew-symmetry of the cross product gives $\omega(x)=x\times x=0$. Thus, $R\subseteq H\cap Z$. The group $G_{2(2)}$ acts transitively on the projectivized real null cone 
$R$. Hence $R$ is the unique closed $G_{2(2)}$-orbit in 
$F_{\alpha_1}=G_2/Q_{\alpha_1}$, and therefore $\mathcal O_{5}=R$. 
Its complement in $H\cap Z$ is $(H\cap Z)\setminus R=\{[z]\in F_{\alpha_1}\mid q(z)=0,\ 
\omega(z)=0,\ \overline z\wedge z\neq0\}$. This subset is relatively open in 
$H\cap Z$ and is $G_{2(2)}$-invariant. Lemma~\ref{thm:orbit-count-maximal-flags} shows that $(H\cap Z)\setminus R$ contains 
the $7$-dimensional orbit. Since the complete orbit decomposition of 
$F_{\alpha_1}$ consists of two open orbits, one orbit of dimension $9$, one 
orbit of dimension $7$, and the closed orbit, no further orbit can occur in 
this stratum. Therefore, $\mathcal O_7=(H\cap Z)\setminus R$. We have thus obtained the disjoint decomposition $F_{\alpha_1}=\mathcal 
O_{10}^{+}\sqcup\mathcal O_{10}^{-}\sqcup\mathcal O_9\sqcup\mathcal 
O_7\sqcup\mathcal O_{5}$.
\end{proof}
We now give explicit equations for the two non-open, non-totally-real orbits on the
projective quadric $F_{\alpha_1}$.
On the chart $z_1\neq0$, every point of $F_{\alpha_1}$ has a unique representative
\begin{equation}\label{eq:G2-alpha1-five-coordinate-chart}
z=(2s,1,x,y,-s^2-xb-yc,b,c),\qquad (s,x,y,b,c)\in\mathbb C^5.
\end{equation}
Indeed, for $z_1=1$ the equation $n(z)=0$ reads
$z_4=-\tfrac14z_0^2-z_2z_5-z_3z_6$, so $z_4$ is determined by
$(z_0,z_2,z_3,z_5,z_6)\in\mathbb C^5$, which can be chosen freely. Hence
$(s,x,y,b,c)=(\tfrac12z_0,z_2,z_3,z_5,z_6)$ are holomorphic coordinates on
$F_{\alpha_1}\cap\{z_1\neq0\}$. Since $\mathbb C^7$ is an irreducible
$G_2$-module, the $G_{2(2)}$-translates of this chart cover $F_{\alpha_1}$.

\begin{theorem}\label{thm:G2-alpha1-explicit-local-models}
Let $\mathcal O_7$ and $\mathcal O_9$ denote the $7$- and $9$-dimensional orbits in the
classification of Theorem~\ref{prop:orbit-equations-X-alpha1}, and let
$R=\{(s,x,y,b,c)\in\mathbb R^5\}$ be the set of real points of the chart. In the chart
\eqref{eq:G2-alpha1-five-coordinate-chart} the following hold.
\par
\smallskip
\noindent\textup{(i)}
Introduce the holomorphic coordinates
$\xi_1=b$, $\xi_2=c-i$, $w_1=s+ib$, $w_2=x-is$, $w_3=y$,
centered at $p=[e_1+ie_6]\in\mathcal O_7$. Then
$\mathcal O_7\cap\{z_1\neq0\}=\Sigma_7\setminus R$, where $\Sigma_7$ is the
real-analytic submanifold of $\mathbb C^5$ of real dimension $7$ defined by
\begin{equation}\label{eq:G2-alpha1-O7-local-model}
\begin{aligned}
\operatorname{Im}w_1
&=\operatorname{Im}(\xi_1\overline\xi_2),\\
\operatorname{Im}w_2
&=\operatorname{Re}(\xi_1\overline\xi_2)
+(\operatorname{Re}w_1)(\operatorname{Im}\xi_2)
-(\operatorname{Re}\xi_2)\operatorname{Im}(\xi_1\overline\xi_2),\\
\operatorname{Im}w_3
&=-|\xi_1|^2
-(\operatorname{Re}w_1)(\operatorname{Im}\xi_1)
+(\operatorname{Re}\xi_1)\operatorname{Im}(\xi_1\overline\xi_2).
\end{aligned}
\end{equation}
\par
\smallskip
\noindent\textup{(ii)}
Introduce the holomorphic coordinates
$\xi_1=s$, $\xi_2=\tfrac12(y+c)$, $\xi_3=\tfrac12(y-c)$, $\zeta=x-i$, $w=-b$,
centered at $p'=[e_1+ie_2]\in\mathcal O_9$. Then
$\mathcal O_9\cap\{z_1\neq0\}=\{q=0\}\setminus\{\omega=0\}$, where
\begin{equation}\label{eq:G2-alpha1-affine-Hermitian-equation}
q(z)=-8\bigl((\operatorname{Im}s)^2+(\operatorname{Im}x)(\operatorname{Im}b)
+(\operatorname{Im}y)(\operatorname{Im}c)\bigr).
\end{equation}
On the open set $1+\operatorname{Im}\zeta\neq0$ the equation $q(z)=0$ is the smooth
hypersurface
\begin{equation}\label{eq:G2-alpha1-O9-local-model}
\operatorname{Im}w=
\frac{(\operatorname{Im}\xi_1)^2+(\operatorname{Im}\xi_2)^2-(\operatorname{Im}\xi_3)^2}
{1+\operatorname{Im}\zeta},
\end{equation}
and $\mathcal O_9$ is the open subset of it where $\omega\neq0$; in particular $\mathcal O_9$
coincides with \eqref{eq:G2-alpha1-O9-local-model} in a neighbourhood of $p'$.
\end{theorem}

\begin{proof}
We retain the matrix $N$ defined in \eqref{N_matrix}, the cross product
\eqref{eq:G2-cross-product}, the Hermitian form $q$, and the components $C_0,\dots,C_6$
of $\omega(z)=\overline z\times z$ introduced before
Theorem~\ref{prop:orbit-equations-X-alpha1}.
Substitution of \eqref{eq:G2-alpha1-five-coordinate-chart} into $q$ gives
\eqref{eq:G2-alpha1-affine-Hermitian-equation}.
A direct computation shows that $C_1=C_5=C_6=0$ is equivalent to
\begin{equation}\label{eq:G2-alpha1-affine-cross-product-equations}
\operatorname{Im}s=\operatorname{Im}(b\overline c),\qquad
\operatorname{Im}x=\operatorname{Im}(\overline sc),\qquad
\operatorname{Im}y=-\operatorname{Im}(\overline sb),
\end{equation}
and that $C_0,C_2,C_3,C_4$ and $q(z)$ vanish whenever
\eqref{eq:G2-alpha1-affine-cross-product-equations} holds. Hence, in the chart,
$\omega(z)=0$ if and only if \eqref{eq:G2-alpha1-affine-cross-product-equations} holds.
Moreover $z\wedge\overline z=0$ if and only if $(s,x,y,b,c)\in R$, since $z_1=1$.

(i) By Theorem~\ref{prop:orbit-equations-X-alpha1}, $\mathcal O_7=\{q=0,\ \omega=0,\
z\wedge\overline z\neq0\}$. By the previous paragraph, $\mathcal O_7\cap\{z_1\neq0\}$
is the set of solutions of \eqref{eq:G2-alpha1-affine-cross-product-equations} outside $R$.
In the coordinates of \textup{(i)} one has $b=\xi_1$, $c=\xi_2+i$, $s=w_1-i\xi_1$,
$x=w_2+iw_1+\xi_1$ and $y=w_3$. Solving
\eqref{eq:G2-alpha1-affine-cross-product-equations} successively for
$\operatorname{Im}w_1$, $\operatorname{Im}w_2$, $\operatorname{Im}w_3$ gives
\eqref{eq:G2-alpha1-O7-local-model}, whose right-hand sides depend only on
$(\xi_1,\xi_2,\operatorname{Re}w_1,\operatorname{Re}w_2,\operatorname{Re}w_3)$. Hence
$\Sigma_7$ has real dimension $10-3=7$. Finally $R\subset\Sigma_7$, and $p\in\mathcal O_7$.

(ii) By Theorem~\ref{prop:orbit-equations-X-alpha1}, $\mathcal O_9=\{q=0,\ \omega\neq0\}$,
and $\omega=0$ is described above. In the coordinates of \textup{(ii)},
$\operatorname{Im}x=1+\operatorname{Im}\zeta$, $\operatorname{Im}b=-\operatorname{Im}w$ and
$(\operatorname{Im}y)(\operatorname{Im}c)=(\operatorname{Im}\xi_2)^2-(\operatorname{Im}\xi_3)^2$,
so $q(z)=0$ becomes \eqref{eq:G2-alpha1-O9-local-model} wherever $1+\operatorname{Im}\zeta\neq0$.
At $p'$ one has $q(z)=0$ and $C_6=4i\neq0$, hence $p'\in\mathcal O_9$.
\end{proof}

\subsection{Orbits on $G_{2}/Q_{\alpha_2}$}

We now describe the $G_{2(2)}$-orbits on $F_{\alpha_2}=G_2/Q_{\alpha_2}$, retaining its
realization as the variety of $\times$-isotropic complex $2$-planes and the notation
introduced in Section~\ref{sec: equation and closure relations for G2/B}. For
$\pi=\langle u,v\rangle_{\mathbb C}\in F_{\alpha_2}$, the Gram matrix of $q|_\pi$ in the
basis $(u,v)$ is
\begin{equation}\label{Gram_matrix}
\mathcal H_\pi(u,v)=
\begin{pmatrix}
q(u,u)&q(u,v)\\
q(v,u)&q(v,v)
\end{pmatrix}
=
\begin{pmatrix}
\overline u^{\,t}Nu&\overline u^{\,t}Nv\\
\overline v^{\,t}Nu&\overline v^{\,t}Nv
\end{pmatrix}.
\end{equation}
Under a change of basis $(u',v')=(u,v)S$, with $S\in\operatorname{GL}(2,\mathbb C)$, one has
$\mathcal H_\pi(u',v')=S^*\mathcal H_\pi(u,v)S$. Hence its rank and signature depend only on
$\pi$. Moreover, $\mathcal H_{g\pi}(gu,gv)=\mathcal H_\pi(u,v)$ for every $g\in G_{2(2)}$.
The reality condition $\pi=\overline\pi$ can be expressed in Pl\"ucker coordinates. Set
$p_{ij}=u_iv_j-u_jv_i$ for $0\leq i<j\leq6$, so that $\pi$ is represented by
$[u\wedge v]=[p_{01}:\cdots:p_{56}]$. Then $\pi=\overline\pi$ if and only if
$p_{ij}\overline{p_{k\ell}}-p_{k\ell}\overline{p_{ij}}=0$ for every $0\leq i<j\leq6$ and
$0\leq k<\ell\leq6$. Again we will use these equations together with the rank and signature
of $q|_\pi$ to distinguish the orbits. We have

\begin{theorem}
\label{thm:orbit-equations-X-alpha2}
The $G_{2(2)}$-orbits on the maximal flag manifold $F_{\alpha_2}=G_2/Q_{\alpha_2}$ are
described in the following table:

\begin{table}[H]
\centering
\small
\renewcommand{\arraystretch}{1.15}
\begin{tabular}{c|c|l|c}
$G_{2(2)}$-orbit & $\dim_{\mathbb R}$ & defining conditions & topology\\
\hline
$\mathcal O_{10}^{+-}$ & $10$ & $\operatorname{sign}(q_{|\pi})=(1,1)$ & open\\
$\mathcal O_{10}^{-}$ & $10$ & $\operatorname{sign}(q_{|\pi})=(0,2)$ & open\\
$\mathcal O_{9}$ & $9$ & $\operatorname{sign}(q_{|\pi})=(0,1)$ & intermediate\\
$\mathcal O_{7}$ & $7$ & $q_{|\pi}=0$, $\pi\neq\overline\pi$ & intermediate\\
$\mathcal O_{5}$ & $5$ & $q_{|\pi}=0$, $\pi=\overline\pi$ & closed
\end{tabular}
\caption{$G_{2(2)}$-orbits on $F_{\alpha_2}=G_2/Q_{\alpha_2}$.}
\label{tab:Xalpha2-orbit-summary}
\end{table}
\end{theorem}

\begin{proof}
We use the Gram matrices $\mathcal H_\pi(u,v)$ introduced in \eqref{Gram_matrix}. If
$(u',v')=(u,v)S$, with $S\in\operatorname{GL}(2,\mathbb C)$, then
$\mathcal H_\pi(u',v')=S^*\mathcal H_\pi(u,v)S$, and hence
$D(\pi;u',v')=|\det S|^2D(\pi;u,v)$, where $D(\pi;u,v):=\det\mathcal H_\pi(u,v)$. Thus,
although the value of $D(\pi;u,v)$ depends on the chosen basis, its sign and vanishing depend
only on $\pi$. Since $q$ is $G_{2(2)}$-invariant, the rank and signature of $q|_\pi$ are
constant along each $G_{2(2)}$-orbit.\par
By Lemma~\ref{thm:orbit-count-maximal-flags}, $G_{2(2)}$ has exactly five orbits on
$F_{\alpha_2}$, and by Section~\ref{subsec:projecting-orbits} each of them is the projection
of one of the orbits $\mathcal O(\gamma_i)$ on $G_2/B$, $i=1,\dots,10$, listed in
Table~\ref{tab:example}. The signature of $q|_\pi$ and the reality of $\pi$ for the planes
$\pi$ associated with the $\gamma_i$ are given in
Section~\ref{sec: equation and closure relations for G2/B}, and the identifications among the
projected orbits in Section~\ref{subsec:projecting-orbits}. Namely, the planes associated with
$\gamma_4$, $\gamma_8$, and $\gamma_{10}$ have signature $(1,1)$ and project to the same
orbit, whereas the plane associated with $\gamma_9$ has signature $(0,2)$. These give the two
open orbits $\mathcal O_{10}^{+-}$ and $\mathcal O_{10}^{-}$, respectively.\par
Their complement is the closed invariant subset
$\Sigma:=\{\pi\in F_{\alpha_2}\mid\operatorname{rank}(q|_\pi)\leq1\}$. Its rank-one stratum
$\Sigma_1:=\{\pi\in F_{\alpha_2}\mid\operatorname{rank}(q|_\pi)=1\}$ is relatively open in
$\Sigma$. The planes associated with $\gamma_5$ and $\gamma_6$ have signature $(0,1)$ and, by
Section~\ref{subsec:projecting-orbits}, project to the same $9$-dimensional orbit. Every other
representative has restriction of rank either zero or two, so by
Lemma~\ref{thm:orbit-count-maximal-flags}
$\mathcal O_9=\Sigma_1=\{\pi\in F_{\alpha_2}\mid\operatorname{sign}(q|_\pi)=(0,1)\}$.\par
It remains to consider the rank-zero locus
$\Sigma_0:=\{\pi\in F_{\alpha_2}\mid q|_\pi=0\}$. For any basis $(u,v)$ of $\pi$, the
condition $q|_\pi=0$ is equivalent to $q(u,u)=q(u,v)=q(v,u)=q(v,v)=0$.\par
Let $R:=\{\pi\in F_{\alpha_2}\mid\pi=\overline\pi\}$ be the real locus. The Pl\"ucker-coordinate
equations given before the theorem show that $R$ is closed, and it is $G_{2(2)}$-invariant. If
$\pi\in R$, then $\pi$ admits a real basis $(u,v)$. Since every plane in $F_{\alpha_2}$ is
$n$-isotropic, polarization gives $u^tNu=u^tNv=v^tNv=0$. For real vectors, $q$ agrees with the
polarization of $n$, and hence $q|_\pi=0$. Therefore $R\subseteq\Sigma_0$.\par
By Section~\ref{sec: equation and closure relations for G2/B}, the planes associated with
$\gamma_1$ and $\gamma_7$ are real, those associated with $\gamma_2$ and $\gamma_3$ are
non-real, and all four satisfy $q|_\pi=0$. By Section~\ref{subsec:projecting-orbits}, the first
pair projects to the closed orbit $\mathcal O_5$, and the second pair to the $7$-dimensional
orbit $\mathcal O_7$. By Lemma~\ref{thm:orbit-count-maximal-flags} there are no other orbits,
hence $R=\mathcal O_5$ and $\Sigma_0\setminus R=\mathcal O_7$. Equivalently,
$\mathcal O_7=\{\pi\in F_{\alpha_2}\mid q|_\pi=0,\ \pi\neq\overline\pi\}$. This proves all the
asserted orbit equations and dimensions.
\end{proof}
On the affine chart $p_{16}\neq0$, every $\pi\in F_{\alpha_2}$ has a unique ordered basis
$(u,v)$ with $(u_1,u_6)=(1,0)$ and $(v_1,v_6)=(0,1)$, since $p_{16}\neq0$ means that the
projection of $\pi$ onto the coordinates with indices $1$ and $6$ is an isomorphism. Write
$u=(s,1,x,y,u_4,u_5,0)$ and $v=(t,0,v_2,v_3,v_4,r,1)$. The components with indices $1$ and $6$
of $u\times v=0$ give $u_5=-t/2$ and $v_2=s/2$, respectively. The components with indices $2$,
$5$, and $0$ then give $u_4=tx/2-s^2/4$, $v_3=-sr/2-t^2/4$, and $v_4=-xr-y-st/4$, and the
remaining components vanish identically. Hence $(s,t,x,y,r)\in\mathbb C^5$ are holomorphic
coordinates on $F_{\alpha_2}\cap\{p_{16}\neq0\}$, and
\begin{equation}\label{eq:G2-alpha2-affine-chart}
\begin{aligned}
u&=(s,1,x,y,\tfrac{tx}{2}-\tfrac{s^2}{4},-\tfrac t2,0),\\
v&=(t,0,\tfrac s2,-\tfrac{sr}{2}-\tfrac{t^2}{4},
-xr-y-\tfrac{st}{4},r,1).
\end{aligned}
\end{equation}

\begin{theorem}\label{thm:local-models-X-alpha2}
Let $\mathcal O_7$ and $\mathcal O_9$ denote the $7$- and $9$-dimensional orbits in the
classification of Theorem~\ref{thm:orbit-equations-X-alpha2}. In the affine coordinates
\eqref{eq:G2-alpha2-affine-chart}, the following hold.
\par
\smallskip
\noindent\textup{(i)}
Let $\pi_7\in\mathcal O_7$ be the plane with $(s,t,x,y,r)=(0,0,i,0,0)$. Introduce holomorphic
coordinates centered at $\pi_7$ by $\xi_1=x-i$, $\xi_2=s$, $w_1=t$, $w_2=r$, and $w_3=y+ir$.
The set $\mathcal O_7\cap\{p_{16}\neq0,\ 1+\operatorname{Im}\xi_1>0\}$ is defined by
\begin{equation}\label{eq:G2-alpha2-O7-local-model}
\begin{aligned}
\operatorname{Im}w_1
&=\tfrac{(\operatorname{Im}\xi_2)^2}{2(1+\operatorname{Im}\xi_1)},\qquad
\operatorname{Im}w_2
=-\tfrac{(\operatorname{Im}\xi_2)^3}{8(1+\operatorname{Im}\xi_1)^2},\\
\operatorname{Im}w_3
&=\tfrac{(\operatorname{Re}\xi_2)(\operatorname{Im}\xi_2)^2}{8(1+\operatorname{Im}\xi_1)}
-\tfrac12(\operatorname{Re}w_1)(\operatorname{Im}\xi_2)
-(\operatorname{Re}w_2)(\operatorname{Im}\xi_1).
\end{aligned}
\end{equation}
These equations define a real-analytic submanifold of $\mathbb C^5$ of real dimension $7$.
\par
\smallskip
\noindent\textup{(ii)}
Let $\pi_9=\langle u_9,v_9\rangle_{\mathbb C}\in\mathcal O_9$, where
$u_9=(i,1,0,0,\tfrac14,0,0)$ and $v_9=(0,0,\tfrac i2,0,0,0,1)$, that is,
$(s,t,x,y,r)=(i,0,0,0,0)$. Introduce holomorphic coordinates centered at $\pi_9$ by
$s=i+\xi_1$, $t=\xi_2$, $x=\xi_3$, $y=\xi_4$, and $r=w$, and set
\begin{equation}\label{eq:G2-alpha2-O9-auxiliary}
\begin{aligned}
A&=(1+\operatorname{Im}\xi_1)^2-2(\operatorname{Im}\xi_2)(\operatorname{Im}\xi_3),\\
B&=(\operatorname{Im}\xi_2)^2+2(1+\operatorname{Im}\xi_1)\operatorname{Im}w,\\
C&=(1+\operatorname{Im}\xi_1)\xi_2-\tfrac{i+\xi_1}{2}\operatorname{Im}\xi_2
+2w\operatorname{Im}\xi_3+2\operatorname{Im}\xi_4.
\end{aligned}
\end{equation}
The set $\mathcal O_9\cap\{p_{16}\neq0,\ A>0\}$ is defined by
\begin{equation}\label{eq:G2-alpha2-O9-local-model}
AB=|C|^2,
\end{equation}
and this equation defines a smooth real-analytic hypersurface of $\mathbb C^5$ in a
neighbourhood of the origin.
\end{theorem}

\begin{proof}
We retain the matrix $N$ defined in \eqref{N_matrix}, the cross product defined in
\eqref{eq:G2-cross-product}, and the Hermitian form $q$. Substitution of
\eqref{eq:G2-alpha2-affine-chart} into the Gram matrix \eqref{Gram_matrix} gives
\begin{equation}\label{eq:G2-alpha2-Gram-entries}
\begin{aligned}
q(u,u)
&=-2(\operatorname{Im}s)^2
  +4(\operatorname{Im}t)(\operatorname{Im}x),\\
q(v,v)
&=-2(\operatorname{Im}t)^2
  -4(\operatorname{Im}s)(\operatorname{Im}r),\\
q(u,v)
&=2i\bigl(
(\operatorname{Im}s)t
-\tfrac{s}{2}\operatorname{Im}t
+2r\operatorname{Im}x
+2\operatorname{Im}y
\bigr),
\end{aligned}
\end{equation}
with $q(v,u)=\overline{q(u,v)}$.

(i) Work on the open set $\operatorname{Im}x>0$. No plane in this set is real: indeed, if
$\pi=\overline\pi$, the uniqueness of the normalized basis would imply $u=\overline u$ and
$v=\overline v$, and hence $x\in\mathbb R$. By Theorem~\ref{thm:orbit-equations-X-alpha2},
$\mathcal O_7\cap\{p_{16}\neq0,\ \operatorname{Im}x>0\}$ is therefore characterized by
$q|_\pi=0$, equivalently by $\mathcal H_\pi(u,v)=0$. The equation $q(u,u)=0$ in
\eqref{eq:G2-alpha2-Gram-entries} gives
$\operatorname{Im}t=(\operatorname{Im}s)^2/(2\operatorname{Im}x)$. Moreover,
$\operatorname{Re}q(u,v)=-(\operatorname{Im}s)(\operatorname{Im}t)
-4(\operatorname{Im}x)(\operatorname{Im}r)$, so $\operatorname{Re}q(u,v)=0$ gives
$\operatorname{Im}r=-(\operatorname{Im}s)^3/(8(\operatorname{Im}x)^2)$. After substituting the
expression for $\operatorname{Im}t$, the equation $\operatorname{Im}q(u,v)=0$ becomes
\begin{equation}\label{eq:G2-alpha2-O7-solved-equations}
\operatorname{Im}y
=
\frac{(\operatorname{Re}s)(\operatorname{Im}s)^2}{8\operatorname{Im}x}
-\frac12(\operatorname{Re}t)(\operatorname{Im}s)
-(\operatorname{Re}r)(\operatorname{Im}x).
\end{equation}
Substitution of the resulting expressions for $\operatorname{Im}t$ and $\operatorname{Im}r$
into $q(v,v)$ makes this entry vanish identically. Hence on $\operatorname{Im}x>0$ the
condition $q|_\pi=0$ is equivalent to the three equations above. In the coordinates of
\textup{(i)} one has $s=\xi_2$, $t=w_1$, $x=i+\xi_1$, $r=w_2$, and $y=w_3-iw_2$, so that
$\operatorname{Im}x=1+\operatorname{Im}\xi_1$ and
$\operatorname{Im}y=\operatorname{Im}w_3-\operatorname{Re}w_2$, and the equations become
\eqref{eq:G2-alpha2-O7-local-model}. Their right-hand sides depend only on
$(\xi_1,\xi_2,\operatorname{Re}w_1,\operatorname{Re}w_2,\operatorname{Re}w_3)$ and are
real-analytic on $1+\operatorname{Im}\xi_1>0$, hence they define a real-analytic graph of real
dimension $10-3=7$. At the origin $q|_\pi=0$ and $x=i\notin\mathbb R$, so
$\pi_7\in\mathcal O_7$.

(ii) At $\pi_9$ one has $\mathcal H_{\pi_9}(u_9,v_9)=\operatorname{diag}(-2,0)$. In the
coordinates of \textup{(ii)}, \eqref{eq:G2-alpha2-Gram-entries} and
\eqref{eq:G2-alpha2-O9-auxiliary} give $q(u,u)=-2A$, $q(v,v)=-2B$, and $q(u,v)=2iC$. Hence
$\det\mathcal H_\pi(u,v)=4(AB-|C|^2)$. On the open set $A>0$ one has $q(u,u)<0$, so
$q|_\pi\neq0$, and $\det\mathcal H_\pi(u,v)=0$ if and only if $q|_\pi$ has rank one with a
negative eigenvalue, that is, $\operatorname{sign}(q|_\pi)=(0,1)$. By
Theorem~\ref{thm:orbit-equations-X-alpha2}, this holds exactly on $\mathcal O_9$. This proves
that $\mathcal O_9\cap\{p_{16}\neq0,\ A>0\}$ is defined by
\eqref{eq:G2-alpha2-O9-local-model}. Finally, $A=1$, $B=0$ and $C=0$ at the origin, and
$\partial B/\partial(\operatorname{Im}w)=2$, so
$\partial(AB-|C|^2)/\partial(\operatorname{Im}w)=2\neq0$ there. The real-analytic implicit
function theorem therefore shows that \eqref{eq:G2-alpha2-O9-local-model} defines a smooth real
hypersurface of $\mathbb C^5$ near the origin, of real dimension $9$.

Since $G_{2(2)}$ acts transitively on $\mathcal O_7$ and on $\mathcal O_9$ by biholomorphisms of
$F_{\alpha_2}$, the models hold at every point up to this action.
\end{proof}
\section{$CR$ structures on $\textbf{G}_{2(2)}$-orbits}
In this section, we study the $G_{2(2)}$-invariant $CR$ structures of all the
orbits obtained above. In order to do this, we use the tools defined in Section~\ref{CR_orbits section}. For each orbit, we determine the corresponding $CR$ algebra (see Definition \ref{dfn_CR_alg}) and compute the $k$-nondegeneracy of its Levi form at that point by evaluating the associated subalgebras sequence \eqref{eq:Freeman-filtration-CRalg}. A key step in this analysis is to determine the action of the real conjugation on the Lie algebra $\mathfrak{g}_2$ and, consequently, on its root-space decomposition.
We follow the description of root involutions given in \cite[Sec.~2]{MariniMedoriNacinovich2024}.
Since the Dynkin diagram of type $G_{2}$ has no nontrivial automorphisms,
every automorphism of its root system $\Delta$ belongs to the Weyl group, 
i.e.  $\mathbf{Aut}(\Delta)=W(\Delta)$. Therefore, for every $\sigma$-stable 
Cartan subalgebra $\mathfrak h\subset\mathfrak g$, the involution induced by 
the conjugation $\sigma$ on $\Delta(\mathfrak g,\mathfrak h)$ can be 
identified with an involution of $W$. Recall that $W=\langle s_{\alpha_{1}},s_{\alpha_{2}}\rangle$ is the dihedral group of order $12$, 
with $s_{\alpha_{1}}^{2}=s_{\alpha_{2}}^{2}=1$ and $(s_{\alpha_{1}}s_{\alpha_{2}})^{6}=1$. Its involutions are $\pm\mathrm{I}$
and the six reflections $s_{\gamma}$, with $\gamma$ ranging over the positive roots. Thus there are exactly eight root involutions. For later use, we record their action on the positive simple roots.
The eight involutions of $W$ and their action on $\alpha_{1}$ and $\alpha_{2}$ are given by the following   table 
\begin{table}[H]
\centering
\small
\renewcommand{\arraystretch}{1.15}
\begin{tabular}{c|cc}
$s_{\alpha}$ & $s_{\alpha}(\alpha_{1})$ & $s_{\alpha}(\alpha_{2})$\\
\hline
$\mathrm{I}$ & $\alpha_{1}$ & $\alpha_{2}$\\
$-\mathrm{I}$ & $-\alpha_{1}$ & $-\alpha_{2}$\\
$s_{\alpha_{1}}$ & $-\alpha_{1}$ & $3\alpha_{1}+\alpha_{2}$\\
$s_{\alpha_{2}}$ & $\alpha_{1}+\alpha_{2}$ & $-\alpha_{2}$\\
$s_{\alpha_{1}+\alpha_{2}}$
& $2\alpha_{1}+\alpha_{2}$
& $-3\alpha_{1}-2\alpha_{2}$\\
$s_{2\alpha_{1}+\alpha_{2}}$
& $-\alpha_{1}-\alpha_{2}$
& $\alpha_{2}$\\
$s_{3\alpha_{1}+\alpha_{2}}$
& $-2\alpha_{1}-\alpha_{2}$
& $3\alpha_{1}+2\alpha_{2}$\\
$s_{3\alpha_{1}+2\alpha_{2}}$
& $\alpha_{1}$
& $-3\alpha_{1}-\alpha_{2}$
\end{tabular}
\caption{Action of the root involutions on the positive simple roots.}
\label{tab:involutions-simple-roots}
\end{table}
We now incorporate the parabolic data into the root-theoretic description. With respect to the fixed simple system $\Pi=\{\alpha_{1},\alpha_{2}\}$, a crossed Dynkin diagram is specified by a subset $\Pi_0\subseteq\Pi$, which determines the standard parabolic subalgebra $\mathfrak q_{\Pi_0}$ and its parabolic root set $\mathcal Q_{\Pi_0}$. For the  computations, the crossed diagram must be considered together with the involution induced by the conjugation on the root system. Let $\tilde{\sigma}$ be the conjugation defining the real form  and
its induced involution on the root system. \par
Fix a point $x=gQ\in F=G/Q$,  the $CR$ algebra at this point is $(\mathfrak{g}_0,\mathfrak{q}_x)$. If we denote by $\widetilde{\sigma}_{x}=\operatorname{Ad}(g^{-1})\circ\sigma\circ\operatorname{Ad}(g)$, we obtain the identification of the $CR$ $(\mathfrak{g}_0,\mathfrak{q}_x)$ with $(\mathfrak{g}^{\tilde\sigma_x},\mathfrak{q})$, where $\mathfrak{g}^{\tilde\sigma_x}=\mathrm{Ad}(g^{-1})\mathfrak{g}_0$.\par
Thus, the crossed Dynkin diagram records the parabolic data, whereas the Weyl involution
records the Cartan class and the root-level action of the conjugation. 
\par
\textbf{Notation.}
Following the root-index notation of \cite[Section~6]{Fels2007}, for every pair $(n,m)\in\mathbb Z^{2}$ such that $n\alpha_{1}+m\alpha_{2}\in\Delta$, we set ${\mathfrak g}^{nm}:={\mathfrak g}_{\,n\alpha_{1}+m\alpha_{2}}$, i.e. the superscript $nm$ is merely an abbreviation for the root $n\alpha_{1}+m\alpha_{2}$.
\subsection{$CR$ structures on the $G_{2(2)}$-orbits in $G_{2}/B$}
For the complete flag the crossed Dynkin diagram is
$$\GTwoCrossedDiagram{0}{0}{1}{1}.$$
We use the fixed Borel $B(\Pi)$  with $\Pi=\{\alpha_1,\alpha_2\}$ and corresponding positive roots
$\Delta^+=
\{\alpha_1,\alpha_2,\alpha_1+\alpha_2,2\alpha_1+\alpha_2,
3\alpha_1+\alpha_2,3\alpha_1+2\alpha_2\}.$
\par\medskip
\begin{center}
\begin{tikzpicture}[
    scale=1.7,
    root/.style={circle,fill=black,inner sep=1.4pt},
    guide/.style={gray!45,thin},
    every node/.style={font=\footnotesize}
]

\coordinate (alpha1)          at (0:1);
\coordinate (alpha1+alpha2)   at (120:1);
\coordinate (2alpha1+alpha2)  at (60:1);
\coordinate (-alpha1)         at (180:1);
\coordinate (-alpha1-alpha2)  at (300:1);
\coordinate (-2alpha1-alpha2) at (240:1);

\coordinate (alpha2)           at (150:{sqrt(3)});
\coordinate (3alpha1+alpha2)   at (30:{sqrt(3)});
\coordinate (3alpha1+2alpha2)  at (90:{sqrt(3)});
\coordinate (-alpha2)          at (330:{sqrt(3)});
\coordinate (-3alpha1-alpha2)  at (210:{sqrt(3)});
\coordinate (-3alpha1-2alpha2) at (270:{sqrt(3)});

\draw[guide] (-alpha1) -- (alpha1);
\draw[guide] (-alpha2) -- (alpha2);
\draw[guide] (-alpha1-alpha2) -- (alpha1+alpha2);
\draw[guide] (-2alpha1-alpha2) -- (2alpha1+alpha2);
\draw[guide] (-3alpha1-alpha2) -- (3alpha1+alpha2);
\draw[guide] (-3alpha1-2alpha2) -- (3alpha1+2alpha2);

\foreach \p in {%
    alpha1,alpha2,alpha1+alpha2,2alpha1+alpha2,%
    3alpha1+alpha2,3alpha1+2alpha2,%
    -alpha1,-alpha2,-alpha1-alpha2,-2alpha1-alpha2,%
    -3alpha1-alpha2,-3alpha1-2alpha2%
}{
    \node[root] at (\p) {};
}

\draw[->,thick] (0,0) -- (alpha1);
\draw[->,thick] (0,0) -- (alpha2);
\fill (0,0) circle (0.6pt);
\node[below right=1pt] at (0,0) {\scriptsize $0$};

\node[right=3pt]        at (alpha1)          {$\alpha_1$};
\node[above left=3pt]   at (alpha2)          {$\alpha_2$};
\node[above left=3pt]   at (alpha1+alpha2)   {$\alpha_1+\alpha_2$};
\node[above right=3pt]  at (2alpha1+alpha2)  {$2\alpha_1+\alpha_2$};
\node[above right=3pt]  at (3alpha1+alpha2)  {$3\alpha_1+\alpha_2$};
\node[above=3pt]        at (3alpha1+2alpha2) {$3\alpha_1+2\alpha_2$};

\node[left=3pt]         at (-alpha1)         {$-\alpha_1$};
\node[below right=3pt]  at (-alpha2)         {$-\alpha_2$};
\node[below right=3pt]  at (-alpha1-alpha2)  {$-\alpha_1-\alpha_2$};
\node[below left=3pt]   at (-2alpha1-alpha2) {$-2\alpha_1-\alpha_2$};
\node[below left=3pt]   at (-3alpha1-alpha2) {$-3\alpha_1-\alpha_2$};
\node[below=3pt]        at (-3alpha1-2alpha2){$-3\alpha_1-2\alpha_2$};

\end{tikzpicture}

\smallskip
{\footnotesize The root system of type $G_2$.}
\end{center}
\medskip
Set
\begin{equation*}
\mathfrak q
:=\mathfrak{b}(\Pi)=
\mathfrak h\oplus
 {\mathfrak g}^{10}\oplus
 {\mathfrak g}^{01}\oplus
 {\mathfrak g}^{11}\oplus
 {\mathfrak g}^{21}\oplus
 {\mathfrak g}^{31}\oplus
 {\mathfrak g}^{32}.
\end{equation*}
For each reflection \(s_{\alpha}\), we set $\overline{\mathfrak q}
:=s_{\alpha}(\mathfrak q)
=
\mathfrak h\oplus
\bigoplus_{\gamma\in s_{\alpha}(\Delta^+)}
 {\mathfrak g}_{\gamma}.$ The $CR$ type is computed from $\mathfrak q\cap\overline{\mathfrak q}
=
\mathfrak h\oplus
\bigoplus_{\gamma\in \Delta^+\cap s_{\alpha}(\Delta^+)}
 {\mathfrak g}_{\gamma}.$ Thus, to shorten the computation if $c=|\Delta^+\cap s_{\alpha}(\Delta^+)|,$
then $\dim_{\rm CR}=6-c,
$ and $
\operatorname{codim}_{\rm CR}=c=\dim_{\mathbb C} {\mathfrak g}/(\mathfrak q+\overline{\mathfrak q}
).$ Indeed, $c$ counts only the root spaces in
\(\mathfrak q\cap\overline{\mathfrak q}\), while the Cartan subalgebra
\(\mathfrak h\) has complex dimension \(2\) and is contained in every
\(\mathfrak q\cap\overline{\mathfrak q}\), since  \(\dim_{\mathbb C}\mathfrak q=8\) and
\(\dim_{\mathbb C}(\mathfrak q\cap\overline{\mathfrak q})=2+c\), one has $\dim_{\mathbb C}(\mathfrak q+\overline{\mathfrak q})=14-c.$ We use the notation of \eqref{eq:Freeman-filtration-CRalg} for the decreasing sequence and carry out the computations for the eight cases corresponding to Weyl involutions as follows:
\begin{enumerate}
\item[($\mathrm{I}$)]
In this case $\mathrm{I}(\Delta^+)=\Delta^+.
$ Hence $\mathfrak q\cap\overline{\mathfrak q}=\mathfrak q.$ The $CR$ type is $(0,6).$ This is the compact closed orbit. It is totally real, and the sequence \eqref{eq:Freeman-filtration-CRalg} is constant $\mathfrak q^{(r)}=\mathfrak q
$ for all $ r\geq 0.$
\item[\textup{(-\textrm{I})}]
In this case $-\mathrm{I}(\Delta^+)=-\Delta^+.$ Hence $\mathfrak q\cap\overline{\mathfrak q}=\mathfrak h,
 $ and $
\mathfrak q+\overline{\mathfrak q}= {\mathfrak g}.$ The $CR$ type is $(6,0).$ This is the open, totally complex case. 
\item[($s_{\alpha_1}$)]One has $s_{\alpha_1}(\Delta^+)=
\{3\alpha_1+2\alpha_2,3\alpha_1+\alpha_2,2\alpha_1+\alpha_2,
\alpha_1+\alpha_2,\alpha_2,-\alpha_1\}.$ Thus $\mathfrak q\cap\overline{\mathfrak q}
=
\mathfrak h\oplus
 {\mathfrak g}^{32}\oplus
 {\mathfrak g}^{31}\oplus
 {\mathfrak g}^{21}\oplus
 {\mathfrak g}^{11}\oplus
 {\mathfrak g}^{01}.$ The $CR$ type is $(1,5),$ so the corresponding orbits have real dimension $7$. They are Levi-flat and non-fundamental. The sequence \eqref{eq:Freeman-filtration-CRalg} is constant  $\mathfrak q^{(r)}=\mathfrak q$ for all  $r\geq 0.$
\item[($s_{\alpha_2}$)] One has $s_{\alpha_2}(\Delta^+)=
\{3\alpha_1+2\alpha_2,3\alpha_1+\alpha_2,2\alpha_1+\alpha_2,
\alpha_1+\alpha_2,\alpha_1,-\alpha_2\}.$ Thus $\mathfrak q\cap\overline{\mathfrak q}
=
\mathfrak h\oplus
 {\mathfrak g}^{32}\oplus
 {\mathfrak g}^{31}\oplus
 {\mathfrak g}^{21}\oplus
 {\mathfrak g}^{11}\oplus
 {\mathfrak g}^{10}.$ The $CR$ type is $(1,5),$ so the corresponding orbits again have real dimension $7$. They are
Levi-flat and non-fundamental. The sequence \eqref{eq:Freeman-filtration-CRalg} is constant $\mathfrak q^{(r)}=\mathfrak q
$ for all $ r\geq 0.$
\item[($s_{\alpha_1+\alpha_2}$)]  One has $s_{\alpha_1+\alpha_2}(\Delta^+)=
\{3\alpha_1+\alpha_2,2\alpha_1+\alpha_2,\alpha_1,
-\alpha_2,-\alpha_1-\alpha_2,-3\alpha_1-2\alpha_2\}.$ Thus $\mathfrak q\cap\overline{\mathfrak q}
=
\mathfrak h\oplus
 {\mathfrak g}^{31}\oplus
 {\mathfrak g}^{21}\oplus
 {\mathfrak g}^{10}.$
The $CR$ type is $(3,3),$ so the corresponding orbits have real dimension $9$. They are fundamental  and holomorphically
degenerate. The sequence \eqref{eq:Freeman-filtration-CRalg} is $\mathfrak q^{(r)}=\mathfrak q^{(1)}$ for all $ r\geq 1,$ where $\mathfrak q^{(1)}=
\mathfrak h\oplus
 {\mathfrak g}^{32}\oplus
 {\mathfrak g}^{31}\oplus
 {\mathfrak g}^{21}\oplus
 {\mathfrak g}^{10}.$
\item[($s_{2\alpha_1+\alpha_2}$)] One has $s_{2\alpha_1+\alpha_2}(\Delta^+)=
\{\alpha_2,-\alpha_1,-\alpha_1-\alpha_2,-2\alpha_1-\alpha_2,
-3\alpha_1-\alpha_2,-3\alpha_1-2\alpha_2\}.$ Thus $\mathfrak q\cap\overline{\mathfrak q}
=
\mathfrak h\oplus {\mathfrak g}^{01}.$ The $CR$ type is $(5,1),$ so the corresponding orbits are real hypersurfaces of dimension $11$. They
are fundamental and holomorphically degenerate. The sequence \eqref{eq:Freeman-filtration-CRalg} is
\[
\mathfrak q^{(1)}=
\mathfrak h\oplus
 {\mathfrak g}^{32}\oplus
 {\mathfrak g}^{21}\oplus
 {\mathfrak g}^{11}\oplus
 {\mathfrak g}^{01},
\]
\[
\mathfrak q^{(2)}=
\mathfrak h\oplus
 {\mathfrak g}^{32}\oplus
 {\mathfrak g}^{11}\oplus
 {\mathfrak g}^{01},
\]
and
\[
\mathfrak q^{(3)}=
\mathfrak h\oplus
 {\mathfrak g}^{11}\oplus
 {\mathfrak g}^{01},
\qquad
\mathfrak q^{(r)}=\mathfrak q^{(3)}
\quad \text{for all } r\geq 3.
\]
\item[($s_{3\alpha_1+\alpha_2}$)] One has
\[
s_{3\alpha_1+\alpha_2}(\Delta^+)=
\{3\alpha_1+2\alpha_2,\alpha_1+\alpha_2,\alpha_2,
-\alpha_1,-2\alpha_1-\alpha_2,-3\alpha_1-\alpha_2\}.
\]
Thus $\mathfrak q\cap\overline{\mathfrak q}
=
\mathfrak h\oplus
 {\mathfrak g}^{32}\oplus
 {\mathfrak g}^{11}\oplus
 {\mathfrak g}^{01}.$
The $CR$ type is $(3,3),$ so the corresponding orbits have real dimension $9$. They are fundamental  and  holomorphically
degenerate. The sequence \eqref{eq:Freeman-filtration-CRalg} is
\[
\mathfrak q^{(1)}=
\mathfrak h\oplus
 {\mathfrak g}^{32}\oplus
 {\mathfrak g}^{31}\oplus
 {\mathfrak g}^{21}\oplus
 {\mathfrak g}^{11}\oplus
 {\mathfrak g}^{01},
\]
\[
\mathfrak q^{(2)}=
\mathfrak h\oplus
 {\mathfrak g}^{32}\oplus
 {\mathfrak g}^{21}\oplus
 {\mathfrak g}^{11}\oplus
 {\mathfrak g}^{01},
\qquad
\mathfrak q^{(r)}=\mathfrak q^{(2)}
\quad \text{for all } r\geq 2.
\]
\item[($s_{3\alpha_1+2\alpha_2}$)]  One has $s_{3\alpha_1+2\alpha_2}(\Delta^+)=
\{\alpha_1,-\alpha_2,-\alpha_1-\alpha_2,-2\alpha_1-\alpha_2,
-3\alpha_1-\alpha_2,-3\alpha_1-2\alpha_2\}.$
Thus $\mathfrak q\cap\overline{\mathfrak q}
=
\mathfrak h\oplus {\mathfrak g}^{10}.$
The $CR$ type is $(5,1),$ so the corresponding orbits are real hypersurfaces of dimension $11$. They
are fundamental and holomorphically degenerate. The sequence \eqref{eq:Freeman-filtration-CRalg} is
\[
\mathfrak q^{(1)}=
\mathfrak h\oplus
 {\mathfrak g}^{32}\oplus
 {\mathfrak g}^{31}\oplus
 {\mathfrak g}^{10},
\]
and
\[
\mathfrak q^{(2)}=
\mathfrak h\oplus
 {\mathfrak g}^{31}\oplus
 {\mathfrak g}^{10},
\qquad
\mathfrak q^{(r)}=\mathfrak q^{(2)}
\quad \text{for all } r\geq 2.
\]
\end{enumerate}
Therefore, together with the compact totally real orbit and the open totally complex orbits
(see Section~\ref{sec:explicit-G2-B} for the count of the ten orbits), the complete flag
manifold contains Levi-flat non-fundamental orbits of CR type $(1,5)$, holomorphically
degenerate fundamental orbits of CR type $(3,3)$, and holomorphically degenerate hypersurface
orbits of CR type $(5,1)$.
Following what we have seen so far and the above computation we have
\begin{lemma}
The $CR$ types and degeneracy properties of the ten $G_{2(2)}$-orbits on the complete flag manifold $G_2/B$ are summarized in Table~\ref{tab:complete-flag-coarse-types-count}, where each row corresponds to an involution of Table~\ref{tab:involutions-simple-roots}. 
\begin{table}[H]
\centering
\small
\renewcommand{\arraystretch}{1.15}
\begin{tabular}{c|c|l|c}
involution & $CR$ type & $CR$ structure & number of orbits\\
\hline
$I$ & $(0,6)$ & closed, totally real & $1$\\
$-I$ & $(6,0)$ & open, totally complex & $3$\\
$s_{\alpha_1}$ & $(1,5)$ & Levi-flat, non-fundamental & $1$\\
$s_{\alpha_2}$ & $(1,5)$ & Levi-flat, non-fundamental & $1$\\
$s_{\alpha_1+\alpha_2}$ & $(3,3)$ & fundamental, holomorphically degenerate & $1$\\
$s_{2\alpha_1+\alpha_2}$ & $(5,1)$ & fundamental, holomorphically degenerate & $1$\\
$s_{3\alpha_1+\alpha_2}$ & $(3,3)$ & fundamental, holomorphically degenerate & $1$\\
$s_{3\alpha_1+2\alpha_2}$ & $(5,1)$ & fundamental, holomorphically degenerate & $1$
\end{tabular}
\caption{$CR$ algebra types for the $G_{2(2)}$-orbits on the complete flag manifold $G/B$.}
\label{tab:complete-flag-coarse-types-count}
\end{table}
Moreover, different involutions give pairwise non-isomorphic $CR$ algebras.
\end{lemma}
\begin{proof}
The $CR$ types and the properties in the table are those computed above 
for each involution. For the last assertion, an isomorphism of $CR$ 
algebras $f:(\mathfrak g_0,\mathfrak q)\to(\mathfrak g_0',\mathfrak 
q')$ extends to an automorphism $\tilde f$ of $\mathfrak g$ with 
$\tilde f(\mathfrak q)=\mathfrak q'$ (see \cite{MedoriNacinovich2005}), which commutes with the conjugations , hence $\tilde 
f(\overline{\mathfrak q})=\overline{\mathfrak q'}$ and $\tilde 
f(\mathfrak q^{(r)})=\mathfrak q'^{(r)}$ for every $r$. Therefore the 
$CR$ type and the dimensions of the iterated Levi kernels are 
invariants of the $CR$ algebra. Involutions with different $CR$ type 
are thus distinguished, and it remains to consider the pairs $\{s_{\alpha_1+\alpha_2},s_{3\alpha_1+\alpha_2}\}$, $\{s_{2\alpha_1+\alpha_2},s_{3\alpha_1+2\alpha_2}\}$ and $\{s_{\alpha_1},s_{\alpha_2}\}$. The first two pairs are distinguished by $\dim\mathcal K^{(1)}$, which is $1$ for $s_{\alpha_1+\alpha_2}$ and $2$ for $s_{3\alpha_1+\alpha_2}$, and $3$ for $s_{2\alpha_1+\alpha_2}$ 
and $2$ for $s_{3\alpha_1+2\alpha_2}$. For the pair $\{s_{\alpha_1},s_{\alpha_2}\}$ the sequence \eqref{eq:Freeman-filtration-CRalg} is constant, $\mathfrak q^{(r)}=\mathfrak q$, so these 
invariants coincide. Write $\mathfrak q=\mathfrak h\oplus\mathfrak q^{n}$ 
with $\mathfrak q^n=\bigoplus_{\gamma\in\Delta^+}\mathfrak g_\gamma$, let 
$p:\mathfrak q\to\mathfrak h$ be the projection, and set $\overline{\mathfrak q}_i=s_{\alpha_i}(\mathfrak q)$ and $\mathfrak i_i=\mathfrak q\cap\overline{\mathfrak q}_i=\mathfrak h\oplus\bigoplus_{\gamma\in\Delta^+\setminus\{\alpha_i\}}\mathfrak g_\gamma$. Suppose that $f$ is an isomorphism between the two $CR$ 
algebras. Then $\tilde f(\mathfrak q)=\mathfrak q$ and $\tilde f(\mathfrak i_1)=\mathfrak i_2$.
Every automorphism of $\mathfrak g_2$ is inner, since $\operatorname{Aut}(\mathfrak g_2)=\operatorname{Int}(\mathfrak g_2)\rtimes\Gamma$, where $\Gamma$ is the group of automorphisms of the Dynkin diagram \cite[\S16.5]{Humphreys1972}, which is trivial for $G_2$ because its two nodes correspond to roots of different lengths; moreover, the normalizer of $\mathfrak q$ in $G_2$ is $B$.
Hence $\tilde f=\operatorname{Ad}(\beta)$ with $\beta\in B$. As 
$\operatorname{Ad}(\beta)Y-Y\in\mathfrak q^n$ for $Y\in\mathfrak q$, we 
have $p\circ\tilde f=p$. 
The quotient $\mathfrak q/\mathfrak i_i$ is 
spanned by the class of $X_{\alpha_i}$, and for $Y\in\mathfrak i_i$ one 
has $[Y,X_{\alpha_i}]\equiv\alpha_i(p(Y))X_{\alpha_i}$ modulo 
$\mathfrak i_i$. The map induced by $\tilde f$ between these quotients 
intertwines the actions of $\mathfrak i_1$ and $\mathfrak i_2$, so 
$\alpha_2(p(\tilde f(Y)))=\alpha_1(p(Y))$ for every $Y\in\mathfrak i_1$.
Since $p\circ\tilde f=p$, we have $p(\tilde f(Y))=p(Y)$, so that $\alpha_2(p(Y))=\alpha_1(p(Y))$ for every $Y\in\mathfrak i_1$. As $\mathfrak i_1\supset\mathfrak h$ and $p$ is the identity on $\mathfrak h$, we get $p(\mathfrak i_1)=\mathfrak h$, hence $\alpha_1(H)=\alpha_2(H)$ for every $H\in\mathfrak h$, that is, $\alpha_1=\alpha_2$. This is a contradiction, since $\alpha_1$ and $\alpha_2$ have distinct length. Finally, for the involution $-I$ there are three open, totally
complex orbits, by Lemma~\ref{thm:orbit-count-complete-flag}.
As open subsets of the complex manifold $G_2/B$, they are locally
$CR$ equivalent. Their associated $CR$ algebras, however, are pairwise non-isomorphic. Indeed, if $f$ is an isomorphism between $(\mathfrak g_{2(2)},\mathfrak q_x)$ and $(\mathfrak g_{2(2)},\mathfrak q_y)$, then $\tilde f=\operatorname{Ad}(g)$ with $g\in G_2$ and $\tilde f(\mathfrak q_x)=\mathfrak q_y$. Since $\tilde f$ commutes with $\tilde\sigma$, $\tilde\sigma(g)^{-1}g\in Z(G_2)=\{e\}$, so $g\in G_2^{\tilde\sigma}=G_{2(2)}$, the real points of the simply connected group $G_2$ being connected. As $G_2/B$ is the space of Borel subalgebras, $\operatorname{Ad}(g)\mathfrak q_x=\mathfrak q_y$ gives $y=g\cdot x$, so $x$ and $y$ lie in the same $G_{2(2)}$-orbit.
\end{proof}
\subsection{$CR$ structures on the $G_{2(2)}$-orbits on $G_{2}/Q_{\Pi_0}$}
We now consider the following two distinct  cases ${\Pi_0}\subset \Pi$.\\
$\boxed{\Pi_0 = \{\alpha_1\}}$ We start with the flag $F_{\alpha_1}{=} G/Q_{\alpha_1}$.  for this maximal parabolic the crossed Dynkin diagram is
$$\GTwoCrossedDiagram{0}{0}{1}{0}$$
We keep the positive Borel fixed. Hence $\mathcal Q_{\alpha_1}:=\Delta^+\cup\{-\alpha_2\},$ where $\Delta^+=
\{\alpha_1,\alpha_2,\alpha_1+\alpha_2,2\alpha_1+\alpha_2,
3\alpha_1+\alpha_2,3\alpha_1+2\alpha_2\}.$
The corresponding parabolic subalgebra is
\[
\mathfrak q_{\alpha_1}=
\mathfrak h\oplus
 {\mathfrak g}^{32}\oplus
 {\mathfrak g}^{31}\oplus
 {\mathfrak g}^{21}\oplus
 {\mathfrak g}^{11}\oplus
 {\mathfrak g}^{10}\oplus
 {\mathfrak g}^{01}\oplus
 {\mathfrak g}^{-01}.
\]
For each involution \(s_\alpha\) set
$\overline{\mathfrak q}_{\alpha_1}:=s_\alpha(\mathfrak q_{\alpha_1})=
\mathfrak h\oplus
\bigoplus_{\gamma\in s_\alpha(\mathcal Q_{\alpha_1})} {\mathfrak g}_{\gamma}.$
The $CR$ type is computed from $\mathfrak q_{\alpha_1}\cap\overline{\mathfrak q}_{\alpha_1}
=
\mathfrak h\oplus
\bigoplus_{\gamma\in \mathcal Q_{\alpha_1}\cap s_\alpha(\mathcal Q_{\alpha_1})} {\mathfrak g}_{\gamma}.$
Thus, if $c=|\mathcal Q_{\alpha_1}\cap s_\alpha(\mathcal Q_{\alpha_1})|,$ then $\dim_{\rm CR}\mathcal{O}=7-c,
$ and $
\operatorname{codim}_{\rm CR}\mathcal{O}=c-2.$ Again, the eight cases coming from the involutions of Table~\ref{tab:involutions-simple-roots} are as follows
\begin{enumerate}
\item[(\textrm{I})]  One has $\mathrm{I}(\mathcal Q_{\alpha_1})=
\{3\alpha_1+2\alpha_2, 3\alpha_1+\alpha_2, 2\alpha_1+\alpha_2, \alpha_1+\alpha_2, \alpha_2, \alpha_1, -\alpha_2\}.$ Therefore $\mathfrak q_{\alpha_1}\cap \overline{\mathfrak q}_{\alpha_1}
=
\mathfrak h\oplus
 {\mathfrak g}^{32}\oplus
 {\mathfrak g}^{31}\oplus
 {\mathfrak g}^{21}\oplus
 {\mathfrak g}^{11}\oplus
 {\mathfrak g}^{01}\oplus
 {\mathfrak g}^{10}\oplus
 {\mathfrak g}^{-01}.$ The $CR$ type is $(0,5),$ and the corresponding real dimension is $5$.
This is a closed totally real case. The sequence \eqref{eq:Freeman-filtration-CRalg} is constant $\mathfrak q_{\alpha_1}^{(r)}=\mathfrak q_{\alpha_1}
$ for all $ r\geq 0.$
\item[( -\textrm{I})]  One has $-\mathrm{I}(\mathcal Q_{\alpha_1})=
\{\alpha_2, -\alpha_1, -\alpha_2, -\alpha_1-\alpha_2, -2\alpha_1-\alpha_2, -3\alpha_1-\alpha_2, -3\alpha_1-2\alpha_2\}.$ Therefore $\mathfrak q_{\alpha_1}\cap \overline{\mathfrak q}_{\alpha_1}
=
\mathfrak h\oplus
 {\mathfrak g}^{01}\oplus
 {\mathfrak g}^{-01}.$
The $CR$ type is $(5,0),$ and the corresponding real dimension is $10$. This is an open totally complex case.
\item[($s_{\alpha_1}$)] One has $s_{\alpha_1}(\mathcal Q_{\alpha_1})=
\{3\alpha_1+2\alpha_2, 3\alpha_1+\alpha_2, 2\alpha_1+\alpha_2, \alpha_1+\alpha_2, \alpha_2, -\alpha_1, -3\alpha_1-\alpha_2\}.$ Therefore $\mathfrak q_{\alpha_1}\cap \overline{\mathfrak q}_{\alpha_1}
=
\mathfrak h\oplus
 {\mathfrak g}^{32}\oplus
 {\mathfrak g}^{31}\oplus
 {\mathfrak g}^{21}\oplus
 {\mathfrak g}^{11}\oplus
 {\mathfrak g}^{01}.$ The $CR$ type is $(2,3),$ and the corresponding real dimension is $7$. The sequence \eqref{eq:Freeman-filtration-CRalg} reaches the intersection at step $1$, so the germ is $1$-nondegenerate.$\mathfrak q_{\alpha_1}^{(1)}=
\mathfrak h\oplus
 {\mathfrak g}^{32}\oplus
 {\mathfrak g}^{31}\oplus
 {\mathfrak g}^{21}\oplus
 {\mathfrak g}^{11}\oplus
 {\mathfrak g}^{01}.$ Moreover, $\mathfrak q_{\alpha_1}^{(r)}=\mathfrak q_{\alpha_1}^{(1)}
$ for all $ r\geq 1.$
\item[($s_{\alpha_2}$)] One has $s_{\alpha_2}(\mathcal Q_{\alpha_1})=
\{3\alpha_1+2\alpha_2, 3\alpha_1+\alpha_2, 2\alpha_1+\alpha_2, \alpha_1+\alpha_2, \alpha_2, \alpha_1, -\alpha_2\}.$ Therefore $\mathfrak q_{\alpha_1}\cap \overline{\mathfrak q}_{\alpha_1}
=
\mathfrak h\oplus
 {\mathfrak g}^{32}\oplus
 {\mathfrak g}^{31}\oplus
 {\mathfrak g}^{21}\oplus
 {\mathfrak g}^{11}\oplus
 {\mathfrak g}^{01}\oplus
 {\mathfrak g}^{10}\oplus
 {\mathfrak g}^{-01}.$
The $CR$ type is $(0,5),$ and the corresponding real dimension is $5$.
This is a closed totally real case. The sequence \eqref{eq:Freeman-filtration-CRalg} is constant $\mathfrak q_{\alpha_1}^{(r)}=\mathfrak q_{\alpha_1}
$ for all $r\geq 0.$
\item[($s_{\alpha_1+\alpha_2}$)]  One has $s_{\alpha_1+\alpha_2}(\mathcal Q_{\alpha_1})=
\{3\alpha_1+2\alpha_2, 3\alpha_1+\alpha_2, 2\alpha_1+\alpha_2, \alpha_1, -\alpha_2, -\alpha_1-\alpha_2, -3\alpha_1-2\alpha_2\}.$ Therefore $\mathfrak q_{\alpha_1}\cap \overline{\mathfrak q}_{\alpha_1}
=
\mathfrak h\oplus
 {\mathfrak g}^{32}\oplus
 {\mathfrak g}^{31}\oplus
 {\mathfrak g}^{21}\oplus
 {\mathfrak g}^{10}\oplus
 {\mathfrak g}^{-01}.$ The $CR$ type is $(2,3),$ and the corresponding real dimension is $7$. The sequence \eqref{eq:Freeman-filtration-CRalg} reaches the intersection at step $1$, so the $CR$ algebra is $1$-nondegenerate. $\mathfrak q_{\alpha_1}^{(1)}=
\mathfrak h\oplus
 {\mathfrak g}^{32}\oplus
 {\mathfrak g}^{31}\oplus
 {\mathfrak g}^{21}\oplus
 {\mathfrak g}^{10}\oplus
 {\mathfrak g}^{-01}.$ Moreover, $\mathfrak q_{\alpha_1}^{(r)}=\mathfrak q_{\alpha_1}^{(1)}
$ for all $ r\geq 1.$
\item[ ($s_{2\alpha_1+\alpha_2}$)]  One has $s_{2\alpha_1+\alpha_2}(\mathcal Q_{\alpha_1})=\{\alpha_2, -\alpha_1, -\alpha_2, -\alpha_1-\alpha_2, -2\alpha_1-\alpha_2, -3\alpha_1-\alpha_2, -3\alpha_1-2\alpha_2\}.$ Therefore $\mathfrak q_{\alpha_1}\cap \overline{\mathfrak q}_{\alpha_1}
=\mathfrak h\oplus{\mathfrak g}^{01}\oplus
 {\mathfrak g}^{-01}.$ The $CR$ type is $(5,0),$ and the corresponding real dimension is $10$.
This is an open totally complex case.
\item[($s_{3\alpha_1+\alpha_2}$)]  One has $s_{3\alpha_1+\alpha_2}(\mathcal Q_{\alpha_1})=\{3\alpha_1+2\alpha_2, \alpha_1+\alpha_2, \alpha_2, -\alpha_1, -2\alpha_1-\alpha_2, -3\alpha_1-\alpha_2, -3\alpha_1-2\alpha_2\}.$ Therefore $\mathfrak q_{\alpha_1}\cap \overline{\mathfrak q}_{\alpha_1}
=
\mathfrak h\oplus
 {\mathfrak g}^{32}\oplus
 {\mathfrak g}^{11}\oplus
 {\mathfrak g}^{01}.$
The $CR$ type is $(4,1),$ and the corresponding real dimension is $9$.
This is a hypersurface case. The sequence \eqref{eq:Freeman-filtration-CRalg} reaches the intersection at step $2$, so the $CR$ algebra is $2$-nondegenerate.
\[
\mathfrak q_{\alpha_1}^{(1)}=
\mathfrak h\oplus
 {\mathfrak g}^{32}\oplus
 {\mathfrak g}^{31}\oplus
 {\mathfrak g}^{11}\oplus
 {\mathfrak g}^{01}.
\]
\[
\mathfrak q_{\alpha_1}^{(2)}=
\mathfrak h\oplus
 {\mathfrak g}^{32}\oplus
 {\mathfrak g}^{11}\oplus
 {\mathfrak g}^{01}.
\]
Moreover,
\[
\mathfrak q_{\alpha_1}^{(r)}=\mathfrak q_{\alpha_1}^{(2)}=\mathfrak q_{\alpha_1}\cap \overline{\mathfrak q}_{\alpha_1}
\qquad \text{for all } r\geq 2.
\]
\item[($s_{3\alpha_1+2\alpha_2}$)]One has $s_{3\alpha_1+2\alpha_2}(\mathcal Q_{\alpha_1})=
\{3\alpha_1+\alpha_2, \alpha_1, -\alpha_2, -\alpha_1-\alpha_2, -2\alpha_1-\alpha_2, -3\alpha_1-\alpha_2, -3\alpha_1-2\alpha_2\}.$
Therefore $\mathfrak q_{\alpha_1}\cap \overline{\mathfrak q}_{\alpha_1}
=
\mathfrak h\oplus
 {\mathfrak g}^{31}\oplus
 {\mathfrak g}^{10}\oplus
 {\mathfrak g}^{-01}.$ The $CR$ type is $(4,1),$ and the corresponding real dimension is $9$. This is a hypersurface case. The sequence \eqref{eq:Freeman-filtration-CRalg} reaches the intersection at step $2$, so the $CR$ algebra is $2$-nondegenerate and holomorphically nondegenerate.
\[
\mathfrak q_{\alpha_1}^{(1)}=
\mathfrak h\oplus
 {\mathfrak g}^{32}\oplus
 {\mathfrak g}^{31}\oplus
 {\mathfrak g}^{10}\oplus
 {\mathfrak g}^{-01}.
\]
\[
\mathfrak q_{\alpha_1}^{(2)}=
\mathfrak h\oplus
 {\mathfrak g}^{31}\oplus
 {\mathfrak g}^{10}\oplus
 {\mathfrak g}^{-01}.
\]
Moreover,
\[
\mathfrak q_{\alpha_1}^{(r)}=\mathfrak q_{\alpha_1}^{(2)}=\mathfrak q_{\alpha_1}\cap \overline{\mathfrak q}_{\alpha_1}
\qquad \text{for all } r\geq 2.
\]
\end{enumerate}
The eight models arising from the Weyl involutions represent only five
distinct $G_{2(2)}$-orbits. Indeed, by
Section~\ref{counting_maximal flags}, $F_{\alpha_1}$ contains a unique
closed orbit of real dimension $5$, two open orbits of real dimension
$10$, and one orbit of each of the real dimensions $7$ and $9$.
Consequently, the totally real models associated with $\mathrm I$ and
$s_{\alpha_2}$ represent the closed orbit; the models associated with
$s_{\alpha_1}$ and $s_{\alpha_1+\alpha_2}$ represent the same
$1$-nondegenerate orbit of $CR$ type $(2,3)$; and those associated with
$s_{3\alpha_1+\alpha_2}$ and $s_{3\alpha_1+2\alpha_2}$ represent the
same $2$-nondegenerate hypersurface orbit of $CR$ type $(4,1)$. 
The remaining cases, associated with $-I$ and
$s_{2\alpha_1+\alpha_2}$, have $CR$ type $(5,0)$ and hence
correspond to open orbits. This data alone
does not distinguish the two open $G_{2(2)}$-orbits on
$F_{\alpha_1}$ as shown in Section~3.1.
The two nontrivial identifications can also be realized explicitly.
Let $h\in G_{2(2)}$ be the element defined in
\eqref{eq:alpha1-conjugating-element}, already used in the projection
onto $F_{\alpha_1}$. Since $h\in G_{2(2)}\cap Q_{\alpha_1}$ and
$h\gamma_7h^{-1}=\gamma_6$, $h\gamma_3h^{-1}=\gamma_4$, writing
$\mathfrak q_i:=\operatorname{Ad}(\gamma_i)\mathfrak q_{\alpha_1}$
gives $\operatorname{Ad}(h)\mathfrak q_7=\mathfrak q_6$ and
$\operatorname{Ad}(h)\mathfrak q_3=\mathfrak q_4$. Thus the same real
inner automorphism realizes both identifications described above.
The resulting five geometrically distinct orbits are summarized in
the following table.
\begin{table}[h]
\centering
\small
\renewcommand{\arraystretch}{1.15}
\begin{tabular}{c|l|c}
$CR$ type 
&
$CR$ structure
&
number of orbits
\\
\hline
$(0,5)$
&
closed, totally real
&
$1$
\\
$(2,3)$
&
fundamental, $1$-nondegenerate
&
$1$
\\
$(4,1)$
&
fundamental, $2$-nondegenerate
&
$1$
\\
$(5,0)$
&
open, totally complex
&
$2$
\end{tabular}
\caption{$G_{2(2)}$-orbits on $F_{\alpha_1}=G/Q_{\alpha_1}$.}
\label{tab:Xalpha-CR-summary}
\end{table}
\par
\boxed{{\Pi_0}=\{\alpha_2\}} We move to the partial flag $F_{\alpha_2}$.
For this case the maximal parabolic has crossed Dynkin diagram $$\GTwoCrossedDiagram{0}{0}{0}{1}$$
We still keep the positive Borel fixed. Hence $\mathcal Q_{\alpha_2}=\Delta^+\cup\{-\alpha_1\},$ where
$\Delta^+=
\{\alpha_1,\alpha_2,\alpha_1+\alpha_2,2\alpha_1+\alpha_2,
3\alpha_1+\alpha_2,3\alpha_1+2\alpha_2\}.$
The corresponding parabolic subalgebra is
\[
\mathfrak q_{\alpha_2}=
\mathfrak h\oplus
 {\mathfrak g}^{32}\oplus
 {\mathfrak g}^{31}\oplus
 {\mathfrak g}^{21}\oplus
 {\mathfrak g}^{11}\oplus
 {\mathfrak g}^{01}\oplus
 {\mathfrak g}^{10}\oplus
 {\mathfrak g}^{-10}.
\]
For an involution \(s_\alpha\), set $\overline{\mathfrak q}_{\alpha_2}:=s_\alpha(\mathfrak q_{\alpha_2})=
\mathfrak h\oplus
\bigoplus_{\gamma\in s_\alpha(\mathcal Q_{\alpha_2})} {\mathfrak g}_{\gamma}.$ The $CR$ type is computed from $\mathfrak q_{\alpha_2}\cap\overline{\mathfrak q}_{\alpha_2}
=
\mathfrak h\oplus
\bigoplus_{\gamma\in \mathcal Q_{\alpha_2}\cap s_\alpha(\mathcal Q_{\alpha_2})} {\mathfrak g}_{\gamma}.$ Thus, if $c=|\mathcal Q_{\alpha_2}\cap s_{\alpha}(\mathcal Q_{\alpha_2})|,$ then $\dim_{\rm CR}=7-c,
$ and $\operatorname{codim}_{\rm CR}=c-2.$
\begin{enumerate}
\item[\textup{($\mathrm{I}$)}]  One has $\mathrm{I}(\mathcal Q_{\alpha_2})=
\{3\alpha_1+2\alpha_2, 3\alpha_1+\alpha_2, 2\alpha_1+\alpha_2, \alpha_1+\alpha_2, \alpha_2, \alpha_1, -\alpha_1\}.$ Therefore $\mathfrak q_{\alpha_2}\cap \overline{\mathfrak q}_{\alpha_2}
=
\mathfrak h\oplus
 {\mathfrak g}^{32}\oplus
 {\mathfrak g}^{31}\oplus
 {\mathfrak g}^{21}\oplus
 {\mathfrak g}^{11}\oplus
 {\mathfrak g}^{01}\oplus
 {\mathfrak g}^{10}\oplus
 {\mathfrak g}^{-10}.$ The $CR$ type is $(0,5),$ and the corresponding real dimension is $5$. This is a closed totally real case. The sequence \eqref{eq:Freeman-filtration-CRalg} is constant $\mathfrak q_{\alpha_2}^{(r)}=\mathfrak q_{\alpha_2}
$ for all $ r\geq 0.$
\item[($-\mathrm{I}$)]  One has $-\mathrm{I}(\mathcal Q_{\alpha_2})=
\{\alpha_1, -\alpha_1, -\alpha_2, -\alpha_1-\alpha_2, -2\alpha_1-\alpha_2, -3\alpha_1-\alpha_2, -3\alpha_1-2\alpha_2\}.$ Therefore $\mathfrak q_{\alpha_2}\cap \overline{\mathfrak q}_{\alpha_2}
=
\mathfrak h\oplus
 {\mathfrak g}^{10}\oplus
 {\mathfrak g}^{-10}.$ The $CR$ type is $(5,0),$ and the corresponding real dimension is $10$.
This is an open totally complex case.
\item[($s_{\alpha_1}$)]  One has $s_{\alpha_1}(\mathcal Q_{\alpha_2})=
\{3\alpha_1+2\alpha_2, 3\alpha_1+\alpha_2, 2\alpha_1+\alpha_2, \alpha_1+\alpha_2, \alpha_2, \alpha_1, -\alpha_1\}.$
Therefore $\mathfrak q_{\alpha_2}\cap \overline{\mathfrak q}_{\alpha_2}
=
\mathfrak h\oplus
 {\mathfrak g}^{32}\oplus
 {\mathfrak g}^{31}\oplus
 {\mathfrak g}^{21}\oplus
 {\mathfrak g}^{11}\oplus
 {\mathfrak g}^{01}\oplus
 {\mathfrak g}^{10}\oplus
 {\mathfrak g}^{-10}.$
The $CR$ type is $(0,5),$ and the corresponding real dimension is $5$.
This is a closed totally real case. This is a totally real case of type \((0,5)\). Since
\(s_{\alpha_1}(\mathcal Q_{\alpha_2})=\mathcal Q_{\alpha_2}\), it gives the same
parabolic $CR$ model as the closed case \((\rm{I})\); it should not be counted as a
second closed \(G_{2(2)}\)-orbit on \(F_{\alpha_2}\). The sequence \eqref{eq:Freeman-filtration-CRalg} is constant $\mathfrak q_{\alpha_2}^{(r)}=\mathfrak q_{\alpha_2}$ for all $ r\geq 0.$
\item[($s_{\alpha_2}$)]  One has $s_{\alpha_2}(\mathcal Q_{\alpha_2})=
\{3\alpha_1+2\alpha_2, 3\alpha_1+\alpha_2, 2\alpha_1+\alpha_2, \alpha_1+\alpha_2, \alpha_1, -\alpha_2, -\alpha_1-\alpha_2\}.$
Therefore $\mathfrak q_{\alpha_2}\cap \overline{\mathfrak q}_{\alpha_2}
=
\mathfrak h\oplus
 {\mathfrak g}^{32}\oplus
 {\mathfrak g}^{31}\oplus
 {\mathfrak g}^{21}\oplus
 {\mathfrak g}^{11}\oplus
 {\mathfrak g}^{10}.$
The $CR$ type is $(2,3),$ and the corresponding real dimension is $7$.
The sequence \eqref{eq:Freeman-filtration-CRalg} reaches the intersection at step $2$, so the germ is $2$-nondegenerate. $$\mathfrak q_{\alpha_2}^{(1)}=
\mathfrak h\oplus
 {\mathfrak g}^{32}\oplus
 {\mathfrak g}^{31}\oplus
 {\mathfrak g}^{21}\oplus
 {\mathfrak g}^{11}\oplus
 {\mathfrak g}^{01}\oplus
 {\mathfrak g}^{10}.$$
$$\mathfrak q_{\alpha_2}^{(2)}=
\mathfrak h\oplus
 {\mathfrak g}^{32}\oplus
 {\mathfrak g}^{31}\oplus
 {\mathfrak g}^{21}\oplus
 {\mathfrak g}^{11}\oplus
 {\mathfrak g}^{10}.$$
Moreover, $\mathfrak q_{\alpha_2}^{(r)}=\mathfrak q_{\alpha_2}^{(2)}
$for all $ r\geq 2.$
\item[($s_{\alpha_1+\alpha_2}$)] One has $s_{\alpha_1+\alpha_2}(\mathcal Q_{\alpha_2})=\{3\alpha_1+\alpha_2, 2\alpha_1+\alpha_2, \alpha_1, -\alpha_2, -\alpha_1-\alpha_2, -2\alpha_1-\alpha_2, -3\alpha_1-2\alpha_2\}.$
Therefore $\mathfrak q_{\alpha_2}\cap \overline{\mathfrak q}_{\alpha_2}
=
\mathfrak h\oplus
 {\mathfrak g}^{31}\oplus
 {\mathfrak g}^{21}\oplus
 {\mathfrak g}^{10}.$
The $CR$ type is $(4,1),$ and the corresponding real dimension is $9$.
This is a hypersurface case. The sequence \eqref{eq:Freeman-filtration-CRalg} reaches the intersection at step $3$, so the germ is $3$-nondegenerate and holomorphically nondegenerate.
\[
\mathfrak q_{\alpha_2}^{(1)}=
\mathfrak h\oplus
 {\mathfrak g}^{32}\oplus
 {\mathfrak g}^{31}\oplus
 {\mathfrak g}^{21}\oplus
 {\mathfrak g}^{11}\oplus
 {\mathfrak g}^{10}.
\]
\[
\mathfrak q_{\alpha_2}^{(2)}=
\mathfrak h\oplus
 {\mathfrak g}^{32}\oplus
 {\mathfrak g}^{31}\oplus
 {\mathfrak g}^{21}\oplus
 {\mathfrak g}^{10}.
\]
\[
\mathfrak q_{\alpha_2}^{(3)}=
\mathfrak h\oplus
 {\mathfrak g}^{31}\oplus
 {\mathfrak g}^{21}\oplus
 {\mathfrak g}^{10}.
\]
Moreover, $\mathfrak q_{\alpha_2}^{(r)}=\mathfrak q_{\alpha_2}^{(3)}
 $ for all $r\geq 3.$
\item[($s_{2\alpha_1+\alpha_2}$)]  One has $s_{2\alpha_1+\alpha_2}(\mathcal Q_{\alpha_2})=\{\alpha_1+\alpha_2, \alpha_2, -\alpha_1, -\alpha_1-\alpha_2, -2\alpha_1-\alpha_2, -3\alpha_1-\alpha_2, -3\alpha_1-2\alpha_2\}.$
Therefore $\mathfrak q_{\alpha_2}\cap \overline{\mathfrak q}_{\alpha_2}
=
\mathfrak h\oplus
 {\mathfrak g}^{11}\oplus
 {\mathfrak g}^{01}\oplus
 {\mathfrak g}^{-10}.$ The $CR$ type is $(4,1),$ and the corresponding real dimension is $9$.
This is a hypersurface case. The sequence \eqref{eq:Freeman-filtration-CRalg} reaches the intersection at step $3$, so the germ is $3$-nondegenerate and holomorphically nondegenerate.
\[
\mathfrak q_{\alpha_2}^{(1)}=
\mathfrak h\oplus
 {\mathfrak g}^{32}\oplus
 {\mathfrak g}^{21}\oplus
 {\mathfrak g}^{11}\oplus
 {\mathfrak g}^{01}\oplus
 {\mathfrak g}^{-10}.
\]
\[
\mathfrak q_{\alpha_2}^{(2)}=
\mathfrak h\oplus
 {\mathfrak g}^{32}\oplus
 {\mathfrak g}^{11}\oplus
 {\mathfrak g}^{01}\oplus
 {\mathfrak g}^{-10}.
\]
\[
\mathfrak q_{\alpha_2}^{(3)}=
\mathfrak h\oplus
 {\mathfrak g}^{11}\oplus
 {\mathfrak g}^{01}\oplus
 {\mathfrak g}^{-10}.
\]
Moreover, $\mathfrak q_{\alpha_2}^{(r)}=\mathfrak q_{\alpha_2}^{(3)}
$ for all $ r\geq 3.$
\item[($s_{3\alpha_1+\alpha_2}$)] One has $s_{3\alpha_1+\alpha_2}(\mathcal Q_{\alpha_2})=
\{3\alpha_1+2\alpha_2, 2\alpha_1+\alpha_2, \alpha_1+\alpha_2, \alpha_2, -\alpha_1, -2\alpha_1-\alpha_2, -3\alpha_1-\alpha_2\}.$
Therefore $\mathfrak q_{\alpha_2}\cap \overline{\mathfrak q}_{\alpha_2}
=
\mathfrak h\oplus
 {\mathfrak g}^{32}\oplus
 {\mathfrak g}^{21}\oplus
 {\mathfrak g}^{11}\oplus
 {\mathfrak g}^{01}\oplus
 {\mathfrak g}^{-10}.$
The $CR$ type is $(2,3),$ and the corresponding real dimension is $7$.
The sequence \eqref{eq:Freeman-filtration-CRalg} reaches the intersection at step $2$, so the germ is $2$-nondegenerate.
\[
\mathfrak q_{\alpha_2}^{(1)}=
\mathfrak h\oplus
 {\mathfrak g}^{32}\oplus
 {\mathfrak g}^{31}\oplus
 {\mathfrak g}^{21}\oplus
 {\mathfrak g}^{11}\oplus
 {\mathfrak g}^{01}\oplus
 {\mathfrak g}^{-10}.
\]
\[
\mathfrak q_{\alpha_2}^{(2)}=
\mathfrak h\oplus
 {\mathfrak g}^{32}\oplus
 {\mathfrak g}^{21}\oplus
 {\mathfrak g}^{11}\oplus
 {\mathfrak g}^{01}\oplus
 {\mathfrak g}^{-10}.
\]
Moreover, $\mathfrak q_{\alpha_2}^{(r)}=\mathfrak q_{\alpha_2}^{(2)}
$ for all $ r\geq 2.$
\item[($s_{3\alpha_1+2\alpha_2}$)]  One has $s_{3\alpha_1+2\alpha_2}(\mathcal Q_{\alpha_2})=
\{\alpha_1, -\alpha_1, -\alpha_2, -\alpha_1-\alpha_2, -2\alpha_1-\alpha_2, -3\alpha_1-\alpha_2, -3\alpha_1-2\alpha_2\}.$ Therefore $\mathfrak q_{\alpha_2}\cap \overline{\mathfrak q}_{\alpha_2}
=
\mathfrak h\oplus
 {\mathfrak g}^{10}\oplus
 {\mathfrak g}^{-10}.$
The $CR$ type is $(5,0),$ and the corresponding real dimension is $10$. This is an open totally complex case.
\end{enumerate}

The eight models arising from the Weyl involutions represent only five
distinct $G_{2(2)}$-orbits. Indeed, by
Section~\ref{counting_maximal flags}, $F_{\alpha_2}$ contains a unique
closed orbit of real dimension $5$, two open orbits of real dimension
$10$, and one orbit of each of the real dimensions $7$ and $9$.
Consequently, the totally real models associated with $\mathrm I$ and
$s_{\alpha_1}$ represent the closed orbit; the models associated with
$s_{\alpha_2}$ and $s_{3\alpha_1+\alpha_2}$ represent the same
$2$-nondegenerate orbit of $CR$ type $(2,3)$; and those associated with
$s_{\alpha_1+\alpha_2}$ and $s_{2\alpha_1+\alpha_2}$ represent the
same $3$-nondegenerate hypersurface orbit of $CR$ type $(4,1)$. 
The remaining root-theoretic cases, associated with $-I$ and
$s_{3\alpha_1+2\alpha_2}$, have $CR$ type $(5,0)$ and hence
correspond to open orbits. These root-theoretic data alone
do not distinguish the two open $G_{2(2)}$-orbits on
$F_{\alpha_2}$, which are characterized by the signatures
$(1,1)$ and $(0,2)$ of $q|_\pi$, respectively, as shown
in Section~3.2.
The two nontrivial identifications can also be realized explicitly.
Let $g\in G_{2(2)}$ be the element defined in
\eqref{eq:alpha2-conjugating-element}, already used in the projection
onto $F_{\alpha_2}$. Since $g\in G_{2(2)}\cap Q_{\alpha_2}$ and
$g\gamma_2g^{-1}=\gamma_3$, $g\gamma_6g^{-1}=\gamma_5$, writing
$\mathfrak q_i=\operatorname{Ad}(\gamma_i)\mathfrak q_{\alpha_2}$
gives $\operatorname{Ad}(g)\mathfrak q_2=\mathfrak q_3$ and
$\operatorname{Ad}(g)\mathfrak q_6=\mathfrak q_5$. 
\begin{remark}\label{rem:G2-three-nondegenerate-known-example}
The  $3$-nondegenerate $CR$-algebra was already presented in
\cite[Example~2.14]{MariniMedoriNacinovich2020}.
Indeed, writing the simple roots used there as
$a_1=e_1-e_2$ and $a_2=2e_2-e_1-e_3$, the conjugation
interchanging $e_1$ and $e_3$ and fixing $e_2$ induces
the reflection $s_{2a_1+a_2}$, since
$e_1-e_3=2a_1+a_2$.
The parabolic subalgebra corresponds to the crossed
long root $a_2$.
Under the identification $a_i\mapsto-\alpha_i$, which
accounts for the opposite choice of Borel subalgebra,
these data agree with
$\mathfrak q_{\alpha_2}$, and conjugation $s_{2\alpha_1+\alpha_2}$
in our notation.
\end{remark}
Thus, again  using also the result of section \ref{counting_maximal flags} for the partial parabolic flag \(  G/Q_{\alpha_2}\),
we summarize
\begin{table}[H]
\centering
\small
\renewcommand{\arraystretch}{1.15}
\begin{tabular}{c|l|c}
$CR$ type 
&
$CR$ structure
&
number of orbits
\\
\hline
$(0,5)$
&
closed, totally real
&
$1$
\\
$(2,3)$
&
fundamental, $2$-nondegenerate
&
$1$
\\
$(4,1)$
&
fundamental, $3$-nondegenerate
&
$1$
\\
$(5,0)$
&
open, totally complex
&
$2$
\end{tabular}
\caption{
$G_{2(2)}$-orbits on $F_{\alpha_2}=G/Q_{\alpha_2}$.}
\label{tab:Xbeta-CR-summary}
\end{table}
\par
\par
The computations of the sequence \eqref{eq:Freeman-filtration-CRalg} for the $CR$ algebras $(\mathfrak g_0,\mathfrak q)$ carried out above  determine the 
dimensions of the iterated Levi kernels setting  $\dim_{\mathbb C}\mathcal K^{(r)}_o=\dim_{\mathbb C}\mathfrak q^{(r)}/(\mathfrak q\cap\overline{\mathfrak q})$ for every $r\ge0$. Since the orbits are 
homogeneous, these dimensions, and hence the order of nondegeneracy, do not 
depend on the point (see Section~\ref{CR_orbits section}), so it suffices to 
compute them at the base point. On each of the flags $F_{\alpha_1}$ and 
$F_{\alpha_2}$ the closed orbit is totally real and the two open orbits are 
totally complex, hence they are of no interest for the study of $CR$ 
geometries. For the rest of the orbits  $\mathcal O_7$ and $\mathcal O_9$ of 
the two flags we have:
\begin{proposition}\label{prop:Levi-kernel-dimensions}
Given the flags $F_{\alpha_1}$ and $F_{\alpha_2}$ of Section~\ref{counting_maximal flags}, for their orbits $\mathcal O_7$ and $\mathcal O_9$ we have:
\begin{enumerate}
\item[\textup{(i)}] On $F_{\alpha_1}$, the orbit $\mathcal O_7$ is $1$-nondegenerate of $CR$ type $(2,3)$, that is, $\mathcal K^{(1)}_o=\{0\}$. The orbit $\mathcal O_9$ is $2$-nondegenerate of hypersurface type, with $CR$ dimension $4$, and $\dim_{\mathbb C}\mathcal K^{(1)}_o=1$, $\mathcal K^{(2)}_o=\{0\}$.
\item[\textup{(ii)}] On $F_{\alpha_2}$, the orbit $\mathcal O_7$ is $2$-nondegenerate of $CR$ type $(2,3)$, with $\dim_{\mathbb C}\mathcal K^{(1)}_o=1$ and $\mathcal K^{(2)}_o=\{0\}$. The orbit $\mathcal O_9$ is $3$-nondegenerate of hypersurface type, with $CR$ dimension $4$, and $\dim_{\mathbb C}\mathcal K^{(1)}_o=2$, $\dim_{\mathbb C}\mathcal K^{(2)}_o=1$, $\mathcal K^{(3)}_o=\{0\}$.
\end{enumerate}
\end{proposition}
The dimension of $\mathcal K^{(1)}_o$ in Proposition~\ref{prop:Levi-kernel-dimensions} is consistent with the Levi form at the origin computed from the local equations of the orbits given in Theorems~\ref{thm:G2-alpha1-explicit-local-models} and~\ref{thm:local-models-X-alpha2}. In each of \eqref{eq:G2-alpha1-O7-local-model}, \eqref{eq:G2-alpha1-O9-local-model}, \eqref{eq:G2-alpha2-O7-local-model} and \eqref{eq:G2-alpha2-O9-local-model} the first derivatives of the defining functions vanish at the origin, except those of $\operatorname{Im}w$. Hence the complex tangent space at the origin is $\{dw=0\}$, and the Levi form is given by the coefficients of $\xi_a\overline{\xi_b}$ in the terms of the defining functions depending only on $\xi$ (and on $\zeta$ in \eqref{eq:G2-alpha1-O9-local-model}), that is, by the Hermitian part of their quadratic terms. For \eqref{eq:G2-alpha1-O7-local-model}, the Hermitian parts of $\operatorname{Re}(\xi_1\overline{\xi_2})$ and $\operatorname{Im}(\xi_1\overline{\xi_2})$ are nondegenerate, so the Levi form has trivial kernel. For the hypersurface \eqref{eq:G2-alpha1-O9-local-model}, $\zeta$ does not appear in the quadratic part, and the Hermitian part of $(\operatorname{Im}\xi_1)^2+(\operatorname{Im}\xi_2)^2-(\operatorname{Im}\xi_3)^2$ is $\tfrac12(|\xi_1|^2+|\xi_2|^2-|\xi_3|^2)$; hence the Levi form is nondegenerate of signature $(2,1)$, up to sign, in the variables $\xi_1,\xi_2,\xi_3$, and its kernel is spanned by $\partial_\zeta$. For \eqref{eq:G2-alpha2-O7-local-model}, the second equation is of order three in $\xi$ and the third has no quadratic term depending only on $\xi$, so the Levi form is given by the Hermitian part $\tfrac14|\xi_2|^2$ of $(\operatorname{Im}\xi_2)^2/2$, and its kernel is spanned by $\partial_{\xi_1}$. For the hypersurface \eqref{eq:G2-alpha2-O9-local-model}, the quadratic part of $AB-|C|^2$ on $\{w=0\}$ is $(\operatorname{Im}\xi_2)^2-|\xi_2-\tfrac{i}{2}\operatorname{Im}\xi_2+2\operatorname{Im}\xi_4|^2$, which does not involve $\xi_1,\xi_3$; the coefficients of $\xi_a\overline{\xi_b}$ in the variables $\xi_2,\xi_4$ form the matrix $\begin{pmatrix}-1/8&-i\\ i&-2\end{pmatrix}$, of determinant $-3/4$. Hence the Levi form has signature $(1,1)$, up to sign, in $\xi_2,\xi_4$, and its kernel is spanned by $\partial_{\xi_1},\partial_{\xi_3}$.

\begin{remark}\label{rem:comparison-literature}
The $2$-nondegenerate orbits of Proposition~\ref{prop:Levi-kernel-dimensions} are $\mathcal O_9\subset F_{\alpha_1}$ and $\mathcal O_7\subset F_{\alpha_2}$. By \eqref{eq:G2-alpha1-O9-local-model}, $\mathcal O_9\subset F_{\alpha_1}$ is an open subset of a tube over a translate of the quadric cone $a_1^2+a_2^2-a_3^2=ut$ in $\mathbb R^5$, where $a_j=\operatorname{Im}\xi_j$, $u=\operatorname{Im}w$ and $t=1+\operatorname{Im}\zeta$. The invariants of these two orbits listed above are also those of two models of \cite{Gregorovic2021} lying in the same flag varieties: the entry $\mathfrak{so}(p+2,q+2)$ with $p+q=3$ of \cite[Table~1, Section~5.3]{Gregorovic2021} for $\mathcal O_9$, and the entry $(\mathfrak g_{2(2)},\{\alpha_1\},\{\alpha_2\})$ of \cite[Table~2]{Gregorovic2021} for $\mathcal O_7$. We do not prove that these orbits are locally $CR$ equivalent to the models, since this requires a finer comparison than the $CR$ type and the Levi kernels, and we do not determine their algebras of infinitesimal $CR$ automorphisms. We only note that $\mathcal O_9\subset F_{\alpha_1}$ is open in an $SO(3,4)$-orbit, hence its symmetry algebra contains $\mathfrak{so}(3,4)\supsetneq\mathfrak g_{2(2)}$.
\end{remark}

\begin{remark}
The natural projections $\pi_j:G_2/B\to G_2/Q_{\alpha_j}$, $j=1,2$, have fibers isomorphic to $\mathbb{CP}^1$, but the restriction of $\pi_j$ to an orbit $\mathcal O(\gamma)$ has fibers of real dimension $\dim\mathcal O(\gamma)-\dim\pi_j(\mathcal O(\gamma))$, which by Section~\ref{subsec:projecting-orbits} is $0$, $1$ or $2$. Whenever such a restriction is a $CR$ fibration with complex fibers and a $k$-nondegenerate base, it realizes the $k$-nondegenerate reduction, see \cite[Proposition~6.3 and Theorem~6.4]{AltomaniMedoriNacinovich2010}, where $k$-nondegenerate was called \emph{weakly nondegenerate}. At the level of $CR$ algebras, the reduction discards the part of the structure on which the sequence \eqref{eq:Freeman-filtration-CRalg} stabilizes at a nontrivial term, i.e.  setting
$$\mathfrak q^{(\infty)}=\bigcap_{r\ge1}\mathfrak q^{(r)},
\qquad
\mathcal K^{(\infty)}_o=\bigcap_{r\ge1}\mathcal K^{(r)}_o\simeq\mathfrak q^{(\infty)}/(\mathfrak q\cap\overline{\mathfrak q}),$$
if $\mathcal K^{(\infty)}_o\neq\{0\}$, that is, if $\mathfrak q^{(\infty)}$ strictly contains $\mathfrak q\cap\overline{\mathfrak q}$, the orbit is holomorphically degenerate, and the reduction is 
$\mathfrak q'=\mathfrak q+\overline{\mathfrak q^{(\infty)}}$. 
Equivalently, $\mathfrak q'$ is the largest complex Lie subalgebra 
satisfying $\mathfrak q\subseteq\mathfrak q'\subseteq\mathfrak q+\overline{\mathfrak q}$. For the six intermediate orbits of $G_2/B$ 
the sequence stabilizes at a kernel of complex dimension one, spanned 
by a root space. The action of $\tilde\sigma_1$ on the roots computed 
in Section~\ref{sec:explicit-G2-B} attaches to 
$\gamma_2,\dots,\gamma_7$ the involutions $s_{\alpha_2}$, 
$s_{3\alpha_1+\alpha_2}$, $s_{3\alpha_1+2\alpha_2}$, 
$s_{2\alpha_1+\alpha_2}$, $s_{\alpha_1+\alpha_2}$ and $s_{\alpha_1}$ of 
Table~\ref{tab:involutions-simple-roots}, respectively. The stable 
kernels are $\mathfrak g_{\alpha_2}$, $\mathfrak g_{2\alpha_1+\alpha_2}$, $\mathfrak g_{3\alpha_1+\alpha_2}$, $\mathfrak g_{\alpha_1+\alpha_2}$, $\mathfrak g_{3\alpha_1+2\alpha_2}$, $\mathfrak g_{\alpha_1}$, and their conjugates are $\mathfrak g_{-\alpha_2}$, $\mathfrak g_{-\alpha_1}$, $\mathfrak g_{-\alpha_2}$, $\mathfrak g_{-\alpha_1}$, $\mathfrak g_{-\alpha_2}$, $\mathfrak g_{-\alpha_1}$. Hence $\mathfrak q'=\mathfrak q+\mathfrak g_{-\alpha_2}=\mathfrak q_{\alpha_1}$ for $\gamma_2,\gamma_4,\gamma_6$ and $\mathfrak q'=\mathfrak q+\mathfrak g_{-\alpha_1}=\mathfrak q_{\alpha_2}$ for $\gamma_3,\gamma_5,\gamma_7$: the reduction is realized by $\pi_1$ in 
the first case and by $\pi_2$ in the second. By 
Section~\ref{subsec:projecting-orbits}, $\pi_1$ maps $\mathcal O(\gamma_2)$, $\mathcal O(\gamma_6)$, $\mathcal O(\gamma_4)$ onto the orbits $\mathcal O_5$, $\mathcal O_7$, $\mathcal O_9\subset F_{\alpha_1}$, which are totally real, $1$-nondegenerate and $2$-nondegenerate, respectively. Likewise $\pi_2$ maps $\mathcal O(\gamma_7)$, $\mathcal O(\gamma_3)$, $\mathcal O(\gamma_5)$ onto $\mathcal O_5$, $\mathcal O_7$, $\mathcal O_9\subset F_{\alpha_2}$, 
which are totally real, $2$-nondegenerate and $3$-nondegenerate, see 
Proposition~\ref{prop:Levi-kernel-dimensions}. In these six cases the 
fibers have real dimension $2$ and the $CR$ dimension drops by one. For 
the other projection the fibers have dimension $0$ or $1$, and 
$\mathfrak q'\neq\mathfrak q_{\alpha_j}$, so it does not realize the reduction.
\end{remark}
\section*{Acknowledgements}
The authors were partially supported by the group G.N.S.A.G.A. of I.N.d.A.M.\ . Moreover, the second author was supported by University of Parma through the action ``Bando di Ateneo 2025 per la ricerca''.

\end{document}